\documentclass[preprint,aos]{imsart}

\RequirePackage{amsthm,amsmath,amsfonts,amssymb}
\RequirePackage[numbers]{natbib}

\RequirePackage{algorithm2e}
\RequirePackage{algorithmic}

\usepackage{breakcites}
\usepackage{graphicx} % Required for inserting images
\usepackage{hyperref}
\usepackage{caption}

\usepackage{color}
\usepackage{hyperref}
\usepackage{ulem}
\usepackage{textcomp}
\usepackage{comment}
\usepackage{xcolor}
\usepackage{amsthm}
\usepackage{bbm}
\usepackage[makeroom]{cancel}

\RequirePackage[shortlabels]{enumitem}

\newtheorem{thm}{Theorem}[section]
\newtheorem{definition}{Definition}
\newtheorem{prop}[thm]{Proposition}
\newtheorem{lem}[thm]{Lemma}
\newtheorem{cor}[thm]{Corollary}

 \newtheorem{rem}{Remark}
 \newtheorem{hyp}{Assumption}
\newcommand{\argmin}{\mathop{\mathrm{argmin}}}

\newcommand{\E}{\operatorname{\mathbb{E}}}
\renewcommand{\P}{\operatorname{\mathbb{P}}}
\newcommand{\tr}{\mathrm{Tr}}
\newcommand{\pa}[1]{\left(#1\right)}
\newcommand{\ac}[1]{\left\{#1\right\}}
\newcommand{\cro}[1]{\left[#1\right]}
\newcommand{\<}{\langle}
\renewcommand{\>}{\rangle}
\newcommand{\eps}{\varepsilon}

\newcommand{\X}{\mathbf{X}}
\newcommand{\R}{\mathbb{R}}
\newcommand{\1}{\mathbf{1}}
\newcommand{\N}{\mathbb{N}}
\newcommand{\cumul}{\mathrm{Cum}}

\usepackage[textwidth=2.0cm, textsize=tiny]{todonotes} % for writting

\begin{document}

%%%%%%%%%%%%%%%%%%%%%%%%%%%%%%%%%%%%%%%%%%%%%%%%%%%%%%%%%%%%%%%%%%%%%%%%%%%%%%%%%%
%%%%%%%%%%%%%%%%%%%%
%%%%%%%%%%%%%%%%%%%%%%%%%%%%%%%%%%%%%%%%%%%%%%%%%%%%%%%%%%%%%%%%%%%%%%%%%%%%%%%%%%%%%%%%
%%%%%%%%%%%%%%%%%%% ATTENTION AVANT SOUMISSION A AoS
%%%%%%%%%%%%%%%%%%%%%%%%%%%%%%%%%%%%%%%%%%%%%%%%%%%%%%%%%%%%%%%%%%%%%%%%%%%%%%%%%%%%%%%%%
%%%%%%%%%%%%%%%%%%%%% Changer les marges ligne 276 de imsart.sty
%%%%%%%%%%%%%%%%%%%%%%%%%%%%%%%%%%%%%%%%%%%%%%%%%%%%%%%%%%%%%%%%%%%%%%%%%%%%%%%%%%%%%%

\begin{frontmatter}
\title{Computational lower bounds in latent models: clustering, sparse-clustering, biclustering}
\runtitle{Computational lower bounds}

\begin{aug}
%%%%%%%%%%%%%%%%%%%%%%%%%%%%%%%%%%%%%%%%%%%%%%%
%% Only one address is permitted per author. %%
%% Only division, organization and e-mail is %%
%% included in the address.                  %%
%% Additional information can be included in %%
%% the Acknowledgments section if necessary. %%
%% ORCID can be inserted by command:         %%
%% \orcid{0000-0000-0000-0000}               %%
%%%%%%%%%%%%%%%%%%%%%%%%%%%%%%%%%%%%%%%%%%%%%%%
\author[A]{\fnms{Bertrand}~\snm{Even}\ead[label=e1]{bertrand.even@universite-paris-saclay.fr}},
\author[A]{\fnms{Christophe}~\snm{Giraud}\ead[label=e2]{christophe.giraud@universite-paris-saclay.fr}}
\and
\author[B]{\fnms{Nicolas}~\snm{Verzelen}\ead[label=e3]{Nicolas.Verzelen@inrae.fr}}
%%%%%%%%%%%%%%%%%%%%%%%%%%%%%%%%%%%%%%%%%%%%%%
%% Addresses                                %%
%%%%%%%%%%%%%%%%%%%%%%%%%%%%%%%%%%%%%%%%%%%%%%

\address[A]{Laboratoire de Math\'ematiques d'Orsay, Universit\'e Paris-Saclay\printead[presep={,\ }]{e1,e2}}

\address[B]{MISTEA, INRAE, Institut Agro,
Univ. Montpellier\printead[presep={,\ }]{e3}}

\end{aug}

\begin{abstract}
 In many high-dimensional problems, like sparse-PCA, planted clique, or clustering, the best known algorithms with polynomial time complexity fail to reach the statistical performance provably achievable by algorithms free of computational constraints.
This observation has given rise to the conjecture of the existence, for some problems, of gaps -- so called statistical-computational gaps -- between the best possible statistical performance achievable without computational constraints, and the best performance achievable with poly-time algorithms. A powerful approach to assess the best performance achievable in poly-time is to investigate the best performance achievable by polynomials with low-degree. 
We build on the seminal paper of  \cite{SchrammWein22}  and propose a new scheme to derive  lower bounds on the performance of low-degree polynomials in some latent space models. By better leveraging the latent structures, we obtain new and sharper results, with simplified proofs.  We then instantiate our scheme to provide computational lower bounds for the problems of clustering, sparse clustering, and biclustering. We also prove matching upper-bounds and some additional statistical results, 
in order to provide a comprehensive description of the statistical-computational gaps occurring in these three problems. 
\end{abstract}



\end{frontmatter}

\section{Introduction}

In high-dimensional statistics, the primary goal is to derive computationally efficient estimation procedures, achieving the best possible statistical performance. Yet, in many problems, such as sparse PCA, planted clique or clustering, the best known algorithms with polynomial-time complexity are unable to match the performances provably achievable by the best estimators (without computational constraints). This observation has lead to several conjectures on the existence of gaps (called statistical-computational gaps) between the optimal statistical performance, i.e. the best performance achievable without computational constraints, and the best performance achievable by polynomial time algorithms. In particular, to assess the quality of a computationally efficient algorithm for a given task, the theoretical performance should not be compared to the optimal statistical performance (without computational constraints), but to the performance of the best poly-time algorithm. This raises the problem of establishing lower-bounds on the performance of the best poly-time algorithms for a wide range of problems.

Since high-dimensional statistics deal with random instances, the classical notions of worst-case hardness, such as P, NP, etc are not suitable for the high-dimensional framework. Instead, lower bounds are obtained for some specific models of computations, such as SoS \cite{HopkinsFOCS17,Barak19}, overlap gap property \cite{gamarnik2021overlap}, statistical query \cite{kearns1998efficient,brennan2020statistical}, and low-degree polynomials \cite{hopkins2018statistical,KuniskyWeinBandeira,SchrammWein22}, possibly combined with reductions between different statistical problems~\cite{brennan2020reducibility,pmlr-v30-Berthet13,brennan2018reducibility}.

Low-degree polynomial lower bounds (LD bounds) have recently attracted a lot of attention due to their ability to provide state-of-the-art lower bounds for a wide class of detection problems, including community detection~\cite{Hopkins17}, spikes tensor models~\cite{Hopkins17,KuniskyWeinBandeira}, sparse PCA~\cite{ding2024subexponential} among others. 
We refer to~\cite{SurveyWein2025} for a recent survey.
The low-degree polynomial framework is a computational model, where we only consider estimators, or test statistics, within the class of multivariate polynomials of degree at most $D$ of the observations. The premise of the LD literature is that for a large class of problems, the polynomials of degree $D=O(\log n)$ are as powerful as any polynomial-time algorithm. Hence, proving failure of  degree $O(\log n)$ polynomials for a given task is an indication~\cite{KuniskyWeinBandeira} that no poly-time algorithm can solve this task.
Interestingly, it has been demonstrated that this framework is closely related to other computational frameworks including statistical queries~\cite{brennan2020statistical}, free-energy landscapes from statistical physics~\cite{bandeira2022franz} or approximate message passing~\cite{montanari2024equivalence}. 
The LD framework has been initially developed for  hypothesis testing (detection problems), where the goal is to detect the existence of a possible planted signal in the data. In addition to predicting computational barriers for algorithms computable in polynomial time, LD polynomials may be used to predict in the Hard regime the amount of time needed to resolve a problem. In sparse PCA, \cite{ding2024subexponential} exhibits a phenomenon where, when the signal-to-noise ratio decreases, the complexity interpolates between being of polynomial time in the easy regime and being exponential at the informational threshold. In general, \cite{hopkins2018statistical} states in its low-degree conjecture that polynomials of degree $D$ are a proxy for algorithm of time complexity roughly $n^{D}$. 

The framework has then been extended to the estimation problem in the seminal paper of Schramm and Wein~\cite{SchrammWein22}. In the estimation framework, the goal is to lower-bound the risk of the best polynomial of degree at most $D$. A key contribution of  \cite{SchrammWein22} 
 is to relate the derivation of LD bounds to the upper-bounding of some  multivariate cumulants. The theory developed in \cite{SchrammWein22} provides a versatile framework to derive LD bounds in estimation and has been applied among others to submatrix estimation~\cite{SchrammWein22}, stochastic block models and graphons~\cite{luo2023computational}, dense cycles recovery~\cite{mao2023detection}, coloring problems~\cite{kothari2023planted}.
 However, it suffers from two limitations:
\begin{enumerate}
\item It can lead to quite complex analyses for some involved problems, and this complexity can limit the range of the results that can be obtained, as e.g. in \cite{Even24} for Gaussian mixture models  or in~\cite{luo2023computational} for biclustering. Those examples are discussed precisely later on.
\item It provides non-sharp thresholds, with spurious poly-log factors.
\end{enumerate}
The second limitation has been recently tackled by Sohn and Wein~\cite{SohnWein25}, which provides more powerful technics to derive sharp thresholds, but at the price of an even higher technicality and complexity, limiting the applicability to more involved problems.

\subsection{Our contributions}
Our main technical contribution is to propose some new derivation schemes for deriving the cumulants in  some latent variable models. The heart of the improvement is to better handle conditional independences in latent variable models by conditioning,  leading to stronger and new results, with simpler proofs. This result is then instantiated in the three following emblematic problems: clustering high-dimensional Gaussian mixtures, sparse clustering and biclustering. Whereas the computational-statistical gaps were previously known in some restrictive, we provide an almost full picture in this work. 
To complement it, we also provide upper-bounds on the statistical and computational rates for these problems.
Let us describe our main contributions into more details.\smallskip

{\bf Bounding multivariate cumulants in a model with latent variables.} 
We consider the following model of data generation. We observe a $n\times p$ matrix $Y\in\R^{n\times p}
$, which can be decomposed as $Y=X+E$, where $E$ is a noise matrix with i.i.d. Gaussian entries, and $X$ is a signal matrix, independent of $E$, structured by a latent variable $Z\in \mathcal{Z}$
\begin{equation}\label{eq:latent-model}
X_{ij}=\delta_{ij}(Z)\nu_{\theta_{ij}(Z)},\quad \textrm{for}\quad (i,j)\in [n]\times [p], 
\end{equation}
with
\begin{itemize}
\item  $\nu_{k\ell}\in \R^{n\times p}$, for $(k,\ell)\in [K]\times [L]$, possibly randomly generated;
\item $\delta_{ij}:  \mathcal{Z} \to \ac{0,1,-1}$ and $\theta_{ij}:  \mathcal{Z} \to  [K]\times [L]$, for any $(i,j)\in [n]\times [p]$.
\end{itemize}

In the analysis of \cite{SchrammWein22}, the key step for proving LD bounds is to upper-bound multivariate cumulants of the form $\kappa_{h(Z),\alpha}=\cumul(h(Z),\ac{X_{ij}: (i,j)\in \alpha})$, where $h:\mathcal{Z}\to \R$ is a measurable function, and  $\ac{X_{ij}: (i,j)\in \alpha}$ is a multiset, where the variable $X_{ij}$ is repeated $\alpha_{ij}$-times.  
Our first contribution is to provide some simple bounds, and simple recursions for bounding such cumulants $\kappa_{h(Z),\alpha}$. These bounds and recursions are obtained by merely applying a conditioning on the latent variable $Z$, and observing that many simplifications occur. While technically very simple, this first  step is the basis  for deriving new lower-bounds in different instances of the latent model (\ref{eq:latent-model}). 
\smallskip

{\bf Clustering Gaussian mixtures.}  
The classical Gaussian mixture model is an instantiation of the latent model (\ref{eq:latent-model}).
For some unknown vectors $\mu_1,\ldots,\mu_K\in\R^p$, some unknown $\sigma>0$, and an unknown partition $G^*=\ac{G^*_1,\ldots,G^*_K}$ of $[n]$, the observations $Y_{ij}$ are sampled independently with distribution
\begin{equation}\label{eq:GMM-intro}
Y_{ij} \sim \mathcal{N}(\mu_{kj},\sigma^2),\quad \text{for}\ i\in G^*_k\quad \text{and}\ j\in[p].
\end{equation}
For simplicity, we focus on the case where the partition is balanced, i.e. where all clusters $G^*_{k}$ have similar cardinality. Denoting by 
\begin{equation}\label{eq:snrclustering}
\Delta^2=\min_{k\neq l}\frac{\|\mu_k-\mu_l\|^2}{2\sigma^2}\enspace,
\end{equation}
the minimal (scaled) separation between clusters, 
we prove in Theorem \ref{thm:lowdegreeclusteringsharp} that, for $p\geq \log^5(n)$, clustering better than a random guessing with $\log(n)$-degree polynomials can be impossible when
\begin{equation}\label{eq:clustering-impossible:intro}
{{\Delta}^2} <  (c\log K)\vee {\pa{\frac{\sqrt{p}}{(\log n)^{9}}\wedge\sqrt{\frac{K^2p}{n}}}}\enspace,
\end{equation}
where $c$ is a positive numerical constant.
This result extends the LD bound of \cite{Even24}, only valid for the high-dimensional regime $p\geq n$,
to the much more challenging moderately high-dimensional regime $\log^5(n)\leq p \leq n$.
 The LD bound is also improved, removing spurious poly-$\log(n)$ factors present in the lower-bound of \cite{Even24}.
 In particular,  the bound recovers the exact BBP threshold $\sqrt{{K^2p}/{n}}$, that was conjectured in \cite{lesieur2016phase} with tools from statistical physics.
 Comparing (\ref{eq:clustering-impossible:intro}) to the statistical threshold 
 \begin{equation}\label{eq:informationalclustering-intro}
\Delta^2\gtrsim \log(K)+\sqrt{\frac{pK}{n}\log(K)}\enspace,
\end{equation}
derived in \cite{Even24}, above which partial clustering is achievable by minimizing \emph{exactly} the Kmeans criterion, we observe the existence of a statistical-computational gap when
$$p >  {n \log(K)\over K^2} \quad \text{and}\quad K \lesssim n^{1-o(1)}.$$
 We also provide some new upper-bounds proving that clustering in polynomial time is possible when $\Delta^2$ is larger (up to log factors) than (\ref{eq:clustering-impossible:intro}), in a wide range of regimes of $K,n,p$.
\smallskip 
 
 {\bf Sparse Clustering.} Sparse clustering is an instance of the  clustering model above, where the means $\mu_{k}$ are sparse. 
  Let $s\in [p]$ and an unknown subset $J^*\subseteq [p]$ with $|J^*|\leq s$.  For some unknown vectors $\mu_1,\ldots, \mu_K$ which are all supported on $J^*$, some unknown $\sigma>0$, and some unknown partition $G^*=\ac{G_1^*,\ldots, G^*_K}$ of $[n]$, 
 the observations $Y_{ij}$ are sampled independently with distribution 
\[Y_{ij} \sim \mathcal{N}(\mu_{kj},\sigma^2),\quad \text{for}\ i\in G^*_k\quad \text{and}\ j\in[p].\]

In Section \ref{sec:sparse},
we prove that clustering better than a random guessing with $\log(n)$-degree polynomials can be impossible when 
\begin{equation}\label{eq:seuil-comput-sparse}
\Delta^2\lesssim_{\log} 1+\min\pa{\sqrt{s}, \sqrt{\frac{s K^2}{n}}}+\sqrt{\frac{s^2}{n}}\quad \text{and}\quad \Delta^2\lesssim_{\log} 1+\min\pa{\sqrt{p}, \sqrt{\frac{p K^2}{n}}}\enspace.
\end{equation}
This result generalizes the  computational lower bound proved in L\"offler et al.~\cite{lffler2021computationallyefficientsparseclustering} for the specific case of $K=2$ groups. Our lower bound~\eqref{eq:seuil-comput-sparse} is valid for any $K$ and $s$ and, thereby, shed lights on the joint dependence of computationally efficient rates on $K$ and $s$. 
The second condition in \eqref{eq:seuil-comput-sparse} corresponds to the Condition \eqref{eq:clustering-impossible:intro} for clustering in poly-time in dimension $p$. 
The third term in the right-hand side of the first condition can be understood as the signal needed to ensure recovery of the $s$ columns supporting the $\mu_{k}$, while the two first terms corresponds to the rate for poly-time clustering in dimension $s$. 
Under the simplifying assumption that the signal is well spread over the $s$ columns, we prove that clustering  above the level (\ref{eq:seuil-comput-sparse}) can indeed be performed in poly-time by (i) selecting the $s$ columns with the largest $\ell^2$-norm, (ii) removing all the other columns, (iii) applying an optimal poly-time clustering algorithm on the remaining matrice. 
To delineate the statistical-computational gap, we prove in Proposition \ref{prop:upperboundsparseIT}, that, by applying an exhaustive search over all the partitions and all the columns, perfect clustering can be achieved in this context as soon as
\begin{equation}\label{eq:seuil-IT-sparse}
\Delta^2\gtrsim_{\log} 1+ \sqrt{\frac{sK}{n}}+{s\sqrt{K}\over n}\ ,\quad \text{or} \quad \Delta^2\gtrsim_{\log} 1+ \sqrt{\frac{pK}{n}} \enspace.
\end{equation}
The first minimal separation in \eqref{eq:seuil-IT-sparse} gathers  the statistical threshold
\eqref{eq:informationalclustering} in dimension $p=s$, with a separation 
$\Delta^2\geq {s\sqrt{K}}/{n}$ corresponding to the separation required for recovering the active columns set $J^*$ \emph{once} the clustering is known. We highlight the following interesting statistical-computational phenomenon in sparse-clustering, with a well spread signal. 
The additional separation $\Delta^2\gtrsim \sqrt{s^2/n}$ required in poly-time corresponds to the separation needed for recovering the active columns \emph{before} clustering, while the statistical additional separation $\Delta^2\gtrsim {s\sqrt{K}}/{n}$ corresponds to the separation needed for recovering the active columns \emph{after} clustering. This feature unveils  a better ability of non poly-time algorithm  to fully exploit the joint sparse-and-clustered structure.
Comparing (\ref{eq:seuil-comput-sparse}) and (\ref{eq:seuil-IT-sparse}), 
we observe the existence of a statistical-computational gap, which, 
depending on the regimes, can be as large as factor $\sqrt{n/K}$ or a factor $\sqrt{K}$. 
\smallskip

{\bf Biclustering.}  As a last example, we investigate the biclustering problem where both rows and columns  can be clustered. 
For some unknown matrix $\mu\in\R^{K\times L}$, some unknown $\sigma>0$, and unknown partitions $G^*=\ac{G^*_1,\ldots, G^*_K}$ of $[n]$ and $H^*=\ac{H_1^*,\ldots, H_L^*}$ of $[p]$, the $Y_{ij}$'s are sampled independently  with distribution 
\[Y_{ij} \sim \mathcal{N}(\mu_{kl},\sigma^2),\quad \text{for}\ (i,j)\in G^*_k\times H^*_l.\]
We observe that when all the clusters in $G^*$ (respectively $H^*$) have the same size $n/K$ (resp. $p/L$), 
we have for $i\in G^*_{k}$ and $i'\in G^*_{k'}$ 
\[{\|X_{i:}-X_{i':}\|^2} = \sum_{l=1}^L |H^*_{l}| (\mu_{kl}-\mu_{k'l})^2 \asymp {p\over L}\ \|\mu_{k:}-\mu_{k':}\|^2\enspace .
\]
Hence, we introduce
$$\Delta^2_r={p\over L}\min_{k\neq k'\in [K]}\frac{\|\mu_{k:}-\mu_{k':}\|^2}{2\sigma^2}\quad \text{and}\quad \Delta^2_c={n\over K}\min_{l\neq l'\in [L]}\frac{\|\mu_{:l}-\mu_{:l'}\|^2}{2\sigma^2}\enspace,$$
which represent the minimum row and column separations.
Given the symmetry of the problem, we can focus on the problem of finding the minimum separations for row clustering, i.e. for recovering partially $G^*$. We investigate if and how the column structure can help for recovering the row clusters.
Our results in Section~\ref{sec:biclustering} show the following dichotomy.
\begin{enumerate}
\item Either $\Delta_c^2 \leq_{\log} 1+\min\pa{\sqrt{n}, \sqrt{{nK^2}/{p}}}$, in which case row-clustering can be impossible in poly-time below the threshold $\Delta_r^2 \leq_{\log} 1+\min\pa{\sqrt{p}, \sqrt{{pK^2}/{n}}}$ corresponding to simple clustering;
\item Or $\Delta_c^2 \geq_{\log} 1+\min\pa{\sqrt{n}, \sqrt{{nK^2}/{p}}}$, in which case row-clustering is possible only above the threshold $\Delta_r^2\overset{\log}{=} 1+\min\pa{\sqrt{L}, \sqrt{{LK^2}/{n}}}$ corresponding to clustering in dimension $L$. 
\end{enumerate}
This result exhibits the following interesting phenomenon.
We observe that the threshold  $\Delta_{c}\geq_{\log} 1+\min\pa{\sqrt{n}, \sqrt{{nK^2}/{p}}}$ corresponds to the minimal level for clustering the $n$-dimensional columns in poly-time. When it is possible to cluster these columns in poly-time (case 2), then an optimal poly-time algorithm amounts to (i) cluster the columns, (ii) average the columns within a same group, reducing the number of columns to $L$, and (iii) apply a poly-time row clustering on the new $n\times L$ matrix. Conversely, when it is not possible to cluster the columns in poly-time (case 1), then the column structure is useless, and the minimal level for clustering the rows in poly-time corresponds to the level for simple clustering. Hence, for poly-time algorithms, the column structure is helpful for row clustering, only when the columns can be clustered in poly-time.

This is in contrast with computationally unconstrained algorithms, which can better leverage the column structure, and only require
$$\bigg\{\Delta^2_{r}\geq_{\log} 1+ \sqrt{KL \over n}\quad \text{and}\quad \Delta^2_{c}\geq_{\log} 1+\sqrt{KL \over p}\,\bigg\},\quad\ \underline{\text{or}}\quad\quad 
\bigg\{\Delta^2_{r}\geq_{\log} 1+ \sqrt{Kp \over n}\,\bigg\},
$$
for recovering $G^*$.
We observe that (i) the column separation $\Delta^2_{c}\geq_{\log} 1+\sqrt{KL / p}$ corresponds to the minimal separation required to recover $H^*$ when $G^*$ is known, and (ii) when this condition is met we can recover $G^*$ with the separation $\Delta^2_{r}\geq_{\log} 1+\sqrt{KL / n}$ which corresponds to the minimal separation required to recover $G^*$ when $H^*$ is known.  Hence, only a $K$-dimensional column separation condition  is needed to benefit from the $L$-dimensional row separation condition $\Delta^2_{r}\geq_{\log} 1+ \sqrt{KL / n}$ for successful clustering.
This feature is in contrast with poly-time algorithms, where the $n$-dimensional column  separation $\Delta^2_{c}\geq_{\log} 1+ \min\pa{\sqrt{n},\sqrt{L^2n / p}}$ is required for benefiting from the $L$-dimensional row separation condition $\Delta^2_{r}\geq_{\log} 1+ \min\pa{ \sqrt{L}, \sqrt{K^2 L/ n}}$.
Our results then unveil a much better ability of non poly-time algorithms to leverage the biclustering structure, compared to poly-time algorithms.

\subsection{Related Literature on clustering problems}

\paragraph*{Gaussian mixture clustering} Gaussian mixtures are arguably the most iconic distribution model for clustering. The corresponding problem has lead to many developments both in statistics and machine learning~\cite{Dasgupta99,VEMPALA2004,lesieur2016phase,LuZhou2016,DBLP:journals/corr/abs-1711-07211,Regev2017,giraud2019partial,fei2018hidden,chen2021hanson,Kwon20,SegolNadler2021,romanov2022,LiuLi2022,diakonikolasCOLT23b,Even24}. In the isotropic Gaussian mixture model, the minimax condition for partial recovery in any dimension was characterized in~\cite{Even24}, although it was already known  in the low-dimensional case, see e.g.~\cite{Regev2017,SegolNadler2021}.% However, the optimal separation condition under polynomial-time constraint, remains unknown despite a lot of works both on the construction of poly-time procedures and on the construction of poly-time lower bounds. 

In an asymptotic regime where $K$ is fixed, $n,p\to\infty$ with $p/n\to \alpha\geq \frac{1}{K^2}$, it was conjectured by \cite{lesieur2016phase} that the problem is indeed hard under the BBP transition $\Delta^2\asymp \sqrt{pK^2/n} $. To do so, they study the fixed points of the sate evolution equation of Approximate Message Passing. In the same asymptotic regime, \cite{banks2018information} proves that spectral detection is possible if and only if the separation is above the BBP transition $\sqrt{pK^2/n}$.

In the high-dimensional regime where $p\geq n$,~\cite{Even24} partially confirmed this conjecture by establishing
a LD lower bound that agrees (up to polylog) with the prediction of~\cite{lesieur2016phase} in the regime where $n\leq K^2$, and by unveiling another rate in the many group regime ($n\geq K^2$). These LD lower bounds  are matched by a combination of a SDP~\cite{giraud2019partial} and hierarchical-clustering techniques.
In contrast, in the low-dimensional regime $n\geq poly(p,K)$, there is no significant statistical-computational gap. Indeed, using iterative projections of high-order tensors, Liu and Li~\cite{LiuLi2022} have proved that it is possible to partially recover the clusters when $\Delta^2\gtrsim \log(K)^{1+\eps}$, with $\eps$ an arbitrary small positive constant, thereby almost matching the informational bound.
The moderately high-dimensional regime $p<n < poly(p,K)$, for some (non-explicit) polynomial $poly(p,K)$ from~\cite{LiuLi2022}, is still to be understood. 
Although there are numerous works on spectral procedures as well as Lloyd's algorithm~\cite{ndaoud2022sharp,LuZhou2016}, SDP~\cite{fei2018hidden,giraud2019partial}, or hierarchical-clustering procedures~\cite{VEMPALA2004} in this moderately high-dimensional regime, it remains largely unknown whether those are optimal among polynomial-time algorithms.

We underline that we focus in this work on the isotropic case. In the non-isotropic case, there is an additional statistical-computational gap which does not come from the high-dimensionality but from the unknown covariance structure. In particular, \cite{SQclustering} and \cite{pmlr-v195-diakonikolas23b} establish some  lower-bounds on the running time of any Statistical-Query algorithms, proving a statistical-computational gap between optimal procedures and SQ algorithms.

\paragraph*{Sparse clustering}
Motivated by practical considerations, Raftery and Dean~\cite{raftery2006variable} have introduced the sparse clustering model, where the clusters only differ on a small number of features. This lead to the development of numerous procedures that aim to building upon this sparsity to improve the clustering --see e.g.~\cite{maugis2011non,marbac2017variable,witten2010framework,mun2025high} and references therein. Notably, \cite{witten2010framework} uses a penalization by the $l_1$-norm in order to use weighted versions of the Kmeans objective. Another class of procedures amounts to alternate between feature selection and clustering (e.g.~\cite{mun2025high}). 
In the specific case where $K=2$, \cite{azizyan2013minimax} have characterized the minimax optimal rate for clustering. They also provided a two-step computationally-efficient procedure, but with significantly worse clustering rate. Under some technical assumptions, \cite{jin2017phase} introduced a more general two-step procedure that (i) selects active columns, (ii) uses a vanilla clustering procedure for $K\geq 2$, and they conjectured the existence of a statistical-computational gap. The corresponding sparse clustering detection problem was studied in~\cite{verzelen2017detection} from a minimax perspective. In the regime where the sparsity $s$ is small, they drawn some informal connection with the sparse PCA problem, for which a statistical-computational gap has been exhibited~\cite{pmlr-v30-Berthet13}. Let us remark that some procedures such as CHIME~\cite{cai2019chime} do not seem to exhibit this statistical-computational gap, but they rely on a good initialization which is only known to be achieved by non-efficient procedures. 
This connection to sparse PCA as well as the interest in sparse clustering has spurred the need for computational lower bounds~\cite{fan2018curse,brennan2020reducibility,lffler2021computationallyefficientsparseclustering}. All of these works are restricted to the case $K=2$, and focus on the sparsity effect. Brennan and Bresler~\cite{brennan2020reducibility} have reduced sparse-clustering to a variant of planted clique, whereas Fan et al.~\cite{fan2018curse} have established a matching statistical query (SQ) lower-bound. Closer to our perspective, \cite{lffler2021computationallyefficientsparseclustering} have provided a LD lower bound for the corresponding detection problem. All these lower bounds suggest that it is impossible to recover the $K=2$ groups in polynomial-time when  $\Delta^2{\leq}_{\log}\,s/\sqrt{n}$ in a high-dimensional regime where $s\leq \sqrt{pn}$. In some way, we extend this theory to the case of a general number $K\geq 2$ of groups, unraveling the impact of $K$ in the statistical-computational gap.

\paragraph*{Biclustering}
The biclustering problem arises when both the rows and the columns of a matrix $Y$ can be clustered~\cite{hartigan1972direct}. A simpler version  of this problem is to detect or estimate a single submatrix hidden in some noise. The latter is one of the earliest problem whose statistical-computational gap has been established~\cite{balakrishnan2011statistical,kolar2011minimax,ma2015computational,cai2017computational,SchrammWein22,SohnWein25}. Closer to biclustering, \cite{dadon2024detection} considers the case where there are multiple planted submatrices by providing in particular LD lower bounds for the detection problem. For the general biclustering problem with $(K,L)$ groups, the minimax estimation rate for estimating the signal matrix $X=\mathbb{E}[Y]$ in Frobenius norm has been characterized in~\cite{gao2016optimal}.
On the computational side, Luo and Gao~\cite{luo2023computational} have built on the general methodology of~\cite{SchrammWein22} to provide a LD lower bound for this problem; they also studied spectral algorithms to match this bound. However, their LD lower bound turns out to be sharp only in the almost square regime $n\asymp p$ and when $\min(K, L)\leq \sqrt{\max(n,p)}$. In particular, handling rectangular settings where $n$ and $p$ differ significantly, requires a more careful control of the cumulants, as done in this manuscript. Another important difference between our work and that of~\cite{luo2023computational} is that we focus on the problem of recovering the clustering of rows instead of reconstructing the mean matrix. In asymmetric regimes, where either $n$ is different from $p$, or $K$ different from $L$, the clustering problem turns out to behave quite differently.

Extensions of stochastic block models (SBM) to biclustering problems have been 
considered e.g. in~\cite{florescu2016spectral,ndaoud2021improved}, but their instance of the  model is quite different from ours, as it is assumed that the number of groups $K$ is equal to $L$, that each group of rows is associated to a group of columns, and that the connection probability is higher between the corresponding nodes. In this sense, the model is closer to the literature on SBM.

\subsection{Organisation and notation}

In Section~\ref{sec:technique}, we introduce the low-degree estimation framework as well as the general latent model. Then, we describe the conditioning techniques and showcase a simple application to Gaussian mixture models.  The reader less interested in the techniques for proving LD lower bounds may skip this section.
Then, we use these techniques to establish tight LD lower bounds for our three main problems: Gaussian mixture clustering (Section~\ref{sec:clustering}), sparse clustering (Section~\ref{sec:sparse}), and biclustering (Section~\ref{sec:biclustering}). A long the way, we provide polynomial-time upper bounds and informational upper bounds when unknown in order to precisely quantify the computational-statistical gaps.
Section~\ref{sec:discussion} provides a discussion of possible extensions and open problems. More technical discussions as well as the proofs  are postponed to the appendix.

\paragraph*{Notation} Given a vector $v$, we write $\|v\|$ for its Euclidean norm. For a matrix $A$, we denote $\|A\|_F$  for its Frobenius norm and and $\|A\|_{op}$ for its operator norm.
For two function $u$ and $v$, we write $u\lesssim v$ if there a exists a numerical constant such that $u\leq c v$. We write $u\lesssim v$ if $u\lesssim v$ and $v\lesssim u$. 
For two functions that may depend on $n$ and $p$, we write $u\leq_{\log} v$, if there exist  numerical constants $c$ and $c'$ such that $u\leq c \log^{c'}(np) v$. If $u,v$ also depend on some other parameter $\gamma$, we write $u\leq_{\log,\gamma} v$, if there exist a constant $c_\gamma$ depending only on $\gamma$ and a numerical constant $c'$ such that $u\leq c_\gamma \log^{c'}(np) v$. For a subset $S$, we write $|S|$ for its cardinality. Given a random variable $x$ and an event $\mathcal{B}$, we write $\mathbf{1}\{\mathcal{B}\}$ for the indicator function of $\mathcal{B}$ and $\mathbb{E}[x; \mathcal{B}]$ for $\mathbb{E}[x\mathbf{1}_{\mathcal{B}}]$.

We identify a matrix $\alpha\in \N^{n\times p}$ with the multiset of $[n]\times[p]$ containing $\alpha_{ij}$ copies of $(i,j)$. For $i\in [n]$, we write $\alpha_{i:}$ the $i$-th row of $\alpha$. Similarly, for $j\in [p]$, we write $\alpha_{:j}$ the $j$-th column of $\alpha$. We denote $supp(\alpha)=\ac{i\in [n], \alpha_{i:}\neq 0}$ and $col(\alpha)=\ac{j\in [p], \alpha_{:j}\neq 0}$. Then, we denote $\#\alpha=|supp(\alpha)|$ and $r_\alpha=|col(\alpha)|$. Finally, we shall write $|\alpha|$ the $l_1$-norm of $\alpha$, which is the cardinality of $\alpha$ viewed as a multiset. Finally, $\alpha!$ stands for $\prod_{ij}\alpha_{ij}!$ and, for any real valued matrix $Q$, $Q^\alpha=\prod_{ij}Q_{ij}^{\alpha_{ij}}$.

For $W_1, \ldots, W_l$ random variables on the same space, we write $\cumul\pa{W_1,\ldots,W_l}$ their joint cumulant. For $Z$ another random variable on the same space, we write $\cumul\pa{W_1,\ldots,W_l|Z}$ the joint cumulant of the random variables taken conditionally on $Z$.

\section{Proof technique for LD bounds in the latent model}\label{sec:technique}

\subsection{Low-degree framework}
Let us consider the latent model introduced earlier, where we observe a matrix $Y\in\R^{n\times p}$, which can be decomposed as the sum $Y=X+E$ of a noise matrix $E$ with i.i.d.\ Gaussian errors, and a signal matrix $X$ structured according to a latent variable $Z\in\mathcal{Z}$ as in (\ref{eq:latent-model})
$$X_{ij}=\delta_{ij}(Z)\nu_{\theta_{ij}(Z)},\quad \textrm{for}\quad (i,j)\in [n]\times [p],$$
with  $\delta_{ij}:  \mathcal{Z} \to \ac{0,1,-1}$ and $\theta_{ij}:  \mathcal{Z} \to  [K]\times [L]$, for any $(i,j)\in [n]\times [p]$.
For example, in the case of clustering, $Z$ is the vector of independent labels $Z=[k^*_1,\ldots, k^*_n]\in [K]^n$, $\delta_{ij}(Z)=1$ and $\theta_{ij}(Z)=(k^*_{i},j)$. 
For proving LD bounds, we make the following additional assumptions.
\begin{hyp}\label{hyp:gaussien}[Gaussian means]
 The $\nu_{kl}$'s are independent of $Z$ and i.i.d. with $\mathcal{N}\pa{0,\lambda^2}$ distribution for some $\lambda>0$.
\end{hyp}
This assumption is very convenient for our analysis, as it leads to many simplifications. We mention yet, that a similar analysis can be done for other data distributions, like the Bernoulli distribution; see Appendix~\ref{sec:bernoulli} for a discussion. 

We consider the problem where we want  to estimate some scalar function of $Z$, that we write $x(Z)$, or simply $x$, with polynomials of the $Y_{ij}$ of degree at most $D$. For example, in the case of clustering, where $Z=[k^*_1,\ldots, k^*_n]\in [K]^n$, we may want to estimate $x(Z)=\1_{k^*_1=k^*_2}$. Our goal is to lower bound the best mean-square error  achieved by a polynomial of degree at most $D$
\begin{equation}\label{eq:MMSEgeneral}
MMSE_{\leq D}:=\inf_{f\in \R_{D}[Y]}\E\cro{\pa{f(Y)-x(Z)}^2}\enspace.
\end{equation}
As noticed by \cite{SchrammWein22}, the $MMSE_{\leq D}$ can be decomposed as 
\begin{equation}\label{eq:MMSE-corr}
MMSE_{\leq D}=\E\cro{x(Z)^2}-corr^2_{\leq D}\enspace,
\end{equation}
where $corr_{\leq D}$ is the $L^2$-norm of the $L^2$-projection of $x(Z)$ on the linear span of polynomials $f(Y)$ with degree at most $D$
\begin{equation}\label{def:corr}
    corr_{\leq D}:=\underset{\E[f^{2}(Y)]=1}{\sup_{f\in\R_{D}[Y]}}\E(f(Y)x(Z))=\underset{\E\pa{f^{2}(Y)}\neq 0}{\sup_{f\in\R_{D}[Y]}}\frac{\E[f(Y)x(Z)]}{\sqrt{\E(f^{2}(Y))}}\enspace.
\end{equation}
Hence, in order to lower-bound $MMSE_{\leq D}$, it is sufficient to prove an upper-bound on $corr_{\leq D}$. Our latent model is a particular instance of the Additive Gaussian Noise Model considered in \cite{SchrammWein22}. Therefore, we can apply Theorem 2.2 from \cite{SchrammWein22} that upper-bounds the low-degree correlation $corr_{\leq D}$ by a sum of squared cumulants -- see Appendix \ref{sec:cumulants} for definitions and properties of cumulants. 
Let us recall their result.
\medskip

\begin{prop}\label{thm:schrammwein}{[Theorem 2.2. in \cite{SchrammWein22}]}
    The degree-$D$ maximum correlation satisfies the upper-bound 
    \begin{equation}\label{eq:corr_D}
    corr^{2}_{\leq D}\leq \underset{|\alpha|\leq D}{\sum_{\alpha\in\N^{n\times p}}}\frac{\kappa_{x,\alpha}^{2}}{\alpha!}\enspace,
    \end{equation}
 with $\alpha!=\prod_{ij\in[n]\times [p]}\alpha_{ij}!$, and   where,  for $\alpha\in \N^{n\times p}$,  $\kappa_{x,\alpha}$ is defined as the cumulant 
\begin{equation}\label{eq:def:kappa}
    \kappa_{x,\alpha}:=\cumul\pa{x, X_{\alpha}}=\cumul\pa{x(Z), \ac{X_{ij}}_{(i,j)\in \alpha}}\enspace,
\end{equation}
 where  $X_{\alpha}=\ac{X_{ij}}_{(i,j)\in \alpha}$ is the multiset containing $\alpha_{ij}$ copies of $X_{ij}$ for $(i,j)\in [n]\times [p]$.
\end{prop}
A key feature noticed by \cite{SchrammWein22}, is that the sum in (\ref{eq:corr_D}) is sparse, due to the nullity of the  cumulants of independent variables \cite{novak2014three}.
\begin{lem}\label{lem:independentcumulant}
Let $W_1,\ldots, W_K$ be random variables on the same space $\mathcal{W}$. Suppose that there exist disjoint sets $K_1$ and $K_2$, non-empty and covering $[1,K]$, such that $(W_i)_{i\in K_1}$ and $(W_i)_{i\in K_2}$ are independent. Then, we have the nullity of the joint cumulant $\cumul\pa{W_1,\ldots, W_K}=0$.
\end{lem}
In light of these two results, the strategy of \cite{SchrammWein22} to upper bound the correlation $corr_{\leq D}$ is 
\begin{enumerate}
\item To find a large set of $\alpha$'s such that $\kappa_{x,\alpha}=0$ by using Lemma \ref{lem:independentcumulant};
\item To upper-bound the cumulants $\kappa_{x,\alpha}$ for the remaining $\alpha$'s. 
\end{enumerate}
The second step is performed by expressing the cumulants $\kappa_{x,\alpha}$  as a linear combination of the mixed moments of the signal matrix $X$ -- see Lemma~\ref{lem:mobiusformula} in Appendix \ref{sec:cumulants} --, and by applying the triangular inequality. However, this method fails for the problem of clustering when $p\leq n$, see \cite{Even24}. We manage to improve this proof strategy, by taking better advantage of the conditional independencies of the entries of the signal matrix $X$ conditionally on  the latent variable $Z$.

\subsection{Conditioning on the latent variables}\label{sec:howtoboundcumulant}
Our first main contribution is to propose a method
for efficiently bounding cumulants $\kappa_{x,\alpha}$ in the latent variable model (\ref{eq:latent-model}). This method then enables us to derive LD bounds for the problems of Clustering, Sparse Clustering and Biclustering. 
Our recipe to enhance  the proof technique of \cite{SchrammWein22} is to better exploit the conditional independences in the model. A key ingredient for handling conditional independences is  the Law of Total Cumulance, that we recall here.

\begin{lem}\label{lem:totalcumulance} [Law of Total Cumulance]
Let $W_1,\ldots,W_K$ and $Z$ be random variables on the same space $\mathcal{W}$. Then, 
$$\cumul\pa{W_1,\ldots,W_l}=\sum_{\pi\in \mathcal{P}([K])}\cumul\pa{{\cumul\pa{\pa{W_i}_{i\in R}|\enspace Z}}_{R\in \pi}}\enspace,$$
where $\mathcal{P}([K])$ denotes the set of all partitions of $[K]$.
\end{lem}

We identify $\alpha\in\N^{n\times p}$ to a multiset of $[n]\times [p]$, where each $(i,j)$ is repeated $\alpha_{ij}$ times. For $\pi\in\mathcal{P}\pa{\alpha\cup x}$ a partition of $\alpha\cup \ac{x}$, we denote by $\pi_0$ the group containing $x$.  Applying Lemma~\ref{lem:totalcumulance} and conditioning on the latent variables $Z$ leads to
\begin{equation}\label{eq:LTC1}
\kappa_{x,\alpha}=\cumul(x,X_{\alpha})=\sum_{\pi\in \mathcal{P}\pa{\alpha\cup \ac{x}}}\cumul\pa{\cumul\pa{x, X_{\pi_{0}\setminus \ac{x}} |Z}, \cumul\pa{X_{R}|Z}_{R\in \pi\setminus \ac{\pi_{0}}}}\enspace.
\end{equation}
In our setting, the benefit of conditioning by $Z$ is that many of the conditional cumulants are zero, and those that are non-zero have very simple expressions.

\begin{lem}\label{lem:boundcumulantLTCgeneral}
In the latent model \eqref{eq:latent-model} and under Assumption \ref{hyp:gaussien}, for  $\beta\in \N^{[n]\times [p]}$,
we have 
$$\cumul\pa{x,X_{\beta}|Z}=x\,\1_{\beta=0}\quad
\text{and} 
\quad \cumul\pa{X_{\beta}|Z}= \lambda^{|\beta|}\delta(Z)^{\beta}\,\1_{|\beta|=2} \,\1_{\Omega_{\beta}(Z)}\enspace,$$
where $\delta(Z)^\beta:=\prod_{(i,j)\in\beta} \delta_{ij}(Z)$, and
\begin{equation}\label{def:Omega}
\Omega_{\beta}(Z):=\big\{\delta_{ij}(Z)\neq 0,\enspace \forall (i,j)\in \beta\big\}\cap \ac{\big|\ac{\theta_{ij}(Z):\enspace (i,j)\in \beta}\big|=1}\enspace.
\end{equation}
\end{lem}

\begin{proof}[Proof of Lemma \ref{lem:boundcumulantLTCgeneral}]
For the first formula, when $\beta\neq 0$, since the variable $x$ is $\sigma(Z)$-measurable, it is independent from $X$ conditionally on $Z$. Lemma \ref{lem:independentcumulant} implies that $\cumul\pa{x,\pa{X_{ij}}_{ij\in \beta}|\enspace Z}=0$. 

When $\beta=0$, we have $\cumul(x|Z)=\E\cro{x|Z}=x$. So, we conclude 
$$\cumul(x,X_{\beta}|Z)=x\,\1_{\beta=0}\enspace.$$

For the second formula, 
if there exists $(i_0,j_0)\in \beta$ such that $\delta_{i_0j_0}(Z)=0$, then $X_{i_0j_0}=0$ and so  $\cumul\pa{X_{\beta}|Z}=0$.

Let us then prove that if $\left|\ac{\theta_{ij}(Z),\enspace (i,j)\in \beta}\right|\geq 2$, then $\cumul\pa{X_{\beta}|Z}=0$. For that purpose, let us write, for some fixed $(k,l)\in \ac{\theta_{ij}(Z),\enspace (i,j)\in \beta}$, $\pa{\beta^{(1)}}_{ij}={\beta}_{ij}\1_{\theta_{ij}(Z)=\pa{k,l}}$ and $\pa{\beta^{(2)}}_{ij}=\beta_{ij}\1_{\theta_{ij}(Z)\neq (k,l)}$. Both $\beta^{(1)}$ and $\beta^{(2)}$ are non-zero, and sum to $\beta$. Since $X_{ij}=\delta_{ij}(Z)\nu_{\theta_{ij}(Z)}$, and since the $\nu_{k,l}$'s are independent, the two families of random variables $\pa{X_{ij}}_{ij\in \beta^{(1)}}$ and $\pa{X_{ij}}_{ij\in \beta^{(2)}}$ are independent conditionally on $Z$. Thus, Lemma \ref{lem:independentcumulant} implies the nullity of $\cumul\pa{X_{\beta}|Z}$.
Finally, since $\nu_{kl}\sim\mathcal{N}(0,\lambda^2)$, when $|\ac{\theta_{ij}(Z):\enspace (i,j)\in \beta}|=1$, we have
$$\cumul\pa{X_{\beta}|Z}= \delta(Z)^{\beta} \cumul\pa{(\nu_{\theta_{ij}(Z)})_{(i,j)\in\beta}|Z}  = \delta(Z)^{\beta}\lambda^{|\beta|}\,\1_{|\beta|=2}.$$
\end{proof}

As a consequence of Lemma \ref{lem:boundcumulantLTCgeneral}, only partitions $\pi\in\mathcal{P}\pa{\alpha\cup \ac{x}}$ fulfilling $\pi_{0}=\ac{x}$ and $|\pi_{j}|=2$ for $j\geq 1$  can provide non-zero terms in the decomposition (\ref{eq:LTC1}). This set of partition is in bijection with the set of 
partitions $\pi=\ac{\pi_{1},\ldots,\pi_{l}}\in\mathcal{P}\pa{\alpha}$ fulfilling $|\pi_{j}|=2$ for $j=1,\ldots,l$. For such a partition $\pi$, there exists at least one decomposition $\alpha=\beta_1+\ldots+\beta_l$, with $l=|\pi|=|\alpha|/2$,  fulfilling $|\beta_1|=\ldots= |\beta_l|=2$, and $\beta_1,\ldots, \beta_l$ representing the groups $\pi_{1},\ldots,\pi_{l}$, i.e $\cro{\beta_{s}}_{ij}$ counts the number of copies of $(i,j)$ in $\pi_s$.
Let us define $\mathcal{B}_{\alpha}=\ac{\beta\in\N^{n\times p}:|\beta_{1}|=\ldots=|\beta_{l}|=2,\ \beta_{1}+\ldots+\beta_{l}=\alpha}$ and denote by $\mathcal{S}_{l}$ the set of permutations on $[l]$. The permutations in $\mathcal{S}_{l}$ act on $\mathcal{B}_{\alpha}$, according to the action $\sigma\cdot\beta=(\beta_{\sigma(1)},\ldots,\beta_{\sigma(l)})$. Since group labels are meaningless for partitions, each partition $\pi\in\mathcal{P}\pa{\alpha}$ with  $|\pi_{j}|=2$ for $j\geq 1$, can be represented by a unique element $\beta(\pi)\in \mathcal{B}_{\alpha}/\mathcal{S}[l]$. 
To sum-up, each  partition $\pi'\in\mathcal{P}\pa{\alpha\cup \ac{x}}$ with $\pi'_{0}=\ac{x}$ and $|\pi'_{j}|=2$ for $j\geq 1$, can be uniquely represented by a partition $\pi\in\mathcal{P}_{2}\pa{\alpha}:=\ac{\pi\in\mathcal{P}(\alpha): |\pi_{j}|=2\ \text{for}\ j\geq 1}$, which, in turns, can be represented by an element  $\beta(\pi)\in \mathcal{B}_{\alpha}/\mathcal{S}[l]$. Hence, we have
\begin{multline}\label{eq:LTCgeneral}
\lefteqn{\sum_{\pi'\in \mathcal{P}\pa{\alpha\cup \ac{x}}}\cumul\pa{\cumul\pa{x, X_{\pi'_{0}\setminus \ac{x}} |Z}, \cumul\pa{X_{R}|Z}_{R\in \pi'\setminus \ac{\pi'_{0}}}}
=}\\\sum_{\pi\in \mathcal{P}_{2}\pa{\alpha}}\cumul\pa{x,\big(\cumul(X_{\beta_{s}(\pi)}|\enspace Z)\big)_{s\in [l]}}
\enspace.
\end{multline}
We can now state our first main result, which provides a simple formula for the cumulant $\kappa_{x,\alpha}$.

\begin{thm}\label{thm:LTC}
For $\alpha\in \N^{n\times p}$, with $|\alpha|=2l$, we define
$\mathcal{P}_{2}\pa{\alpha}:=\ac{\pi\in\mathcal{P}(\alpha): |\pi_{j}|=2\ \text{for}\ j\in[l]}.$
In the latent model \eqref{eq:latent-model} and under Assumption \ref{hyp:gaussien}, for $\alpha\in \N^{n\times p}$, with $|\alpha|=2l$,  
the cumulant $\kappa_{x,\alpha}$ can be decomposed as a sum of cumulants 
\begin{equation}\label{eq:LTCgeneral:thm}
\kappa_{x,\alpha}=\lambda^{|\alpha|}\sum_{\pi\in \mathcal{P}_{2}(\alpha)}C_{x, \beta_1(\pi),\ldots, \beta_l(\pi)}\enspace,
\end{equation}
where  $\beta(\pi)\in\ac{\beta\in\pa{\N^{n\times p}}^l:|\beta_{1}|=\ldots=|\beta_{l}|=2,\ \beta_{1}+\ldots+\beta_{l}=\alpha}$, with $\cro{\beta_{s}(\pi)}_{ij}$ counting the number of copies of $(i,j)$ in $\pi_s$, and where
\begin{align}
C_{x,\beta_{1},\ldots,\beta_{l}} 
&= \cumul\pa{x,\delta(Z)^{\beta_{1}}\1_{\Omega_{\beta_{1}}(Z)},\ldots,\delta(Z)^{\beta_{l}}\1_{\Omega_{\beta_{l}}(Z)}},\label{eq:def-C}
\end{align}
with $\Omega_{\beta}(Z)$  defined in \eqref{def:Omega}, and $\delta(Z)^\beta:=\prod_{(i,j)\in\beta} \delta_{ij}(Z)$.

In particular, denoting $\beta[S]=\ac{\beta_{s}:s\in S}$, the cumulants $C_{x,\beta_{1},\ldots,\beta_{l}}$ fulfill the recursive bound
\begin{equation}\label{eq:rec-C}
|C_{x,\beta_{1},\ldots,\beta_{l}}| \leq \E\cro{|x|; \underset{s\in [l]}{\cap}\Omega _{\beta_{s}}} + \sum_{S \subsetneq [l]} |C_{x,\beta[S]}| \P\cro{\underset{s\in [l]\setminus S}{\cap}\Omega _{\beta_{s}}}. 
\end{equation}
\end{thm}

\begin{proof}[Proof of Theorem \ref{thm:LTC}]

Formula (\ref{eq:LTCgeneral:thm}) follows from (\ref{eq:LTCgeneral}), Lemma \ref{lem:boundcumulantLTCgeneral}, and the homogeneity of cumulants. Formula (\ref{eq:rec-C}) readily follows from the recursion formula for cumulants -- see \eqref{eq:rec:cumul}, page \pageref{eq:rec:cumul} --
$$C_{x,\beta_{1},\ldots,\beta_{l}} = \E\cro{x\prod_{j\in[l]}\delta(Z)^{\beta_{j}}; \underset{s\in [l]}{\cap}\Omega _{\beta_{s}}(W)} - \sum_{S \subsetneq [l]} C_{x,\beta[S]}\E\cro{\prod_{{j\in [l]\setminus S}}\delta(Z)^{\beta_{j}};\underset{s\in [l]\setminus S}{\cap}\Omega _{\beta_{s}}(Z)}\enspace, $$
and $|\delta_{ij}(Z)|\leq 1$. 
\end{proof}

 For the sake of completeness, we derive below a simple upper-bound on the cumulant (\ref{eq:def-C}), which is good enough to get results up to poly-$\log$ factors.
 
\begin{cor}\label{cor:LTC}
Under the hypotheses of Theorem~\ref{thm:LTC}, the cumulant (\ref{eq:def-C}) can be upper-bounded by
\begin{equation}\label{eq:bound-C}
|C_{x,\beta_{1},\ldots,\beta_{l}} | \leq 2 f_{l}  \max_{\pi\in \mathcal{P}\pa{[l]\cup \ac{x}}}\left\{\E\cro{|x|; \underset{s\in \pi_{1}\setminus\ac{x}}{\cap}\Omega _{\beta_{s}}} \prod_{k=2}^{|\pi|} \P\cro{\underset{s\in \pi_{k}}{\cap}\Omega _{\beta_{s}}}\right\},
\end{equation}
with $f_{l}$ the Fubini number, which fulfills
$f_{l} \leq  3\,l!\, 2^l.$
\end{cor}
 
We refer to Appendix \ref{sec:cor:LTC} for a proof of this Corollary.
The Bound  (\ref{eq:bound-C}) enables to prove meaningful computational barriers in the models considered, up to poly-log degree $D$. Yet, to prove computational barriers with sharp constants and/or higher degree $D$, we need a refined analysis, inspired by \cite{SohnWein25}, based on the recursive bound (\ref{eq:rec-C}) of Theorem~\ref{thm:LTC}.

\subsection{Deriving bounds on cumulants}
Let us now sketch how we can easily derive from Theorem~\ref{thm:LTC} some useful bounds on the cumulants $\kappa_{x,\alpha}$ of Proposition~\ref{thm:schrammwein}. While the overall strategy is similar for the different models, the precise derivation is model specific. 
As an example, we outline the derivation of a bound on $\kappa_{x,\alpha}$ for the emblematic  problem of Clustering a Gaussian mixture  (\ref{eq:GMM-intro}), which is an instantiation of the 
latent model  \eqref{eq:latent-model}, with $\nu=\mu$,
 $$Z=k^*=[k^*_1,\ldots,k^*_n]\sim\mathcal{U}([K]^n),\quad \delta_{ij}(k^*)=1,\quad \textrm{and}\quad \theta_{ij}(k^*)=(k^*_{i},j).$$
Our goal is to estimate $x=\1_{k^*_1=k^*_2}$.
 We only describe here a simple proof strategy to get a  bound on $\kappa_{x,\alpha}$, with a tight dependence in $K$, but a suboptimal dependence in $|\alpha|$.
 We refer to Appendix~\ref{prf:lowdegreeclusteringsharp} for the detailed and tighter analysis, and we refer  to  Section \ref{sec:clustering} for detailed results on the clustering problem.

 In the Gaussian clustering model,  for $\alpha\in \N^{n\times p}$, we seek to control the cumulant 
 $$\kappa_{x,\alpha}:=\cumul\pa{x, \pa{\nu_{k^*_i,j}}_{(i,j)\in \alpha}}\enspace.$$
  In the light of Theorem \ref{thm:LTC}, it is sufficient to bound, for $\beta_1+\ldots+ \beta_l=\alpha$ with $\beta_s=\ac{(i_s,j_s);(i'_s,j'_s)}$, the cumulant 
\begin{align*}
C_{x,\beta_1,\ldots,\beta_l}&=\cumul\pa{x,\pa{\1\ac{k^*_{i_s}=k^*_{i'_s}}\1\ac{j_s=j'_s}}_{s\in [l]}}\\
&=\cumul\pa{\pa{\1\ac{k^*_{i_s}=k^*_{i'_s}}\1\ac{j_s=j'_s}}_{s\in [0,l]}}\enspace,
\end{align*}
where we take the convention $i_0=1$, $i'_0=2$ and $j_0=j'_0=0$. 
Since the cumulant $C_{x,\beta_1,\ldots,\beta_l}$ is zero when $j_s\neq j'_s$ for some $s\in[l]$,
we focus on the case where $j_s=j'_s$ for all $s\in [l]$, and we seek to upper-bound 
$$C_{x,\beta_1,\ldots,\beta_l}=\cumul\pa{\pa{\1\ac{k^*_{i_s}=k^*_{i'_s}}}_{s\in [0,l]}}\enspace.$$

For any subset $S\subseteq [l]$, we write $\beta[S]=\ac{\beta_s, s\in S}$. It is convenient to introduce a graph $\mathcal{V}$ on $[0,l]$ with an edge between $s,s'$ if and only if $\ac{i_s,i'_s}$ intersects $\ac{i_{s'}, i'_{s'}}$. A first step is to remark that, according to Lemma~\ref{lem:independentcumulant}, when $S\neq \emptyset$, for having $C_{x,\beta[S]}\neq 0$, one needs to have (see Lemma~\ref{lem:nullityLTCclusteringsharp} for details)
\begin{enumerate}
\item $1,2\in \cup_{s\in S}\ac{i_s,i'_s}$;
\item The restriction of $\mathcal{V}$ to $\ac{0}\cup S$, denoted $\mathcal{V}\cro{\ac{0}\cup S}$, is connected.
\end{enumerate}
Let us call {\it active} subsets $S\subseteq [l]$, subsets  either satisfying these two conditions, or being empty. Building on the recursive bound (\ref{eq:rec-C}), we have that, for any {\it active} $S$, 
\begin{equation}\label{eq:sketch1}
|C_{x,\beta[S]}|\leq \P\cro{\forall s\in S\cup \ac{0},\enspace k^*_{i_s}=k^*_{i'_s}}+\underset{S' \text{active}}{\sum_{S'\subseteq S}}|C_{x,\beta[S']}|\P\cro{\forall s\in S\setminus S',\enspace k^*_{i_s}=k^*_{i'_s}}\enspace.
\end{equation}
Let us denote by $\# \alpha$ the number of non-zero rows of $\alpha$. 
Since the graph $\mathcal{V}\cro{S\cup \ac{0}}$ is connected, and since $1,2\in \cup_{s\in S}\ac{i_s,i'_s}$,  we have 
$$\P\cro{\forall s\in S\cup\ac{0},\enspace k^*_{i_s}=k^*_{i'_s}}=\pa{\frac{1}{K}}^{\#\alpha_{S}-1}\enspace,$$
and, for all $S'\subseteq S$, we  have
 \begin{equation}\label{eq:sktech3}
 \P\cro{\forall s\in S\setminus S',\enspace k^*_{i_s}=k^*_{i'_s}}=\pa{\frac{1}{K}}^{\#\alpha_{S\setminus S'}-cc\pa{\mathcal{V}[S\setminus S']}}\enspace,
 \end{equation}
where $cc\pa{\mathcal{V}[S\setminus S']}$ stands for the number of connected components of $\mathcal{V}[S\setminus S']$. Plugging these two formulas in \eqref{eq:sketch1}, we get
\begin{equation}\label{eq:sketch2}
|C_{x,\beta[S]}|\leq \pa{\frac{1}{K}}^{\#\alpha_{S}-1}+\underset{S' \text{active}}{\sum_{S'\subseteq S}}|C_{x,\beta[S']}|\pa{\frac{1}{K}}^{\#\alpha_{S\setminus S'}-cc\pa{\mathcal{V}[S\setminus S']}}\enspace.
\end{equation}
From this recursive bound, we derive the following upper-bound on $|C_{x, \beta[S]}|$.
\begin{lem}\label{lem:sketch}
There exists a constant $c_{|S|}$ such that $|C_{x, \beta[S]}|\leq  c_{|S|} \pa{\frac{1}{K}}^{\#\alpha_{S}-1}$.
\end{lem}
\begin{proof}[Proof of Lemma \ref{lem:sketch}]
We prove  Lemma \ref{lem:sketch} by induction over $S$.  
The initialization is immediate since $C_{x, \emptyset}=\E\cro{x}=K^{-1}$. By induction, we get from \eqref{eq:sktech3} and \eqref{eq:sketch2}
$$|C_{x,\beta[S]}|\leq \pa{\frac{1}{K}}^{\#\alpha_{S}-1}+\pa{\frac{1}{K}}^{1+\#\alpha_S-cc(\mathcal{V}[S])}+\underset{S' \text{active}}{\sum_{\emptyset\neq S'\subseteq S}}c_{|S'|}\pa{\frac{1}{K}}^{\#\alpha_{S'}-1+\#\alpha_{S\setminus S'}-cc\pa{\mathcal{V}[S\setminus S']}}.$$
Since $\mathcal{V}\cro{S\cup \ac{0}}$ is connected, it is clear that $cc(\mathcal{V}[S])\leq 2$. Thus ${1+\#\alpha_S-cc(\mathcal{V}[S])}\geq {\#\alpha_{S}-1}$, and the second term in the right-hand side in not larger than the first one.

It remains to prove that for any active non empty subset $S'\subsetneq S$, we have 
$$\#\alpha_{S'}+\#\alpha_{S\setminus S'}-1-cc(\mathcal{V}[S\setminus S'])\geq \#\alpha_S-1.$$ 
To do so, we shall use the fact that $\mathcal{V}[S\cup\ac{0}]$ is connected. All connected component $cc$ of $\mathcal{V}[S\setminus S']$ must be connected to $\ac{0}\cup S'$. In other words, for such a connected component $cc$, there exists $i\in supp(\alpha_{cc})\cap supp(\alpha_{S'})$, where $supp(\alpha)$ is the set of non-zero rows of $\alpha$. From this, we deduce that 
$$\#\alpha_{S\setminus S'}\geq \left|supp(\alpha_{S\setminus S'})\setminus supp(\alpha_{S'})\right|+cc(\mathcal{V}[S\setminus S']),$$ 
and we conclude $\#\alpha_{S'}+\#\alpha_{S\setminus S'}-1-cc(\mathcal{V}[S\setminus S'])\geq \#\alpha_s-1$. This concludes the proof of the induction. 
\end{proof}

Lemma \ref{lem:sketch} ensures that
$\displaystyle{C_{x,\beta_1(\pi),\ldots, \beta_l(\pi)}\leq c_{l}\pa{\frac{1}{K}}^{\#\alpha-1}}$, for all $\pi\in\mathcal{P}_{2}(\alpha)$.
Summing this bound over all the partitions in $\mathcal{P}_{2}(\alpha)$, ensures the existence of  a constant $C_{|\alpha|}$, only depending on the norm $|\alpha|$, such that 
\begin{equation}\label{eq:bound:sketch}
\kappa_{x,\alpha}\leq C_{|\alpha|}\pa{\frac{1}{K}}^{\#\alpha-1}\enspace.
\end{equation}
\medskip

\begin{rem}
In Appendix \ref{prf:lowdegreeclusteringsharp}, we improve this (sketch of) proof by  controlling more carefully the terms in the induction of Lemma~\ref{lem:sketch} and the number of partitions in $\mathcal{P}_{2}(\alpha)$ such that $C_{x,\beta_1(\pi),\ldots,\beta_l(\pi)}\neq 0$. This allows us to avoid powers of $D$ in the computational barrier of clustering, and to catch the BBP threshold at the exact constant, when $n\geq poly(D,K)$.
\end{rem}

\begin{rem}
In addition to upper-bounding the cumulants $\kappa_{x,\alpha}$ for any multiset $\alpha$, one also needs to prune a large number of multisets $\alpha$ such that $\kappa_{x,\alpha}=0$. This is done in the proof of Theorem \ref{thm:lowdegreeclusteringsharp} in Appendix \ref{prf:lowdegreeclusteringsharp}. Let us underline another advantage of the conditioning: it reveals that for having $\kappa_{x,\alpha}\neq 0$, it is necessary that, for all $i\in supp(\alpha)$, $|\alpha_{i:}|\geq 2$. Such a condition is necessary  to catch the exact BBP constant.
\end{rem}

\begin{rem}
In \cite{Even24}, the control of $\kappa_{x,\alpha}$ is performed without conditioning. The power of  $\frac{1}{K}$ in the upper-bound of \cite{Even24} is not ${\#\alpha-1}$ as in \eqref{eq:bound:sketch}, but instead $\max(1,{\#\alpha+r_\alpha-|\alpha|/2-1})$, where $r_{\alpha}$ is the number of non-zero rows and $|\alpha|$ is the $\ell^1$-norm of $\alpha$. This last power
 is much worse than ${\#\alpha-1}$. For example, if one considers the matrix $\alpha$  defined by $\alpha_{ij}=\1\ac{i\leq m}\1\ac{j\leq 2}$, with $m$ even, we obtain, 
 $$\text{a bound}
\ {O\pa{{\frac{1}{K^{m-1}}}}}\ \text{with conditioning,}\quad
\text{and a bound}
\ {O\pa{\frac{1}{K}}}\ \text{without conditioning.}$$
This is the reason why our result  for clustering in Theorem \ref{thm:lowdegreeclusteringsharp} holds in any dimension $p$, and not only when $p\geq n$, as  in \cite{Even24}. 
\end{rem}
 
\section{Clustering Gaussian mixtures}\label{sec:clustering}

\paragraph*{Set-up} For the reader convenience, let us recall the  Gaussian Mixture set-up \eqref{eq:GMM-intro}. 
We observe a set of $n$ points $Y_1,\ldots,Y_n\in \R^p$, which have been generated as follows. For some unknown vectors $\mu_1,\ldots,\mu_K\in\R^p$, some unknown $\sigma>0$, and an unknown partition $G^*=\ac{G^*_1,\ldots,G^*_K}$ of $\ac{1,\ldots,n}$, the points $Y_1,\ldots,Y_n$ are sampled independently with distribution
\[Y_i \sim \mathcal{N}(\mu_k,\sigma^2I_p),\quad \text{for}\ i\in G^*_k.\]
For simplicity, we focus henceforth on the case where the clusters are balanced:
\begin{equation}\label{eq:balanced}
    {\max_k |G^*_k|\over \min_k |G^*_k|}\leq \gamma,\quad \text{for some $\gamma\geq 1$.}
\end{equation}  
The clustering objective is to recover, partially or perfectly, the partition $G^*$.
Our aim is to determine what is the minimal (scaled) separation $\Delta^2$, defined by (\ref{eq:snrclustering}), required for performing better than random clustering in polynomial time.

 The minimal informational separation $\Delta^2$ for clustering better than at random has been established in \cite{Even24}. When $p\gtrsim \log(K)$, the minimal separation for partial recovery (having less than some fixed proportion of misclassified points) is -- see Theorems 2 and 3 in \cite{Even24}
\begin{equation}\label{eq:informationalclustering}
\Delta^2\gtrsim \log(K)+\sqrt{\frac{pK}{n}\log(K)}\enspace .
\end{equation} 
%where $\gtrsim$ hides some numerical constants.
However, partial recovery at the minimal separation level \eqref{eq:informationalclustering} is achieved by exactly minimizing   the Kmeans criterion over all partitions of $[n]$. The problem of minimizing the Kmeans criterion is known to be NP-hard, and even hard to approximate \cite{awasthi2015hardness}. In fact, in high dimension $p\geq n$, \cite{Even24} provides a low-degree polynomial lower-bound suggesting that the problem is computationally hard when, up to logarithmic factors, 
\begin{equation}\label{eq:computationalclustering}
\Delta^2\leq_{\log} 1+\min\pa{\sqrt{\frac{pK^2}{n}},\sqrt{p}}\enspace.
\end{equation}
It is also conjectured in \cite{lesieur2016phase} that there exists a statistical-computational gap in high dimension $p\geq \frac{n}{K^2}$. They consider the asymptotic regime, with $K$ fixed and $n/ p\to \alpha>0$, where they study the stable fixed points of the state evolution equation of Approximate Message Passing. In that setup, based on replica theory in statistical physics,  they conjecture the hardness of clustering when $p\geq_{\log} n/K^2$ and $\Delta^2\leq \sqrt{{p}K^2/n}$. Our goal in this section is to give evidence of this phenomenon, in a non-asymptotic regime, using the low-degree framework.
More precisely, 
\begin{enumerate}
\item for $K\leq_{\text{poly-log}} \sqrt{n}$, we prove low-degree hardness at the  BBP threshold $\Delta^2\leq \sqrt{{pK^2}/{n}}$  with exact constant;
\item for $K\geq_{\text{poly-log}} \sqrt{n}$, we prove low-degree hardness at the lower separation $\Delta^2\leq_{\log}\sqrt{p}$,  with a matching upper-bound, dismissing the conjecture of \cite{lesieur2016phase} when the number of clusters is high. 
\end{enumerate}
Compared to \cite{Even24}, our results are valid in any dimension $p$, including the challenging intermediate dimensions $n/K^2 \leq p \leq n$, and they are more precise as we prove a computational barrier at the exact level of the BBP transition. We also provide in Section~\ref{sec:upperbounds-clustering} a poly-time algorithm matching the LD bound in most regimes, up to poly-log factors.

\subsection{LD lower-bound for clustering}

Our main contribution for the clustering problem is to prove low-degree hardness for 
$$\Delta^2 \leq \min\pa{\sqrt{pK^2\over n},\sqrt{p\over \log^{18} n}}.$$
As explained in Section~\ref{sec:technique}, our proof starts from Proposition~\ref{thm:schrammwein} lifted from \cite{SchrammWein22}, and then build on Theorem~\ref{thm:LTC} together with some arguments adapted from \cite{SohnWein25} to derive bounds on cumulants.

Low-degree polynomials are not well-suited for directly outputting a partition $\hat{G}$, which is combinatorial by nature. Instead, we focus on the problem of estimating the partnership matrix $M^*$ defined by $M^*_{ij}=\1\{i\stackrel{G^*}{\sim} j \}$.
Indeed, proving computational hardness for estimating $M^*$, implies computational hardness for estimating $G^*$.
Given an partition $G$, define the partnership matrix $M^G$ by $M^G_{ij}=\1\{i\stackrel{G}{\sim} j \}$. By~\cite{Even24} p.5, we know that, 
\begin{equation}\label{eq:MtoG}
\frac{1}{n(n-1)}\|M^{G}-M^{*}\|_F^2 \leq \min_{\pi \in \mathcal{S}_{K}} {1\over n} \sum_{k=1}^K |G^*_{k}\triangle \hat G_{\pi(k)}|=:2\,err(\hat G,G^*),
\end{equation}
where $\triangle$ represents the symmetric difference,  $\mathcal{S}_{K}$ the permutation group on $[K]$, and where $err(\hat G,G^*)$ is the average proportion of misclassified points in $\hat G$. Hence, estimating $M^{*}$ in polynomial-time with small square-Frobenius distance is no harder than building a polynomial-time estimator $\hat{G}$ with a small error $err(\hat{G},G^*)$. By  linearity, we focus on estimating the functional $x=M^*_{12}=\1\{1\stackrel{G^*}{\sim} 2\}$. 
In Appendix~\ref{subsec:impossibility_reconstruction}, we also show that the hardness for reconstructing $x$ within a square error $K^{-1}(1+(o(1)))$ implies that all polynomial-time balanced estimator $\hat{G}$ achieve an error $err(\hat{G},G^*)\geq 1+o(1)$. In other words, it is impossible to perform better than random guess.

For proving the LD bound, we consider the following prior on the means $\mu_{k}$ and the partition $G^*$.
\begin{definition}\label{def:priorclustering}
Let $k^*$ be a random variable uniformly distributed on $[K]^n$, and  for $k\in[K]$ set  $G^*_{k}=\ac{i\in[n]:k^*_{i}=k}$. Furthermore, let the $\mu_{k}$ be random variables independent of $k^*$, with distribution
$$\mu_{kj}\stackrel{\text{i.i.d.}}{\sim} \mathcal{N}(0,\lambda^2),\quad \text{with}\ \lambda^2=\frac{1}{p}\bar{\Delta}^{2}\sigma^2.$$
\end{definition}
The prior in Definition \ref{def:priorclustering} is an instantiation of the model \eqref{eq:latent-model}, with 
$$Z=k^*,\quad \delta_{ij}(k^*)=1,\quad \textrm{and}\quad \theta_{ij}(k^*)=(k^*_{i},j)\enspace.$$
We emphasize that, with high probability, we have a separation $\Delta^2 =\Bar{\Delta}^2\pa{1+o_p(1)}$ under this prior.
 The risk of the trivial estimator $\hat x=\E[x]$ of $x$ is
$$MMSE_{\leq 0}=\text{var}(x)=\frac{1}{K}-\frac{1}{K^2}.$$
The next theorem provides conditions ensuring that $MMSE_{\leq D}=MMSE_{\leq 0}\pa{1+o_{K}(1)}$.

\begin{thm}\label{thm:lowdegreeclusteringsharp}
Let $D\in\N$ with $D^5\leq p$ and assume that $\zeta:=\frac{\Bar{\Delta}^4}{p\sigma^4}\max\pa{ D^{18},\frac{n}{K^2}}<1$. Then, under the prior of Definition \ref{def:priorclustering}, $$MMSE_{\leq D}\geq \frac{1}{K}-\frac{1}{K^2}\cro{1+\frac{\zeta}{(1-\sqrt{\zeta})^3}}\enspace.$$
\end{thm}

Theorem \ref{thm:lowdegreeclusteringsharp}, proved in  Appendix~\ref{prf:lowdegreeclusteringsharp}, improves on Theorem 1 of \cite{Even24} in two directions:
 First, it is valid for any dimension $p\geq D^5$, while Theorem 1 of \cite{Even24} only covered the simplest case $p\geq n$; 
 Second, the exact BBP threshold $\sqrt{{pK^2}/{n}}$ appears in $\zeta$, whereas there is a spurious factor $D^{12}$  in Theorem 1 of \cite{Even24}.
Theorem \ref{thm:lowdegreeclusteringsharp} implies that, for any numerical constant $\eps>0$, if $\sqrt{\zeta}\leq 1-\eps$ and $p\geq D^5$, then there exists $C(\eps)>0$ such that $$MMSE_{\leq D}\geq \frac{1}{K}-\frac{C(\eps)}{K^2}=MMSE_{\leq 0}\pa{1+o_{K}(1)}\enspace.$$
Hence, no degree $D$ polynomials can perform significantly better than the trivial estimator in this regime.
In particular, taking $D=(\log n)^{1+\eta}$, if $p\geq (\log n)^{5(1+\eta)}$, we prove $(\log n)^{1+\eta}$--degree hardness when
\begin{equation}\label{eq:computationalclusteringsharp}
\bar{\Delta}^2\leq (1-\eps)\min\pa{\sqrt{\frac{pK^2}{n}},\sqrt{\frac{p}{(\log n)^{18(1+\eta)}}}}\enspace.
\end{equation}
Since polynomials of degree at most $\pa{\log n}^{1+\eta}$ are considered as a proxy for algorithms that are computable in polynomial time~\cite{KuniskyWeinBandeira,SchrammWein22},  Theorem \ref{thm:lowdegreeclusteringsharp} together with \eqref{eq:MtoG} suggest the computational hardness of clustering in the regime \eqref{eq:computationalclusteringsharp}
We remark that, when $K^2\leq  n/ (\log n)^{18(1+\eta)}$, the computational barrier \eqref{eq:computationalclusteringsharp} reduces to
\begin{equation}\label{eq:BBPclustering}
\bar{\Delta}^2\leq (1-\eps)\sqrt{\frac{pK^2}{n}}\enspace,
\end{equation}
as conjectured in \cite{lesieur2016phase} with replica heuristics. This barrier matches exactly  the BBP transition threshold \cite{banks2018information}, where,
in the asymptotic regime $n/p\to \alpha$ and $K\ll n,p$, 
 the leading eigenvalues of the matrix $Y^T Y$ become significantly larger from those of the Wigner matrix $E^TE$. We build on this property in Proposition~\ref{prop:upper-boundclustering}, in order to design a poly-time algorithm, which recovers the partition $G^*$ after projecting the data onto the low dimensional space spanned by the leading eigenvectors of the matrix $Y^T Y$.

Finally, we remark that, in low-dimension $p\leq n (\log(K)/K)^2$, the computational barrier \eqref{eq:computationalclusteringsharp} is smaller than the informational barrier \eqref{eq:informationalclustering}. Combining these two barriers, we  provide evidence, when $p\geq (\log n)^{5(1+\eta)}$, that partial recovery of $G^*$ is computationally hard below the threshold
\begin{equation}\label{eq:computationalclusteringsharp+informational}
\bar{\Delta}^2\leq (c \log K)\vee \pa{(1-\eps)\min\pa{\sqrt{\frac{pK^2}{n}},  \sqrt{\frac{p}{\log(n)^{18(1+\eta)}}}}}\enspace.
\end{equation}
In the next section, we show that clustering in poly-time is possible, in almost all regimes, above the level \eqref{eq:computationalclusteringsharp+informational}. This provides an almost complete picture of the computational barrier for clustering Gaussian mixtures.

\subsection{Matching the LD bound with a Spectral Method}\label{sec:upperbounds-clustering}
Let us first recall some known poly-time  algorithms that, for some regimes of $n,p,K$, succeed to recover $G^*$  above the separation level \eqref{eq:computationalclusteringsharp+informational}, up to log factors.
\begin{itemize}
\item \underline{Many groups $K\gtrsim \sqrt{n}$.} Hierarchical Clustering with single linkage, recovers exactly, with high probability,  the partition $G^*$  when $\Delta^2\gtrsim {\log(n)+\sqrt{p\log(n)}}$, see e.g. Proposition~4 in \cite{Even24}. Thus, if $n\leq c K^2$ for some constant $c$, it is easy to recover exactly $G^*$ when the separation is larger, up to some logarithmic factor, than the barrier \eqref{eq:computationalclusteringsharp+informational}, i.e.\ $\Delta^2\geq_{\text{poly-log}} \sqrt{p}$ in this regime;
\item \underline{High dimension $p\geq n$.} When $p\geq n$, some SDP relaxation of Kmeans, recover partially $G^*$, with high probability, when the separation is, up to some numerical constant, above the BBP threshold $\sqrt{{pK^2}/{n}}$, see Theorem 1 and Formula (12) in~\cite{giraud2019partial};
\item \underline{Low dimension $\text{poly}(p,K)\leq n$.} Liu and Li~\cite{LiuLi2022} provides a poly-time algorithm which, with high probability,  recover partially $G^*$ in poly-time,  when $\Delta^2\geq ( \log K)^{1+c}$, and exactly when $\Delta^2\geq  (\log n)^{1+c}$, where $c$ is a constant depending on the polynomial of the condition $n\geq \text{poly}(p,K)$, see Theorem~2.5 and Corollary~2.6 in~\cite{LiuLi2022}. Their result ensures in particular the absence of statistical-computational gap (up to log factors) in low dimension.
\end{itemize}
These three results show that some poly-time algorithms succeed  to recover $G^*$  above the separation level~\eqref{eq:computationalclusteringsharp+informational} -- up to log factors --, either when the number of groups is high ($K\gtrsim \sqrt{n}$), or when the dimension is high ($p\geq n$) or small ($\text{poly}(p,K)\leq n$). It remains to figure out if there exist some algorithms succeeding above the separation level~\eqref{eq:computationalclusteringsharp+informational} in moderate dimension $\text{poly-log}(n)\leq p\leq n$ with a small number of groups $K\lesssim \sqrt{n}$. Below, we show that we can find such  algorithms in almost all, but not all, this regime.  

For simplicity, we focus on the objective of perfect recovery of $G^*$ with high probability. The next proposition shows that clustering is possible in poly-time above the threshold \eqref{eq:computationalclusteringsharp+informational} --up to logarithmic factors--, except in the regime where both $p\leq {n}/{K}$ and $K^2\lesssim n \leq \text{poly}(K)$, where the problem of optimal poly-time clustering remains open. We postpone to Appendix~\ref{prf:upper-boundclustering} the proof of this proposition.

\begin{prop}\label{prop:upper-boundclustering}
Assume that the unknown partition $G^*$ is $\gamma$-balanced~\eqref{eq:balanced}.   
\begin{enumerate}
\item There exist positive constants $c_{\gamma}$, $c'_{\gamma}$ depending only on $\gamma$, such that the following holds.\\ If $n\geq c_{\gamma} K^2$, $n\geq p\geq \frac{n}{K}$ and $\Delta^2\geq c'_{\gamma}{\log(n)+\sqrt{{pK^2\log(n)}/{n}}}$, it is possible to recover exactly $G^*$ in poly-time with probability $1-O(n^{-2})$;
\item There exists a constant $c''_{\gamma}>0$ depending only on $\gamma$, such that, for all $\eps>0$, there exists $c_3(\eps,\gamma)>0$ satisfying the following.\\ If $n\geq K^{c_3(\eps,\gamma)}$ and $\Delta^2\geq c''_{\gamma}{\log(n)^{1+\eps}+\sqrt{{pK^2}/{n}}}$, it is possible to recover exactly $G^*$ in poly-time with probability $1-O(n^{-c})$.
\end{enumerate}
\end{prop}

According to Proposition~\ref{prop:upper-boundclustering}, perfect recovery can be achieved in poly-time above the threshold 
$$\Delta^2\geq_{\log} {1+\min\pa{\sqrt{p}, \sqrt{\frac{pK^2}{n}}}}\enspace,$$
except when both poly-log$(n)\leq p\leq \frac{n}{K}$ and $K^2\leq n\leq \text{poly}(K)$.
The poly-time algorithm achieving  this result, essentially proceeds as follows.\\
1- First, it projects the data points onto the $K$-dimensional space spanned by the $K$ leading eigenvectors of $Y^TY$;\\ 
2- Second, it applies on the projected data points,  either  hierarchical clustering with single linkage  (first claim of Proposition \ref{prop:upper-boundclustering}), or  the tensor-based algorithm of~\cite{LiuLi2022} (second claim). 

The actual algorithm turns out to be a bit more involved, with some sample splitting to handle dependencies between the first and second step, we refer to  Appendix~\ref{prf:upper-boundclustering}  for the detailed description. 
The key result on which Proposition~\ref{prop:upper-boundclustering}
relies is Lemma~\ref{lem:projmixture}. This lemma ensures that, above the BBP threshold, at least a positive fraction of the signal is remaining after projection along the $K$ leading eigenvectors of $Y^TY$. Hence, the ambient dimension is reduced from $p$ to $K$, while preserving a fraction of the signal, reducing the initial problem to the problem of clustering in dimension $K$, for which we can apply existing optimal algorithms.

\medskip
\begin{rem}
In Section \ref{prf:upper-boundclustering}, Proposition \ref{prop:hierarchical2} provides a result  valid in any dimension, completing Proposition~\ref{prop:upper-boundclustering}. This result states that when $n\gtrsim K^2$ and 
\begin{equation}\label{eq:suboptimal-clustering}
\Delta^2\geq c_\gamma\pa{\log(n)+\sqrt{K\log(n)}+\sqrt{\frac{pK^2\log(n)}{n}}}\enspace,
\end{equation}
hierarchical clustering with single linkage applied on the projected dataset
exactly recovers  $G^*$ with high probability. Despite not matching the lower bound of Theorem \ref{thm:lowdegreeclusteringsharp}, Proposition \ref{prop:hierarchical2} improves on the result from \cite{giraud2019partial} where a separation $\Delta^2\gtrsim \log(n)+K+\sqrt{{pK(K+\log(n))}/{n}}$ is required. 
\end{rem}

\section{Sparse clustering}\label{sec:sparse}

In this section, we investigate the same problem of clustering an isotropic Gaussian mixture, but with the additional assumption  that the means of the mixture are sparse.

\paragraph*{Set-up} Let us recall the sparse clustering model. We observe a set of $n$ points $Y_1,\ldots, Y_n\in \R^p$ which have been generated as follows. For some known $s\in [p]$, there exists an unknown subset $J^*\subseteq [p]$, with cardinality $|J^*|\leq s$, such that, the unknown means $\mu_1,\ldots, \mu_K$ are all supported on $J^*$, which means that $\mu_{kj}=0$ for all $j\notin J^*$. 
Then, for some $\sigma>0$, and some unknown partition $G^*=\ac{G_1^*,\ldots, G^*_K}$ of $[n]$, the points $Y_1,\ldots, Y_n\in\R^p$ are sampled independently with distribution 
\[Y_i \sim \mathcal{N}(\mu_k,\sigma^2I_p),\quad \text{for}\ i\in G^*_k.\]
Again, we assume that the hidden partition $G^*$ is balanced, i.e that it satisfies \eqref{eq:balanced},
and 
as in Section \ref{sec:clustering}, we analyse the minimal separation 
 \eqref{eq:snrclustering} 
required for successful clustering in poly-time and without time constraints.

The sparse clustering model is a particular instance of the Gaussian mixture model. Hence, the upper-bounds for clustering an isotropic Gaussian Mixture still hold in the case of sparse clustering. The computational lower-bound of Theorem~\ref{thm:lowdegreeclusteringsharp} does not yet hold here, since the prior of Definition~\ref{def:priorclustering} is not sparse. We investigate in this section, whether the sparsity of the centers can help to recover in poly-time the partition $G^*$
 below the computational barrier \eqref{eq:computationalclusteringsharp+informational}. Our contributions are:
 \begin{enumerate}
\item Implementing the technique of Theorem \ref{thm:LTC}, we provide a LD lower bound for the problem of sparse clustering. When $s \leq_{\log} \sqrt{p(K^2\wedge n)}$,  this lower-bound 
corresponds to the computational barrier \eqref{eq:computationalclusteringsharp+informational} in reduced dimension $p=s$, plus an additional term $\sqrt{s^2/n}$, which can be interpreted as the minimal signal required to recover the active columns set $J^*$ in poly-time  before clustering.
When $s \leq_{\log} \sqrt{p(K^2\wedge n)}$, this  additional term $\sqrt{s^2/n}$ becomes larger than the computational barrier \eqref{eq:computationalclusteringsharp+informational} in dimension $p$, and only this computational barrier applies. 
\item Inspired by this lower-bound, we analyze a method that seeks to estimate the active columns set $J^*$ in polynomial time, and then clusters the points after removing the non-selected columns. Under an additional  assumption of homogeneity  of the signal along $J^*$, this method succeeds to cluster above the low-degree barrier obtained in the first step, up to log factors. This result supports our interpretation that the minimal separation for poly-time clustering is the sum of the separation $\sqrt{s^2/n}$  required for first recovering the active columns set $J^*$, plus the minimal separation for poly-time clustering in dimension $s$.
\item Under the same additional homogeneity assumption, we analyse an algorithm, not computable in polynomial time, which succeeds to cluster above the statistical rate \eqref{eq:informationalclustering} in reduced dimension $p=s$, when, in addition, $\Delta^2\geq {s\sqrt{K}}/{n}$, up to log terms. This last constraint corresponds to the separation required for recovering the active columns set $J^*$ once the clustering is known. 
 Since $K\leq n$, we observe that this separation level $\Delta^2\geq {s\sqrt{K}}/{n}$ is always smaller than the separation level $\sqrt{s^2/n}$ required for poly-time algorithms. 
We underline then a contrastive phenomenon for sparse-clustering under the homogeneity assumption. 
The additional separation $\Delta^2\gtrsim \sqrt{s^2/n}$ required in poly-time corresponds to the separation needed for recovering the active columns \emph{before} clustering, while the statistical additional separation $\Delta^2\gtrsim {s\sqrt{K}}/{n}$ corresponds to the separation needed for recovering the active columns \emph{after} clustering, exhibiting a better ability to fully exploit the joint sparse-and-clustered structure by non poly-time algorithms.

When $s\leq n$, we observe that the separation level $\Delta^2\gtrsim {s\sqrt{K}}/{n}$ is even  smaller than the minimal statistical separation for clustering in dimension $s$, hence, in this specific case, sparse clustering without computational constraints is not harder than clustering in dimension $s$ without computational constraints. 
\end{enumerate}
Comparing the statistical and the computational rates,  
we observe the existence of a statistical-computational gap when either (i) $s\geq n$, or (ii) $s\in [K,n]$ and $n\leq_{\log} [pK^2\wedge s^2]$, or (iii) $s\leq K$ and $n\leq_{\log} K^2s$.
%$s\geq_{\log} K$, or when $n\wedge K^2\geq_{\log} K( 1 \vee {s\over n})$.
In particular, while the sparsity of the means makes the problem easier, both statistical and computationally, it  widens the computational gap.

\subsection{LD lower-bound for sparse clustering}

Let us introduce the prior under which we derive our LD bound.
A simple choice could be to consider the same prior as in Definition \ref{def:priorclustering}, introducing sparsity by keeping the signal only on $s$ columns randomly chosen. Yet, such a prior introduces some weak-dependencies between the entries of a column, and the LD bounds that we would obtain under this prior  would be suboptimal --see the discussion and derivations in Appendix~\ref{sec:discussion:prior:sparse}. 
 To overcome this issue, we introduce some symmetrization in the generation of the means. For simplicity, we consider a partition into $2K$ groups,  generate $K$ means as suggested before and then symmetrize them to get the remaining $K$ means. This process could have be applied with $K$ groups, instead of $2K$, by symmetrizing the centers of the first $2\lfloor K/2 \rfloor$ groups. However, this adds an unnecessary layer of complexity in  the proof. In the following prior, the groups correspond to the different values of $(k^*_i, \eps_i)\in [K]\times \ac{-1,1}$. 
\begin{definition}\label{def:prior_sparse}
The signal matrix  $X\in \R^{n\times p}$ is generated as follows. We sample independently:\\ \smallskip
- $k^*_1,\ldots ,k^*_n$ independent with uniform distribution on $[K]$,\\ \smallskip
- $z_1,\ldots,z_p$ independent, with Bernoulli  distribution $\mathcal{B}(\rho)$, where $\rho={\Bar{s}}/{p}$,\\ \smallskip
- $\eps_1,\ldots, \eps_n$ independent with uniform distribution on $\ac{-1,1}$,\\ \smallskip
- $\nu_{k,j}$, for $k,j\in [K]\times [p]$, independent, with $\nu_{k,j}\sim\mathcal{N}\pa{0,\lambda^2}$, where $\lambda^2=\Bar{\Delta}^2\sigma^2/{\rho p}$.\\ \smallskip
Then, we set 
\begin{equation}\label{eq:prior-sparse}
\X_{ij}= z_j\eps_i\nu_{k^*_i,j}\ .
\end{equation}
\end{definition}

Under the prior~\eqref{eq:prior-sparse}, the set of active columns is $J^*=\ac{j:z_{j}=1}$, and the partition $G^*$ is defined by
$G^*_{k}=\ac{i: k^*_{i}=k,\ \text{and}\ \eps_{i}=1}$ for $k=1,\ldots,K$ and  $G^*_{k}=\ac{i: k^*_{i}=k-K,\ \text{and}\ \eps_{i}=-1}$ for $k=K+1,\ldots,2K$.

\begin{rem}
With high probability, under the sparse prior (\ref{eq:prior-sparse}), we have a separation $\Delta^2= \Bar{\Delta}^2(1+o_{\rho p}(1))$. 
\end{rem}

\begin{rem}
Let $s(z):=|\ac{j\in [p]; z_j=1}|$. 
It readily follows from large deviation inequality for Bernoulli variable, see e.g. Section 12.9.7 in \cite{HDS2},that $\P[s(z)>5\rho p]\leq \exp\pa{-\rho p/2}$.
Hence, when $\rho p\geq log(n)$, with high probability, the model fulfills a sparsity assumption with $s=5\rho p$. 
\end{rem}

As in Section \ref{sec:clustering}, we consider the estimation of the variable $x={\1}_{k^*_{1}=k^*_{2}}$.
The next theorem, proved in Section \ref{prf:lowdegreesparsesharp},  provides a lower-bound on the $MMSE_{\leq D}$ for estimating $x$.
\begin{thm}\label{thm:lowdegreesparsesharp}
Let $D\in \N$ and assume $\zeta:=\frac{\bar{\Delta}^4}{\rho^2 p^2}\max\pa{D^{14}, D^7n, D^7\rho^2 p, \rho^2p\frac{n}{K^2}}<1$. Then, under the prior distribution of Definition \ref{def:prior_sparse}, 
$$MMSE_{\leq D}\geq \frac{1}{K}-\frac{1}{K^2}\cro{1+\frac{\zeta}{\pa{1-\sqrt{\zeta}}^3}}\enspace.$$
\end{thm}

In particular, for any $\eps>0$, if $\sqrt{\zeta}\leq 1-\eps$, then $$MMSE_{\leq D}=MMSE_{\leq 0}(1+o_K(1))\enspace.$$
If $D\leq (\log n)^{1+\eta}$ and $\bar{\Delta}^2\leq (1-\eps)\min\pa{\sqrt{\frac{\rho^2p^2}{n(\log n)^{9(1+\eta)}}}, \sqrt{\frac{p}{(\log n)^{9(1+\eta)}}}, \sqrt{\frac{pK^2}{n}}}$, then $MMSE_{\leq D}=\frac{1}{K}-\frac{1}{K^2}\pa{1+o_K(1)}$. Since the class of polynomials of degree at most $(\log n)^{1+\eta}$ is considered as a proxy for algorithms computable in polynomial time, Theorem~\ref{thm:lowdegreesparsesharp} provides evidence that sparse clustering is hard when 
\begin{equation}\label{eq:sparse:barrier1}
\bar{\Delta}^2\leq (1-\eps)\min\pa{\sqrt{\frac{\bar{s}^2}{n(\log n)^{9(1+\eta)}}}, \sqrt{\frac{p}{(\log n)^{9(1+\eta)}}}, \sqrt{\frac{pK^2}{n}}}\enspace.
\end{equation}
We recognize in \eqref{eq:sparse:barrier1} the barrier $\min (\sqrt{\frac{p}{(\log n)^{9(1+\eta)}}}, \sqrt{\frac{pK^2}{n}})$ from clustering in dimension $p$ -- see \eqref{eq:computationalclusteringsharp}. 
We notice yet that we have here the BBP threshold $\sqrt{pK^2/n}$ for $K$ groups, instead of $2K$ groups, due to the symmetrization in the prior of Definition~\ref{def:prior_sparse}, thereby using a factor $2$. The barrier  \eqref{eq:sparse:barrier1} can yet be smaller than the barrier for clustering in dimension $p$, due to the additional term $\sqrt{{\bar{s}^2}/\pa{n(\log n)^{9(1+\eta)}}}$. As we will see in the next section, this term can be interpreted as the barrier for estimating the non-zero columns of the signal. 

We observe that when $\bar{s}	\leq_{\log} n\wedge K^2$, the barrier  \eqref{eq:sparse:barrier1} becomes smaller than the computational barrier $c\log(K)\vee\min \pa{\sqrt{\frac{\bar{s}}{(\log n)^{18(1+\eta)}}}, \sqrt{\frac{\bar{s}K^2}{n}}}$ for clustering in dimension $\bar{s}$. 
Furthermore, the partial matrix $Y_{:J_{*}^c}$ is independent of $X$ conditionally on $J^*$.
So sparse-clustering is at least as hard as clustering the $n\times |J^*|$ matrix $Y_{:J^*}$. 
In other words, sparse-clustering cannot be easier than clustering in dimension $|J^*|$. Hence, lifting the bound proved in Theorem~\ref{thm:lowdegreeclusteringsharp} with $p=s$, 
we get that sparse clustering is low-degree hard when 
\begin{multline}\label{eq:computationalsparse}
\lefteqn{\bar{\Delta}^2\leq (c \log K)\vee \pa{(1-\eps)\min\pa{\sqrt{\frac{\bar{s}K^2}{n}},  \sqrt{\frac{\bar{s}}{(\log n)^{18(1+\eta)}}}}}}\\
\vee\pa{(1-\eps)\min\pa{\sqrt{\frac{\bar{s}^2}{n(\log n)^{9(1+\eta)}}}, \sqrt{\frac{p}{(\log n)^{9(1+\eta)}}}, \sqrt{\frac{pK^2}{n}}}}\enspace,
\end{multline}
which reduces, up to poly-logarithmic factors, to 
\begin{equation}\label{eq:sparse-barrier-simple}
\bar{\Delta}^2\leq_{\log}1+ {\min\pa{\sqrt{\frac{\bar{s}K^2}{n}}, \sqrt{\bar{s}}}+\sqrt{\frac{\bar{s}^2}{n}}}
\quad\text{and}\quad
\bar{\Delta}^2\leq_{\log} 1+ {\min\pa{\sqrt{p}, \sqrt{\frac{pK^2}{n}}}}\enspace.
\end{equation}
The first term in \eqref{eq:sparse-barrier-simple} is the sum of the computational barrier for clustering in dimension $\bar{s}$, plus a computational barrier  $\sqrt{\bar{s}^2/{n}}$ for recovering the $\bar{s}$ active columns. The second condition is simply the computational barrier for clustering in dimension $p$.  In some way, our bound~\eqref{eq:sparse-barrier-simple} extends the LD lower bound of~\cite{lffler2021computationallyefficientsparseclustering} for $K=2$ to all $(K,s)$ regimes.

\subsection{Poly-time sparse clustering}\label{sec:polytimesparse}
The setup of sparse clustering is a particular instance of clustering. Hence all the upper-bounds for clustering a Gaussian Mixture hold for sparse clustering, and we can cluster in poly-time above the second term in \eqref{eq:sparse-barrier-simple} in almost all regimes of $n,p,K$, see Section~\ref{sec:upperbounds-clustering}.

It remains to check that sparse clustering is possible in polynomial time when, up to some logarithmic factors, 
\begin{equation}\label{eq:upper:sparse:sparse}
\Delta^2\geq_{\log} {1+\sqrt{\frac{s^2}{n}}+\min\pa{\sqrt{s}, \sqrt{\frac{s K^2}{n}}}}\enspace.
\end{equation}
Since the term ${1+\min\pa{\sqrt{s}, \sqrt{{s K^2}/{n}}}}$ corresponds, up to some logarithmic factors, to the computational barrier \eqref{eq:computationalclusteringsharp+informational} for clustering in dimension $s$, it is natural~\cite{lffler2021computationallyefficientsparseclustering,mun2025high} to proceed by first detecting the active columns $J^*$ of $Y$ on which the signal is supported, then remove all the other columns of $Y$, and finally applying a clustering procedure on the reduced $Y_{:J^*}$. 

In a general, it is not easy, for $K\geq 2$, to recover exactly all the columns on which the signal is supported. We can still find columns on which most of the centers $\mu_k$ have a large part of their weight, and then recover the corresponding groups. 
However, this strategy is hard to analyse --see Appendix~\ref{sec:discussion:w}, and, for simplicity of the proof, we consider a much simpler algorithm, merely 
selecting  the $s$ columns $\hat{J}$ of $Y$ with largest Euclidean norm.
This simple method will be successful under
a minimum column-signal assumption. We denote by
\begin{equation}\label{eq:definition_omega_J*}
\omega_{J^*}^2:=\frac{1}{\sigma^2}\min_{j\in J^*}\sum_{i\in n}X_{ij}^2={1\over \sigma^2}\sum_{k\in [K]}|G^*_{k}|\mu_{k,j}^2\enspace,
\end{equation}
the minimum $\ell^2$-norm of the active columns of $X$,
and we assume that $\omega_{J^*}^2\geq_{\log}\sqrt{n}$.
 Next lemma states that, under the previous minimum column-signal assumption, the estimator $\hat{J}$ contains  $J^*$ with high probability.
\begin{lem}\label{prop:recoverycolumns}
There exists a numerical constant $c_1>0$ such that the following holds. If $\omega_{J^*}^2\geq c_1\pa{\sqrt{n \log(pn)}+\log(p)}$, then, with  probability higher than  $1-\frac{1}{n^2}$, $\hat{J}$ contains  $J^*$. 
\end{lem} 
Let us briefly explain why the condition $\omega_{J^*}^2\geq_{\log} \sqrt{n}$ condition ensures that $\hat J=J^*$ with high probability, we refer to Appendix \ref{prf:polytimesparse} for the detailed proof of Lemma~\ref{prop:recoverycolumns}.
The square norm $\|Y_{:j}\|^2$ has mean
$\E\cro{\|Y_{:j}\|^2}=\|X_{:j}\|^2+n\sigma^2$, 
and standard deviation sdev$\cro{\|Y_{:j}\|^2}=\sigma^2\sqrt{n}$. Hence,  as soon as $\omega_{J^*}^2 \geq_{\log} \sqrt{n}$, the set $J^*$ belongs to the $s$ columns of $Y$ with maximum Euclidean norm.

It turns out --see Lemma~\ref{lem:homogeneity} and Corollary~\ref{cor:polytimesparse} below-- that, under some homogeneity conditions on the signal, the condition  $\omega_{J^*}^2\geq_{\log} \sqrt{n}$ corresponds to our regimes of interest in~\eqref{eq:upper:sparse:sparse}. 
Once the columns $J^*$ have been retrieved, we can remove all the other columns, and apply a clustering procedure in dimension $s$, leading to the next result proved  in Appendix \ref{prf:polytimesparse}

\begin{prop}\label{prop:polytimesparse}
There exist constants $c$, $c_1$,  and $c_2>0$ such that the following holds for any $\gamma$-balanced partition $G^*$ --see~\eqref{eq:balanced}. Suppose that either $s\notin [\text{poly-log}(n), n/K]$ or $n\notin [K^2, K^c]$. Then, if $$\omega_{J^*}^2\geq c_1\pa{\sqrt{n\pa{\log(pn)}}+\log(p)}\quad \text{and}\quad \Delta^2 \geq_{\log,\gamma} {1+\min\pa{\sqrt{s}, \sqrt{\frac{sK^2}{n}}}}\enspace,$$
there exists an algorithm computable in polynomial time which recovers exactly $G^*$ with probability higher than $1-n^{-c_2}$. 
\end{prop}

When the signal is well spread along the active columns of $J^*$, a large separation $\Delta^2$  implies a large minimum $l^2$-norm of the active columns of $X$. In the following, we consider the case where the matrix $X$ satisfies the following homogeneity assumption.

\begin{hyp}\label{ass:homogeneity}[$\eta$-homogeneity]
For some $\eta\geq 1$, the matrix $X$ satisfies 
\begin{equation}\label{eq:assump2}
\frac{\max_{j\in J^*}\|X_{:j}\|^2}{\min_{j\in J^*}\|X_{:j}\|^2}\leq \eta\enspace.
\end{equation}
\end{hyp}

\begin{rem}
With probability larger than $1-{(n\vee p)^{-2}}$, the prior of Definition \ref{def:prior_sparse} is $\eta$-homogeneous with $\eta\leq c\sqrt{{\log(np)}/{n}}$. 
\end{rem}

We remark that when $X$ satisfies the $\eta$-homogeneity assumption, we can lower-bound the minimum $l^2$-norm of the active columns of $X$. We postpone to Section \ref{prf:homogeneity} the proof of the next lemma.

\begin{lem}\label{lem:homogeneity}
Assume that $X$ satisfies the $\eta$-homogeneity Assumption \eqref{eq:assump2}. Then, $$w_{J^*}^2\geq \frac{n(K-1)}{2sK\gamma\eta}\Delta^2\enspace.$$
\end{lem}

Combining Lemma \ref{lem:homogeneity} with Proposition \ref{prop:polytimesparse} directly implies the following corollary. 

\begin{cor}\label{cor:polytimesparse}
Assume that $X$ satisfies both the $\eta$-homogeneity assumption \eqref{eq:assump2} and the balancedness condition  \eqref{eq:balanced} with $\eta, \gamma\leq \text{poly-log}(np)$. Then, except in the regime where $s\in [\text{poly-log}(n), n/K]$ and $n\in [K^2, poly(K)]$, if $$\Delta^2\geq_{\log} 1+\min\pa{\sqrt{s},\sqrt{\frac{sK^2}{n}}}+\frac{s}{\sqrt{n}}\enspace,$$ we can recover perfectly $G^*$ in polynomial time, with probability higher than $1-n^{-c}$, for some constant $c$. 
\end{cor}

We also prove in Appendix~\ref{sec:discussion:w}, that, when $K\leq 4$, the condition of $\eta$-homogeneity can be dropped in Corollary~\ref{cor:polytimesparse}. 

In summary, we have introduced polynomial-time estimators that match our low-degree polynomial lower bound~\eqref{eq:sparse-barrier-simple} in almost all regimes. Our results show-case that, for the sparse clustering problem, in polynomial-time, one cannot do significantly better, than  applying an agnostic clustering procedure (oblivious of the sparse structure) or applying a simple dimension reduction scheme together with clustering procedure in the reduced space. 
The only regimes where there is mismatch, namely when $n\in [K^2; poly(K)]$ and either $s\in [\text{poly-log}(n), n/K]$ or $p\in  [\text{poly-log}(n), n/K]$, are the counterparts of those that have arisen in Proposition~\ref{prop:upper-boundclustering} for  clustering.

\subsection{Upper-bound on the minimal statistical separation for sparse clustering}\label{sec:itsparse}

Our previous results characterize the optimal separation conditions in polynomial-time~\eqref{eq:sparse-barrier-simple}. We now highlight the statistical-computational gaps for this problem 
by providing sufficient conditions for, possibly non-polynomial time procedures, to recover the partition $G^*$. We already deduce from~\cite{Even24} --see also the previous subsection-- that 
$\Delta^2\gtrsim \log(K)+ \sqrt{pK\log(K)/n}$ is sufficient for the exact $K$-means estimator to achieve partial recovery, thereby lowering  the second part of low-degree Condition~\eqref{eq:sparse-barrier-simple} by a factor $\sqrt{K/\log(K)}$. Here, we focus on the second statistical-computational gap arising in sparse clustering, which is pertaining to the detection of the active columns. 
Recall the definition~\eqref{eq:definition_omega_J*} of $\omega_{J^*}^2$ as the minimum squared $\ell^2$-norm of the active columns of $X$.

\begin{prop}\label{prop:upperboundsparseIT}
There exist two numerical constants $c_1, c_2$ and an estimator $\hat{G}$ such that the following holds. If $w^2_{J^*}\geq c_1\gamma^2\pa{\sqrt{K\log(np)}+\log(np)}$ and $\Delta^2\geq c \gamma^{5/2}\left[\sqrt{\frac{sK}{n}\left[\log(n)\right]}+ \log(n)\right]$, then, with probability at least $1-4/n^2$, we have $\hat{G}=G^*$.
\end{prop}

The rationale underlying the algorithm of Proposition \ref{prop:upperboundsparseIT} is to jointly select the column set $\hat J$ and the partition $\hat G$ by exactly minimizing some variant of the Kmeans criterion. Such a minimization requires to
 jointly scan over all the columns and partitions. In particular, the algorithm cannot be computed in poly-time.
  
Let us interpret the condition $w^2_{J^*}\geq_{\log} \sqrt{K}$ appearing in Proposition~\ref{prop:upperboundsparseIT}. 
We write $\bar Y^{G^*}\in\R^{K\times p}$ for the matrix obtained by averaging the rows within a same cluster: $\bar Y^{G^*}_{kj}=\text{average}\ac{Y_{ij}:i\in G^*_{k}}$. We observe that, in the balanced case where $|G^*_{k}|=n/K$ for all $k$,  we have
$\E\cro{\|\bar Y^{G^*}_{:j}\|^2} = \|\mu_{:j}\|^2+K^2\sigma^2/n$ with standard deviation sdev$\cro{\|\bar Y^{G^*}_{:j}\|^2}=K^{3/2}\sigma^2/n$. Hence,  when knowing $G^*$,
 it is possible to recover the active columns $J^*$
 as soon as
\begin{equation}\label{eq:detection-G-known}
\min_{j\in J^*}  \|\mu_{:j}\|^2 \geq_{\log} {K^{3/2}\sigma^2 \over n}\ ,
\end{equation}
by selecting the $s$ columns of $\bar Y^{G^*}$ with largest $\ell^2$-norm.
Since $\omega^2_{J^*} \asymp {n\over K\sigma^2} \min_{j\in J^*}  \|\mu_{:j}\|^2$,
Condition \eqref{eq:detection-G-known} is equivalent to  $w^2_{J^*}\geq_{\log} \sqrt{K}$. Hence, the first condition of Proposition~\ref{prop:upperboundsparseIT} corresponds to the condition for recovering $J^*$ when the partition $G^*$ is known \emph{beforehand}. As for the second condition 
$\Delta^2 \geq_{\log} 1+\sqrt{sK/n}$, 
it corresponds to the optimal condition for recovering $G^*$ when $J^*$ is known, by applying exact Kmeans on the matrix $Y_{:J^*}$, where we have only kept the active columns. 
Hence, non poly-time algorithms can fully leverage the sparse-clustering set-up by only requiring the minimal column signal for selecting the active columns when the clustering $G^*$ is known beforehand, in addition to the minimal separation for clustering when the active columns $J^*$ are known beforehand. This situation is in contrast with the poly-time algorithms, which require the minimal column signal for selecting the active columns with no clustering information.

Under the homogeneity assumption \eqref{eq:assump2},  Lemma \ref{lem:homogeneity} ensures that $\|\mu_{:j}\|^2  \gtrsim n\Delta^2 \sigma^2/s$ for all $j\in J^*$, so the condition $\Delta^2 \geq_{\log} s\sqrt{K} / n$ ensures \eqref{eq:detection-G-known}. Combining this with Proposition \ref{prop:upperboundsparseIT} leads to next corollary. 

\begin{cor}\label{cor:ITsparse}
Assume that $X$ satisfies both the $\eta$-homogeneity assumption \eqref{eq:assump2} and the balancedness condition  \eqref{eq:balanced} with $\eta, \gamma\leq \text{poly-log}(np)$. Then, if $$\Delta^2\geq_{\log} 1+\sqrt{\frac{sK}{n}}+\frac{s\sqrt{K}}{n}\enspace,$$ we have $\hat{G}=G^*$ with probability $1-4/n^2$.
\end{cor}

Combining this corollary with the bounds of exact Kmeans in dimension $p$, we deduce that it is possible to recover $G^*$ as soon as 
\[
\Delta^2\geq_{\log} 1+\min\left[\sqrt{\frac{sK}{n}}+\frac{s\sqrt{K}}{n}, \sqrt{\frac{pK}{n}}\right]\ . 
\]
In comparison to the LD lower bound~\eqref{eq:sparse-barrier-simple}, we see that, depending on the regimes, the statistical-computational gap is possibly as large as factor $\sqrt{n/K}$ or a factor $\sqrt{K}$.

\section{Biclustering}\label{sec:biclustering}

We now turn our attention to the problem of biclustering, where both rows and columns are structured. 
Our goal is to understand if and how the clustering structure on the columns can help for recovering the clustering structure on rows, both statistically and in poly-time.

\paragraph*{Set-up} In the biclustering model, we observe a matrix $Y\in \R^{p\times n}$ generated as follows. There exists two
unknown partitions: $G^*=\ac{G^*_1,\ldots, G^*_K}$, partition of $[n]$, and $H^*=\ac{H_1^*,\ldots, H_L^*}$, partition of $[p]$. 
Then, for some unknown matrix $\mu\in\R^{K\times L}$, and unknown $\sigma>0$, the entries $Y_{ij}$ are independent  with distribution 
\[Y_{ij} \sim \mathcal{N}(\mu_{kl},\sigma^2),\quad \text{for}\ (i,j)\in G^*_k\times H^*_l.\]
We assume in the following that both $G^*$ and $H^*$ fulfill the balancedness condition~\eqref{eq:balanced}.
We observe that under the balancedness condition~\eqref{eq:balanced}, for $i\in G^*_{k}$ and $i'\in G^*_{k'}$ we have
$${\|X_{i:}-X_{i':}\|^2} = \sum_{l=1}^L |H^*_{l}| (\mu_{kl}-\mu_{k'l})^2 \asymp {p\over L}\ \|\mu_{k:}-\mu_{k':}\|^2.$$
Hence, we introduce
$$\Delta^2_r={p\over L}\min_{k\neq k'\in [K]}\frac{\|\mu_{k:}-\mu_{k':}\|^2}{2\sigma^2}\quad \text{and}\quad \Delta^2_c={n\over K}\min_{l\neq l'\in [L]}\frac{\|\mu_{:l}-\mu_{:l'}\|^2}{2\sigma^2}\enspace,$$
which represents, up to a constant factor, the minimum separation relative to the rows of $X$, and the minimum separation relative to the columns of $X$, respectively.
The biclustering model being symmetric, we focus on the problem of recovering $G^*$. We investigate the minimal separation 
$\Delta^2_{r}$ required for recovering $G^*$, and how it depends on $\Delta^2_{c}$. Our contributions are
\begin{enumerate}
\item Implementing the technique of Theorem \ref{thm:LTC}, we provide a LD lower bound for the biclustering problem, unveiling the following phenomenon. When $\Delta_{c}^2$ is below the minimal threshold $\Delta_c^2 \leq_{\log} 1+\min\pa{\sqrt{n}, \sqrt{{nK^2}/{p}}}$ for  poly-time clustering, then the clustering of the rows is as hard as when there is no column structure,  and the separation~\eqref{eq:computationalclusteringsharp+informational}  is required on $\Delta^2_{r}$ for recovering $G^*$. On the contrary, when $\Delta_{c}^2$ is above the minimal threshold for  poly-time clustering, then the column structure can be leveraged to reduce the dimension from $p$ to $L$,  and only the separation  $\Delta^2_{r}\geq_{\log} 1+\min\pa{\sqrt{L}, \sqrt{{LK^2}/{n}}}$ is required  for recovering $G^*$. In this last case,  recovery of $G^*$ is possible when $\Delta^2_{r}\geq_{\log} 1+\min\pa{\sqrt{L}, \sqrt{{LK^2}/{n}}}$ by  (i) clustering the columns, (ii) averaging all the columns within a same group, reducing the number of columns to $L$, and (iii) applying a poly-time row clustering on the new $n\times L$ matrix.
\item We prove that non poly-time algorithms can much better leverage the biclustering structure by merely requiring the separations
$$\Delta^2_{r}\geq_{\log} 1+ \sqrt{KL \over n}\quad \text{and}\quad \Delta^2_{c}\geq_{\log} 1+\sqrt{KL \over p},$$
or $\Delta^2_{r}\geq_{\log} 1+ \sqrt{Kp / n}$.
The separation $\Delta_{r}\geq_{\log} \sqrt{KL / n}$ corresponds to the statistical separation for clustering  the $n$  rows in dimension $L$, while the separation $\Delta^2_{c}\geq_{\log} \sqrt{KL / p}$ corresponds to the statistical separation for clustering the $p$ columns in dimension $K$. The separation $\Delta^2_{c}\geq_{\log} 1+\sqrt{KL / p}$, required on the columns to benefit from the dimension reduction phenomenon, i.e. to benefit from the reduced requirement  $\Delta^2_{r}\geq_{\log} 1+ \sqrt{KL / n}$ on the rows, is much smaller than the separation  $\Delta_c^2 \geq_{\log} 1+\min(\sqrt{n}, \sqrt{{nK^2}/{p}})$ required by poly-time algorithms. 
This separation $\Delta^2_{c}\geq_{\log} \sqrt{KL / p}$ corresponds to the separation needed to cluster the columns (recover $H^*$) when $G^*$ is known. Indeed, clustering the columns at this level of separation can be obtained when $G^*$ is known by (i) averaging the rows along the partition $G^*$, reducing the row dimension from $n$ to $K$, and (ii) clustering the columns of the transformed $K \times p$ matrix. Interestingly, this separation needed to recover $H^*$ when $G^*$ is known then allows to recover $G^*$ with the same separation $\Delta^2_{r}\geq_{\log} \sqrt{KL / n}$ as if the partition $H^*$ was known. Hence, non poly-time algorithms can fully leverage the biclustering structure.
\end{enumerate}

\subsection{LD lower-bound for biclustering}\label{sec:LD:biclustering}
Let us introduce the prior distribution under which we derive our LD bound. As for sparse clustering, we use a symmetrization of the means in order to derive tight lower-bound, and for convenience we consider a setting with $2K$ row-clusters and $2L$ column-clusters.  
\begin{definition}\label{def:prior_bi}
The signal matrix $X\in\R^{n\times p}$ is generated as follows. We sample independently\\ \smallskip
- $k_1^*,\ldots,k^*_n$ i.i.d. with uniform distribution on $[K]$,\\ \smallskip
- $l_1^*,\ldots, l_p^*$  i.i.d. with uniform distribution on $[L]$,\\ \smallskip
- $\eps^r_1,\ldots,\eps^r_n$  i.i.d. with uniform distribution on $\ac{-1,1}$,\\ \smallskip
- $\eps^c_1,\ldots,\eps^c_p$  i.i.d. with uniform distribution on $\ac{-1,1}$,\\ \smallskip
- $\pa{\nu_{k,l}}_{k\in [K], l\in [L]}$ i.i.d. with $\mathcal{N}\pa{0,\lambda^2}$ distribution, with $\lambda>0$.\\ \smallskip
Then, we set 
\begin{equation}\label{eq:prior-bi}
X_{ij}=\eps^r_i\eps^c_j\nu_{k^*_i,  l^*_j}\ .
\end{equation}
\end{definition}

Under the prior~\eqref{eq:prior-bi},  the partition $G^*$ is defined by
$G^*_{k}=\ac{i: k^*_{i}=k,\ \text{and}\ \eps^r_{i}=1}$ for $k=1,\ldots,K$, and  $G^*_{k}=\ac{i: k^*_{i}=k-K,\ \text{and}\ \eps^r_{i}=-1}$ for $k=K+1,\ldots,2K$; while the partition $H^*$ is defined by
$H^*_{l}=\{j: l^*_{j}=l,\ \text{and}\ \eps^c_{j}=1\}$ for $l=1,\ldots,L$, and  $H^*_{l}=\{j: l^*_{j}=l-L,\ \text{and}\ \eps^c_{j}=-1\}$ for $l=L+1,\ldots,2L$.
Furthermore, under the assumption that $L\geq \log(K)$ and $K\geq \log(L)$, we have
$$\Delta_{r}^2= {\lambda^2p\over  \sigma^2}\pa{1+O\pa{\sqrt{\log(K)\over L}}}\quad \text{and}\quad \Delta_{c}^2= {\lambda^2 n\over \sigma^2} \pa{1+O\pa{\sqrt{\log(L)\over K}}}.$$

The next result provides two LD lower-bounds for the problem of clustering under the prior distribution~\eqref{eq:prior-bi}.

\begin{thm}\label{thm:lowdegreebi}
Let $D\in \N$ and suppose that $\zeta:=\frac{\lambda^4}{\sigma^4} D^{8}\max\pa{p,n, \frac{pn}{K^2}, \frac{pn}{L^2}}<1.$ Then, under the prior distribution of Definition \ref{def:prior_bi}, we have
\begin{equation}\label{eq:LB_MMSE_D_SC1}
 MMSE_{\leq D}\geq \frac{1}{K}-\frac{1}{K^2}\pa{1+\frac{\zeta}{\pa{1-\sqrt{\zeta}}^3}}.
\end{equation}
Moreover, if $\zeta':=\frac{\lambda^4}{\sigma^4}D^{10}\frac{5p^2}{L}\max\pa{1,\frac{n}{K^2}}<1$, then, under the prior distribution of Definition \ref{def:prior_bi}, we have
\begin{equation}\label{eq:LB_MMSE_D_SC2}
MMSE_{\leq D}\geq \pa{1-L\exp\pa{-\frac{5p}{2L}\log(5)}}\pa{\frac{1}{K}-\frac{1}{K^2}\frac{\sqrt{\zeta'}}{(1-\sqrt{\zeta'})^2}}\enspace.
\end{equation}
\end{thm}
We refer to Appendix~\ref{prf:lowdegreebi} for a proof of this result.
Instantiating Theorem~\ref{thm:lowdegreebi} for the degree $D=\log(n)^{1+\eta}$, we get that $MMSE_{\leq D}=\pa{\frac{1}{K}-\frac{1}{K^2}}\pa{1+o(1)}$
\begin{align}
&\text{if}\quad  n\lambda^2  \leq_{\log} \sigma^2\min\pa{ \sqrt{{n}}, \sqrt{\frac{L^2n}{p}}} 
\quad \text{and}\quad   p\lambda^2  \leq_{\log} \sigma^2\min\pa{\sqrt{{p}}, \sqrt{\frac{K^2p}{n}}} , \label{eq:bi1}\\ 
\text{or,}\ &\text{if}\quad  p\gg L\log L\quad  \text{and}\quad p\lambda^2\leq_{\log} \sigma^2 \min\pa{\sqrt{L},\sqrt{\frac{LK^2}{n}}}\enspace . \label{eq:bi2}
\end{align}
Let us interpret these two conditions.

\underline{The first Condition \eqref{eq:bi1}} can be reformulated as
$$\Delta_{c}^2  \leq_{\log} \min\pa{ \sqrt{{n}}, \sqrt{\frac{L^2n}{p}}}
\quad \text{and}\quad  \Delta_{r}^2  \leq_{\log} \min\pa{\sqrt{{p}}, \sqrt{\frac{K^2p}{n}}}.$$
We recognize in this condition the Threshold \eqref{eq:computationalclusteringsharp} for clustering in poly-time the $p$ columns in dimension $n$, and  for clustering in poly-time the $n$ rows in dimension $p$. In particular, this result unravels that, below the  threshold for  poly-time clustering of the columns $\Delta_{c}^2  \leq_{\log} \min\pa{ \sqrt{{n}}, \sqrt{{L^2n}/{p}}}$, the poly-time clustering of the rows is as hard as if there was no column structure. In other words, poly-time algorithms can leverage the biclustering structure only when either the columns or the rows have a separation larger than the separation~\eqref{eq:computationalclusteringsharp} for simple clustering.

\underline{The second Condition \eqref{eq:bi2}} shows that, when $\Delta_{c}^2  \geq_{\log} \min\pa{ \sqrt{{n}}, \sqrt{{L^2n}/{p}}}$, poly-time clustering of the rows can be impossible when $ p\gg L\log L$ and
$$\Delta_{r}^2\leq_{\log}  \min\pa{\sqrt{L},\sqrt{\frac{LK^2}{n}}}.$$
We recognize here the Threshold \eqref{eq:computationalclusteringsharp} for clustering in poly-time $n$ points  in dimension $L$, into $K$ groups. This means that when columns can be clustered into $L$ groups, row clustering is as hard as clustering in dimension $L$.
This threshold can be simply understood as follows. Let us define for $(i,l)\in [n]\times[L]$,
\begin{equation}\label{eq:YH}
\bar Y^{H^*}_{il}={1\over |H^*_{l}|}\sum_{j\in H^*_{l}} Y_{ij}=\mu_{k_{i}l}+\bar E^{H^*}_{il}.
\end{equation}
We observe that for  $j\in H^*_{l}$, we have
$Y_{ij}=\bar Y^{H^*}_{il}+\tilde E_{ij}$, where  $\tilde E_{ij}=E_{ij}-\bar E^{H^*}_{il}$ is independent of $\bar Y^{H^*}_{il}$, with a distribution independent of $G^*$ and $\mu$.  
Hence, clustering the rows of $Y$ is at least as hard as clustering the rows of $\bar Y^{H^*}$. 
Conversely, when $\Delta_{c}^2  \geq_{\log} \min\pa{ \sqrt{{n}}, \sqrt{{L^2n}/{p}}}$, the column structure $H^*$ can be recovered  in poly-time, with high-probability, in almost all regimes of $L,n,p$, so we can compute $\bar Y^{H^*}$ --see Section~\ref{sec:clustering}. Hence, when $\Delta_{c}^2  \geq_{\log} \min\pa{ \sqrt{{n}}, \sqrt{{L^2n}/{p}}}$,  
 clustering the rows of $Y$ is not harder than clustering the rows of $\bar Y^{H^*}$. 
Since 
var$(\bar E^{H^*}_{il})\asymp L\sigma^2/p$,  the row separation for $\bar Y^{H^*}$ is
$$\Delta^2 \asymp \min_{k\neq k'} {\|\mu_{k:}-\mu_{k':}\|^2\over 2 L\sigma^2/p}=\Delta_{r}^2$$
and the Condition \eqref{eq:bi2} corresponds to the Threshold \eqref{eq:computationalclusteringsharp} for clustering the rows of 
$\bar Y^{H^*}$ in poly-time.

\subsection{Upper bound on the statistical rate for biclustering}

To get a complete picture of the biclustering problem, let us now investigate the minimal separations $\Delta_{r}$ and $\Delta_{c}$ above which row clustering is possible. 
First of all, according to \eqref{eq:informationalclustering}, row clustering with exact Kmeans is always possible when $\Delta^2_{r}\geq_{\log} 1+ \sqrt{Kp / n}$, regardless of $\Delta_{c}^2$. Let us now examine how non poly-time algorithms can leverage the biclustering structure.

Let us consider the bi-Kmeans estimator (which is the MLE in the Gaussian setting)
\begin{equation}\label{ed:bi-Kmeans}
\big(\hat G,\hat H\big) \in \mathop{\text{argmin}}_{G,H} \underset{l\in [L]}{\sum_{k\in [K]}}\ \underset{j\in H_{l}}{\sum_{i\in G_{k}}} \pa{Y_{ij}-\bar Y_{kl}^{G\times H}}^2,
\end{equation}
where $\bar Y_{kl}^{G\times H}$ is the average value of $Y_{ij}$ over $G_{k}\times H_{l}$. In some way, this least-square estimator shares some similarities with that in~\cite{gao2016optimal}, although Gao et al.~\cite{gao2016optimal} focus their attention on the reconstruction of $X$ in Frobenius norm.
The next proposition, proved in Section \ref{prf:IT_Biclustering}, provides a condition under which bi-Kmeans is able to recover the partitions $G^*$ and $H^*$.

\begin{prop}\label{prp:IT_Biclusterling}
There exists numerical constants $c$, $c'$, $c''$ such that the following holds for all $(n\vee p)\geq c'$ and for all $\gamma>1$. 
Assume that the hidden partitions $G^*$ and $H^*$ fulfill the balancedness condition \eqref{eq:balanced}.
Then, as long as we have 
\begin{eqnarray}\label{eq:condition_sepation_line_IT}
\Delta^2_r &\geq c     \gamma^{5/2} \left[\sqrt{\frac{KL\log(n\vee p)}{n}}+ \log(n\vee p)\right],
\\  \label{eq:condition_sepation_column _IT}
\text{and}\quad \Delta^2_c &\geq c     \gamma^{5/2} \left[\sqrt{\frac{KL\log(n\vee p)}{p}}+ \log(n\vee p)\right]
\ , 
\end{eqnarray}
we have $\hat G=G^*$ and $\hat H=H^*$, 
with probability higher than $1-c''/(n\vee p)^2$.
\end{prop}

Proposition~\ref{prp:IT_Biclusterling} ensures that 
non poly-time algorithms can recover the row (or column) partition $G^*$ as soon as
$$\Delta^2_{r}\geq_{\log} 1+ \sqrt{KL \over n}\quad \text{and}\quad \Delta^2_{c}\geq_{\log} 1+\sqrt{KL \over p}.
$$
The separation $\Delta_{r}\geq_{\log} \sqrt{KL / n}$ corresponds to the statistical separation for clustering   $n$  rows in dimension $L$, while the separation $\Delta^2_{c}\geq_{\log} \sqrt{KL / p}$ corresponds to the statistical separation for clustering  $p$ columns in dimension $K$. 
The separation $\Delta^2_{c}\geq_{\log} \sqrt{KL / p}$ can be interpreted as
 the separation needed to cluster the columns (recover $H^*$) when $G^*$ is known, by computing
$$\bar Y^{G^*}_{kj}={1\over |G^*_{k}|}\sum_{i\in G^*_{k}} Y_{ij},\quad \text{for}\  (k,j)\in [K]\times[p],$$
 and then clustering the columns of $\bar Y^{G^*}$.  Indeed, the variance of the entries of $\bar Y^{G^*}$ is around $K\sigma^2/n$, and  
 the column separation for $\bar Y^{G^*}$ is then
   $$\Delta^2 \asymp \min_{l\neq l'} {\|\mu_{:l}-\mu_{:l'}\|^2\over 2 K\sigma^2/n}=\Delta_{c}^2\enspace .$$
 Strikingly, above the column separation $\Delta^2_{c}\geq_{\log} \sqrt{KL / p}$ needed to recover $H^*$ when $G^*$ is known, we are able to recover $G^*$ with the separation $\Delta^2_{r}\geq_{\log} \sqrt{KL / n}$ required when the partition $H^*$ is known. Hence, only a $K$-dimensional column separation condition  is needed to benefit from the $L$-dimensional row separation condition $\Delta^2_{r}\geq_{\log} 1+ \sqrt{KL / n}$ for successful clustering.
This feature is in contrast with poly-time algorithms, where the $n$-dimensional column  separation $\Delta^2_{c}\geq_{\log} 1+ \min\pa{\sqrt{n},\sqrt{L^2n / p}}$ is required for benefiting from the $L$-dimensional row separation condition $\Delta^2_{r}\geq_{\log} 1+ \min\pa{ \sqrt{L}, \sqrt{K^2 L/ n}}$.
Hence, our analysis unravels a much better ability of non poly-time algorithms to leverage the biclustering structure, compared to poly-time algorithms.

To complete the picture, we underline that the condition $\Delta^2_{r}\geq_{\log} 1+ \sqrt{KL /n}$ is minimal for recovering $G^*$. Indeed, we can argue as in Section~\ref{sec:LD:biclustering}, and consider, for $j\in H^*_{l}$ the decomposition $Y_{ij}=\bar Y^{H^*}_{il}+\tilde E_{ij}$, where $\bar Y^{H^*}$ is defined in \eqref{eq:YH}
and  $\tilde E_{ij}=E_{ij}-\bar E^{H^*}_{il}$ is independent of $\bar Y^{H^*}$, with a distribution independent of $G^*$ and $\mu$. The problem of clustering  the rows of $Y$ is at least as hard as the problem of clustering the rows of $\bar Y^{H^*}$,  
and the condition $\Delta^2_{r}\geq_{\log} 1+ \sqrt{KL /n}$ corresponds to the threshold above which row clustering of 
$\bar Y^{H^*}$ is statistically possible. Hence, row-clustering of $Y$ is, in general, impossible when $\Delta^2_{r}\leq_{\log} 1+ \sqrt{KL /n}$.
To sum-up, the minimal condition for recovering $G^*$ without computational constraints is 
$$\Delta^2_{r}\geq_{\log} 1+ \sqrt{Kp \over n}\,,\quad \text{or}\quad\quad \Delta^2_{r}\geq_{\log} 1+ \sqrt{KL \over n}\quad \text{and}\quad \Delta^2_{c}\geq_{\log} 1+\sqrt{KL \over p}.
$$

\section{Discussion and Open Problems}\label{sec:discussion}

The technique developed in Section~\ref{sec:technique} enables to derive computational lower bounds for three important clustering problems, matching the upper-bounds for poly-time algorithms in most of the regimes of parameters $n,p,K$. It is likely that this technique can also be successfully applied to other problems like tensor PCA \cite{TensorPCA24}, semi-supervised sparse clustering~\cite{azar2024semi}, ...

One major limitation of our technique  is that it relies on the lower-bound on the $MMSE_{\leq D}$ of \cite{SchrammWein22} in terms of a sum of cumulants. This lower bound may not be tight in some regimes, due to a Jensen inequality at the heart of the analysis of \cite{SchrammWein22}, see e.g. the discussion in Appendix~\ref{sec:discussion:prior:sparse}.  
In particular, for clustering Gaussian mixture, in the regime where both $p\leq {n}/{K}$ and $K^2\lesssim n \leq \text{poly}(K)$, our LD bound \eqref{eq:computationalclusteringsharp+informational} and our poly-time upper bound \eqref{eq:suboptimal-clustering} do not match. We suspect that, in this regime, both the LD and the poly-time upper-bounds are suboptimal. Some other proof techniques are probably needed to handle this very challenging regime, and the minimal separation for poly-time algorithms in this regime remains an open problem.

\begin{funding}
 This work has been partially supported by ANR-21-CE23-0035 (ASCAI, ANR).
\end{funding}

%%%%%%%%%%%%%%%%%%%%%%%%%%%%%%%%%%%%%%%%%%%%%%%%%%%%%%%%%%%%%
%%                  The Bibliography                       %%
%%                                                         %%
%%  imsart-???.bst  will be used to                        %%
%%  create a .BBL file for submission.                     %%
%%                                                         %%
%%  Note that the displayed Bibliography will not          %%
%%  necessarily be rendered by Latex exactly as specified  %%
%%  in the online Instructions for Authors.                %%
%%                                                         %%
%%  MR numbers will be added by VTeX.                      %%
%%                                                         %%
%%  Use \cite{...} to cite references in text.             %%
%%                                                         %%
%%%%%%%%%%%%%%%%%%%%%%%%%%%%%%%%%%%%%%%%%%%%%%%%%%%%%%%%%%%%%

%% if your bibliography is in bibtex format, uncomment commands:
\bibliographystyle{imsart-number} % Style BST file (imsart-number.bst or imsart-nameyear.bst)
\bibliography{biblio.bib}       % Bibliography file (usually '*.bib')

\newpage

\appendix

\section{Technical discussions}\label{sec:technical:discussion}

All the results stated in this section are proved in Appendix~\ref{sec:proofs:technicaldiscussion}

\subsection{Sparse clustering: discussion of  $w^2_{J^*}$ and of $\eta$-homogeneity condition}\label{sec:discussion:w}

In Sections \ref{sec:polytimesparse} and~\ref{sec:polytimesparse}, we introduced a condition on $w_{J^*}^2$ the minimum $l_2$-norm of the active columns of signal matrix $X$, in order to analyze our sparse clustering clustering procedures. As explained in that section, our condition in $\omega_{J^*}$ matches the LD lower bound when an $\eta$-homogeneity condition is satisfied (Assumption~\ref{ass:homogeneity}).

We now shortly discuss how one could extend our sparse clustering procedures to bypass the condition on $\omega_{J^*}$ or equivalently Assumption~\ref{ass:homogeneity}. First, observe that, in general, it is impossible to recover all the active  columns $J^*$ with high probability without any assumption on $w_J^*$. Nevertheless, as long as $\Delta^2$ is large enough, our feature selection procedure selects, with high probability, a subset $\Bar{J}$ of $J^*$ which separates well a large fraction of the $\mu_k$'s.

\begin{lem}\label{lem:whyomega2}
There exist a subset $\Bar{J}\subseteq J^*$ and a subset $\mathcal{K}\subset [K]$ with $|\mathcal{K}|\geq \frac{8K}{10}$ satisfying:
\begin{enumerate}
\item For all $j\in \Bar{J}$, $\sum_{k\in [K]}|G_k^*|\pa{\mu_k}_j^2\geq \frac{n\Delta^2}{80s\gamma}\sigma^2$;
\item For all $k\in \mathcal{K}$ and $l\in [K]$, $\|(\mu_k)_{\Bar{J}}-(\mu_l)_{\Bar{J}}\|^2\geq \frac{1}{8}\Delta^2\sigma^2$.
\end{enumerate}
\end{lem}
The first part of the above lemma ensures that the square norm of $X_{:j}$ is large, so that $j\in\Bar{J}$ can be detected by looking at the norm $Y_{:j}$. The second part of Lemma ~\ref{lem:whyomega2}, ensures that reducing our attention to $\Bar{J}$ allows to separate well most of the groups --but not all of them.

Assume that $\Delta^2\geq_{\log,\gamma}1+{s/\sqrt{n}}+\min(\sqrt{s},\sqrt{{sK^2}/{n}})$, which corresponds to the LD lower bound. In principle, we could then use a hierarchical scheme. We would first build independent copies of $Y$ --to the price of slightly lowering the seperation $\Delta$. Then, we could first select a subset $\hat{J}$ of size $s$  of the columns with largest empirical norm. With large probabilily, it turns out that $\hat{J}$ will contain the subset $\overline{J}$ of Lemma~\ref{lem:whyomega2}.  Considering the second independent sample and focusing on the columns $\hat{J}$, we are back to a Gaussian mixture model with $K$ groups, $|\mathcal{K}|$ of which, are separated by, up to constants, at least $\Delta^2$ from all the other groups. Hence, if we could adapt Proposition~\ref{prop:upper-boundclustering} to the case of a Gaussian mixture model, where a small proportion of the groups are not well separated, then we could distinguish the well-separated groups. Applying recursively the scheme to the subgroups and applying Lemma~\ref{lem:whyomega2} to these subgroups we would be able to recover features that allow to distinguish new groups. 

However, the technical hurdle behind this approach is that Proposition~\ref{prop:upper-boundclustering} is only valid for Gaussian mixture models such that all groups are well-separated. 
In fact, Proposition~\ref{prop:upper-boundclustering} is based on a combination of clustering techniques: spectral projections, hierarchical clustering, the high-order tensor projection method of Li and Liu~\cite{LiuLi2022}, and a SDP version of Kmeans. It is not difficult to show that both the spectral projection and hierarchical steps can be easily adapted to this setting. As for the high-order tensor projection method, we really suspect that it will be able to distinguish the groups $G^*_k$ with $k\in \mathcal{K}$, but we did not check all the details. However, we do not know how to adapt  the SDP analysis of~\cite{giraud2019partial} to this setting. As we mainly focus this manuscript on LD lower bounds, we do not pursue in this direction.

\paragraph*{Case with $K\leq 4$ clusters} When $K\leq 4$, Lemma~\ref{lem:whyomega2} straightforwardly entails that $\overline{J}$ contain enough variables so that the restriction of $X$ to $\overline{J}$ still ensures a minimum separation  at least $\Delta^2/8$ between all the clusters. Since, with high probability, $\hat{J}$ contains $J^*$, this implies that, in the second step of our polynomial-time sparse clustering procedure, we will apply a clustering procedure to a dimension $s$ Gaussian model with $K$ groups with separation $\Delta^2/8$. As a consequence, Corollary~\ref{cor:polytimesparse} is valid for $K\leq 4$ without requiring the $\eta$-homogeneity condition.

\subsection{Impossibility of partition reconstruction}\label{subsec:impossibility_reconstruction}
Given a partition $G$, define the (unnormalized) partnership matrix $M^G\in \{0,1\}^{n\times n}$ by $M^{G}_{ij}=\mathbf{1}_{k^*_i=k^*_j}$. We write $M^*$ for $M^{G^*}$.
To simplify the discussion, we consider an asymptotic setting where $(K,n)$ (and possibly $p$) go to infinity, with $K=o(n)$.
\begin{lem}\label{lem:MG-random}
If $\E[\|M^*\|^2_{F}]=n^2K^{-1}(1+o(1))$ and 
$$MMSE_{poly}:= \inf_{\hat M\ poly-time} {1\over n(n-1)}\E\cro{\|\hat M-M^*\|^2_{F}} = {1\over K}(1+o(1)),$$
then 
$$\inf_{\hat G\ poly-time} {1\over n(n-1)}\E\cro{\| M^*-M^{\hat G}\|^2_{F}}\geq {2\over K}(1+o(1)).$$
\end{lem}
Since the error $2K^{-1}(1+o(1))$ corresponds to the error rate of a uniform random partition, it means that, if we cannot estimate $M^*$ better than by its mean value, then we cannot estimate $\hat G$ better than at random in terms of the metric $\E\cro{\| M^*-M^{\hat G}\|^2_{F}}$.

Let us consider the  prior distribution of Definition~\ref{def:priorclustering} with $K\ll n/\log(n)$ so that
$\E[\|M^*\|^2_{F}]=n^2K^{-1}(1+o(1))$ and $G^*$ is $\gamma$-balanced with $\gamma= 1+o(1)$.

\begin{prop}\label{prp:Impossibility:reconstruction}
Assume that $MMSE_{poly}=K^{-1}(1+o(1))$.
Then, all polynomial-time estimators of the partition that are $\gamma$-balanced satisfy
$\mathbb{E}[err(\hat{G},G^*)]= 1+o(1)$.
\end{prop}

In other words, it is impossible to reconstruct in polynomial-time a (balanced) partition $\hat{G}$  better than random guessing in the regime where $MMSE_{poly}=K^{-1}(1+o(1))$.

\subsection{On the prior distribution for sparse clustering}\label{sec:discussion:prior:sparse}

For the LD bound for sparse clustering, we introduce some symmetry in the prior of Definition~\ref{def:prior_sparse}. 
Let us explain why we did not consider a closer variant of the prior of Definition \ref{def:priorclustering}, by simply keeping the signal  on $s$ columns randomly chosen.
We underline in this appendix that, under this prior, we cannot derive the desired lower bound on the low-degree MMSE. Although this could come from the fact that this (non-symmmetric) prior is not suited for establishing the hardness of spase clustering, we suspect that this behaviour is rather due to the Jensen bound at the heart of the general method of~\cite{SchrammWein22}. 

Let us first define a (non-symmetric) prior for sparse clustering, which is more in line to Definition~\ref{def:priorclustering}.  We assume in this subsection that $\sigma=1$. 
\begin{definition}\label{def:prior_sparse_simple}
The signal matrix  $X\in \R^{n\times p}$ is generated as follows. We sample independently:\\ \smallskip
- $k^*_1,\ldots ,k^*_n$ independent with uniform distribution on $[K]$,\\ \smallskip
- $z_1,\ldots,z_p$ independent, with Bernoulli  distribution $\mathcal{B}(\rho)$, where $\rho={\Bar{s}}/{p}$,\\ \smallskip
- $\nu_{k,j}$, for $k,j\in [K]\times [p]$, independent, with $\nu_{k,j}\sim\mathcal{N}\pa{0,\lambda^2}$, where $\lambda^2=\Bar{\Delta}^2\sigma^2/{\rho p}$.\\ \smallskip
Then, we set 
\begin{equation}\label{eq:prior-sparse-simple}
\X_{ij}= z_j\nu_{k^*_i,j}\ .
\end{equation}
\end{definition}

In our signal plus noise Gaussian model $Y=X+E$, it is straightforward to retrace the (strict) inequalities in the general bound of~\cite{SchrammWein22}. Indeed, given a polynomial $f(Y)$, the crux of~\cite{SchrammWein22} is to apply Jensen inequality to lower bound the second moments of $f(Y)$, that is 
\[
\mathbb{E}[f(Y)^2]\geq \mathbb{E}_Z\left[\left(\mathbb{E}_X[f(X+Z)]\right)^2\right]\ , 
\]
where $\mathbb{E}_X[.]$ ($\mathbb{E}_Z[.]$) stands for the expectation with respect to $X$ (resp. $Z$). With this bound in mind, we define
\begin{equation}\label{def:cor:SW}
    \widetilde{corr}^{(SW)}_{\leq D}:=\underset{\E\pa{f^{2}(Y)}\neq 0}{\sup_{f\in\R_{D}[Y]}}\frac{\E[f(Y)x(Z)]}{\sqrt{\mathbb{E}_Z\left[\left(\mathbb{E}_X[f(X+Z)]\right)^2\right]}}\enspace , 
\end{equation}
which is an upper bound of $corr_{\leq D}$. We readily deduce from the proof of Theorem 2.2 in~\cite{SchrammWein22} that this modified degree-$D$ maximum correlation satisfies the equality
    \begin{equation}\label{eq:corr_D_SW}
    \left(\widetilde{corr}^{(SW)}_{\leq D}\right)^2=\underset{|\alpha|\leq D}{\sum_{\alpha\in\N^{n\times p}}}\frac{\kappa_{x,\alpha}^{2}}{\alpha!}\enspace . 
    \end{equation}

The following proposition, provides a lower bound on $ \left(\widetilde{corr}^{(SW)}_{\leq D}\right)^2$.     
\begin{prop}\label{prop:sparse:bad:prior}
The exist two numerical constants $c$ and $c'$ such that the following holds. Consider any even $D$ degree such 
that $n\geq 4D$ and $\rho\leq [12 (D/2)! 2^{D/2}]^{-1}$. Then, we have 
\[
 \left(\widetilde{corr}^{(SW)}_{\leq D}\right)^2\geq c'e^{-cD\log(D) }p \frac{n^{D-2}}{K^D} \lambda^{2D} \rho^2  \ , 
\]
\end{prop}
In light of the definition of $\lambda$, we conclude that $\left(\widetilde{corr}^{(SW)}_{\leq D}\right)^2\geq 1$ 
as soon as
\begin{equation}\label{eq:condition:bad:prior:sparse}
\Bar{\Delta}^2\geq c'e^{c\log(D)} \frac{\bar{s}K}{n}\cdot  \left(\frac{n^2p}{\bar{s}^2}\right)^{1/D}\ .
\end{equation}
In particular, it is therefore not possible to provide a non-trivial lower bound on $MMSE_{\leq D}$ using the approach of~\cite{SchrammWein22} as long as the above condition is satisfied. 

Consider for a instance a regime where $K$ is  a constant and $\sqrt{n}\ll s\ll p$. Then, our main result in Section~\ref{sec:sparse} --see~\ref{eq:sparse-barrier-simple}-- entails that clustering is impossible as soon as 
\[
\Bar{\Delta}^2\leq_{\log} \frac{\bar{s}}{\sqrt{n}}\  , 
\]
whereas, equipped with the prior of Definition~\ref{def:prior_sparse_simple}, we would at best be able to prove the condition $\Bar{\Delta}^2\leq_{\log} {\bar{s}}/{n}$, which is looser by a factor $\sqrt{n}$. 
\medskip

\subsection{Extension to binary observations}\label{sec:bernoulli}

When the data $Y\in\ac{0,1}^{n\times p}$ are binary, with $\P\cro{Y_{ij}=1| X}=X_{ij}$ and $1<\tau_{0}\leq X_{ij}\leq \tau_{1}< 1$, \cite{SchrammWein22} provides 
a lower bound on the $MMSE_{\leq D}$ in terms of the cumulants $\cumul(x,X_{\alpha})$
\begin{equation}\label{eq:SW-binary}
MMSE_{\leq D}\geq \E[x^2] - \sum_{\alpha\in \ac{0,1}^{n\times p},\ |\alpha|\leq D} {\cumul(x,X_{\alpha})^2\over (\tau_{0}(1-\tau_{1}))^{|\alpha|}}\ .
\end{equation}
In this case, we cannot choose in the latent model the $\nu_{kj}$ with Gaussian distribution, since it should be bounded a.s. 
Instead, we can consider $\nu$ uniformly distributed on a shifted hypercube. 
We explain below how to handle this setting.

Let us suppose that the signal matrix is  the form $$X_{ij}=m_{ij}+\delta_{ij}(Z)\nu_{\theta_{ij}(Z)}\enspace,$$  with $m_{ij}\in (0,1)$ and the $\nu_{kl}$ taken i.i.d uniformly on $\ac{-\lambda, \lambda}$, instead of being Gaussian. We then have $X_{ij}\in[\tau_{0},\tau_{1}]$ a.s., with $\tau_{0}=\min_{ij} m_{ij}-\lambda$ and $\tau_{1}=\max_{ij} m_{ij}+\lambda$.
Since the $m_{ij}$'s are constant, by multilinearity of the cumulants combined with Lemma~\ref{lem:independentcumulant}, we have that, for any multiset $\alpha\in \N^{n\times p}$,
$$\cumul(x,X^{\alpha})=\cumul(x,(\delta_{ij}(Z)\nu_{\theta_{ij}(Z)})_{(i,j)\in\alpha}).$$
Therefore, without loss of generality, we can assume that $m_{ij}=0$ for all $(i,j)$.

As for the Gaussian prior, we can still apply the Law of Total Cumulance and get, 
$$\kappa_{x,\alpha}=\cumul(x,X_{\alpha})=\sum_{\pi\in \mathcal{P}\pa{\alpha\cup \ac{x}}}\cumul\pa{\cumul\pa{x, X_{\pi_{0}\setminus \ac{x}} |Z}, \cumul\pa{X_{R}|Z}_{R\in \pi\setminus \ac{\pi_{0}}}}\enspace.$$
For $\beta\neq 0$, it is still true that $\cumul\pa{x, \pa{X_{ij}}_{ij\in \beta} |Z}=0$. The difference lies in the expression of $\cumul\pa{X_{\beta}|Z}$ which is 
$$\cumul\pa{X_{\beta}|Z}= c_{|\beta|}\lambda^{|\beta|}\delta(Z)^{\beta}\,\1_{|\beta|\equiv0 [2]} \,\1_{\Omega_{\beta}(Z)}\enspace,$$
where $\delta(Z)^\beta:=\prod_{(i,j)\in\beta} \delta_{ij}(Z)$, 
\begin{equation}\label{def:Omega2}
\Omega_{\beta}(Z):=\big\{\delta_{ij}(Z)\neq 0,\enspace \forall (i,j)\in \beta\big\}\cap \ac{\big|\ac{\theta_{ij}(Z):\enspace (i,j)\in \beta}\big|=1}\enspace,
\end{equation}
and 
$$c_{|\beta|}=\sum_{\pi\in \mathcal{P}([|\beta|])}m(\pi)\1\ac{\forall R\in \pi,\enspace |R|\equiv 0\enspace [2]}\enspace.$$
We recall that with a Gaussian prior, this conditional cumulant was null whenever $|\beta|\neq 2$. For $\alpha\in \N^{n\times p}$ a multiset, we write $\mathcal{P}^{even}(\alpha)$ the set of all partitions $\pi$ of $\alpha$ such that, for all $R\in \pi$, $|R|\equiv 0\enspace [2]$. We end up with the following proposition. 

\begin{prop}\label{prop:LTC2}
For all $\alpha\in\N^{n\times p}$
\begin{equation}\label{eq:LTCgeneral:thm2}
\kappa_{x,\alpha}=\lambda^{|\alpha|}\sum_{\pi\in \mathcal{P}^{even}(\alpha)}\pa{\prod_{s\in |\pi|} c_{|\beta_s(\pi)|}}C_{x, \beta_1(\pi),\ldots, \beta_{|\pi|}(\pi)}\enspace,
\end{equation}
 with $\cro{\beta_{s}(\pi)}_{ij}$ counting the number of copies of $(i,j)$ in $\pi_s$, and where
\begin{align}
C_{x,\beta_{1},\ldots,\beta_{|\pi|}} 
&= \cumul\pa{x,\delta(Z)^{\beta_{1}}\1_{\Omega_{\beta_{1}}(Z)},\ldots,\delta(Z)^{\beta_{l}}\1_{\Omega_{\beta_{|\pi|}}(Z)}},\label{eq:def-C2}
\end{align}
with $\Omega_{\beta}(Z)$  defined in \eqref{def:Omega2}, and $\delta(Z)^\beta:=\prod_{(i,j)\in\beta} \delta_{ij}(Z)$.
\end{prop}

Applying Proposition \ref{prop:LTC2} to the three problems considered, we would obtain the same upper-bounds on $\kappa_{x,\alpha}$ than with a Gaussian Prior, up to some multiplicative constant of the form $|\alpha|^{c|\alpha|}$. The obtained computational barrier would be similar up to some power of $D$. The downside of this being that we would not be able to catch the exact BBP constant when $n\geq poly(K,D)$ for the problems of Clustering and Sparse Clustering. For Biclustering, we would obtain the same result, with a modification of the power of $D$ in the expression of $\zeta$ (see Theorem \ref{thm:lowdegreebi}).

\subsection{Discussion of other frameworks of computational lower-bounds}\label{sec:discussion:computational lower bounds}

LD Polynomials is a popular restricted class of estimators used for understanding computational limits. However, it is not the only class that is used to predict computational barriers; here is an non-exhaustive list of other classes of estimators that are widely used when trying to understand computational limits. 

\paragraph*{Sum-of-Square Hierarchy} The Sum of Square hierarchy is a family of Semi-Definite programs that is used for the task of \textit{certification} \cite{hopkins2018statistical, HopkinsFOCS17,Barak19}. The question to answer is wether SoS can certify the absence of a structure. However, \textit{certification} problems can be sometimes harder than the associated recovery problem \cite{bandeira2019computationalhardnesscertifyingbounds}. 

\paragraph*{Approximate Message Passing (AMP)} AMP is an iterative procedure that is believed to be optimal amongst polynomial time algorithms. Failure of AMP is often taken as an evidence for the Hardness of a problem \cite{lesieur2016phase,DavidL.DonohoandArianMalekiandAndreaMontanari}. AMP can be approximated by low-degree polynomials; hence LD hardness is stronger than AMP hardness. \cite{montanari2024equivalence} proved the equivalence of AMP and LD polynomials in Rank-One Matrix Estimation.

\paragraph*{SQ Lower bound~\cite{kearns1998efficient}} The Statistical Query model corresponds to a framework where we can access to queries, which are noised version of the expectation of a chosen function. Lower-bounding the number of queries needed is taken as a proxy for the time complexity of this problem \cite{SQclustering} \cite{pmlr-v195-diakonikolas23b}. In a lot of detection models, the SQ framework is equivalent to the LD framework \cite{brennan2020statistical}.

\paragraph*{Spectral Methods} Spectral Algorithms rely on computing leading eigenvectors or singular vectors of a well chosen matrix constructed from the data. \cite{BBP05} proves an asymptotic phase transition for spectral detection in sparse PCA that we refer as the BBP transition. Numerous spectral methods have studied for different models, such as clustering \cite{VEMPALA2004, banks2018information}, learning of distributions \cite{AchlioptasMcSherry2005}, the Stochastic Block model \cite{YunP14b, LeiRinaldo}, or as said before sparse PCA. Since eigenvectors of matrices can be approximated via power-iteration, spectral methods with respect to matrices which are polynomials of the data can be approximated by LD polynomials.

\paragraph*{Landscape Analysis} In optimization problems, it is possible to provide a barrier for stable Algorithms with the Overlap Gap Property \cite{gamarnik2021overlap}. This property can be extended to problems of estimation with planted structure and provide barriers for algorithms such as MCMC~\cite{chen2024low}.

\section{Background on cumulants}\label{sec:cumulants}\label{sec:cumulant}

 From \cite{SchrammWein22}, we know that in a Gaussian Additive model, proving LD lower bounds can be reduced to computing some joint cumulants $\kappa_{x,\alpha}$. We provide in this section a brief overview on cumulants, and we refer e.g. to \cite{novak2014three} for more details. 

\begin{definition}
Let $W_1,\ldots, W_l$ be random variables on the same space $\mathcal{W}$. Their cumulant generating function is the function $$K(t_1,\ldots, t_l):=\log\E\cro{\exp\pa{\sum_{l'=1}^l}t_{l'}W_{l'}}\enspace,$$
and their joint cumulant is the quantity $$\cumul\pa{W_1,\ldots,W_l}:=\pa{\pa{\prod_{l'=1}^l\frac{\partial}{\partial t_{l'}}}K(t_1,\ldots,t_l)}_{t_1,\ldots, t_l=0}\enspace.$$
\end{definition}

The joint cumulant of random variables can be expressed as a linear combination of their mixed moments and vice-versa. 

\begin{lem}\label{lem:mobiusformula}
Let $W_1,\ldots, W_l$ be random variables on the same space $\mathcal{W}$. Let $\mathcal{P}([l])$ stands for the collections of permutations of $[l]$. Then 
\begin{align}
\E\cro{W_{1}\cdots W_{l}}  &= \sum_{\pi\in\mathcal{P}([l])}\prod_{R\in \pi}\cumul\pa{W_{j}:j\in R}\enspace,\label{eq:cumul2moment}\\
\text{and}\quad\cumul\pa{W_1,\ldots, W_l}&=\sum_{\pi\in\mathcal{P}([l])}m(\pi)\prod_{R\in \pi}\E\cro{\prod_{l'\in R}W_{l'}}\enspace,
\end{align}
where $m(\pi)$ is the Möbius function $m(\pi)=(-1)^{|\pi|-1}\pa{|\pi|-1}!$.
\end{lem}

By considering aside the trivial partition $\pi$ with one group, and by enumerating the partitions $\pi$ on $[l]$ by considering 
$\ac{l}\subset R_{0} \subset \ac{1,\ldots,l}$ and partitions $\pi'$ of $[l]\setminus R_{0}$, we get from \eqref{eq:cumul2moment}, with the short notation $\cumul[R]:=\cumul\pa{W_{j}:j\in R}$
\begin{align}
\E[W_{1}\cdots W_{l}]&=\cumul\cro{[l]} +\sum_{\ac{l}\subset R_{0} \subsetneq \ac{1,\ldots,l}}\cumul\cro{R_{0}}\sum_{\pi'\in \mathcal{P}([l]\setminus R_{0})}  \prod_{R\in \pi'} \cumul\cro{R}\nonumber\\
&=\cumul\cro{[l]} + \sum_{\ac{l}\subset R_{0} \subsetneq \ac{1,\ldots,l}}\cumul\cro{R_{0}}\ \E\bigg[\prod_{j\in [l]\setminus R_{0}}W_{j}\bigg]. \label{eq:rec:cumul}
\end{align}

As noticed by \cite{SchrammWein22}, a key feature for proving LD lower-bounds is next lemma, which gives a sufficient condition for the nullity of cumulants.

\begin{lem}\label{lem:independentcumulant2}
Let $W_1,\ldots, W_l$ be random variables on the same space $\mathcal{W}$. Suppose that there exist disjoint sets $L_1$ and $L_2$, non-empty and covering $[l]$, such that $(W_i)_{i\in L_1}$ and $(W_i)_{i\in L_2}$ are independent. Then, we have the nullity of the joint cumulant $\cumul\pa{W_1,\ldots, W_l}=0$.
\end{lem}

\subsection{Further Notation for the control of the cumulants}
Here, we gather some notation that we use repeatdly in the proof.  For $\alpha\in \N^{n\times p}$ a multiset and $i\in [n]$, we write $\alpha_{i:}$ the $i$-th row of $\alpha$. Similarly, for $j\in [p]$, we write $\alpha_{:j}$ the $j$-th column of $\alpha$. We denote $supp(\alpha)=\ac{i\in [n], \alpha_{i:}\neq 0}$ and $col(\alpha)=\ac{j\in [p], \alpha_{:j}\neq 0}$. Then, we denote $\#\alpha=|supp(\alpha)|$ and $r_\alpha=|col(\alpha)|$. Finally, we shall write $|\alpha|$ the $l_1$-norm of $\alpha$. Finally, $\alpha!$ stands for $\prod_{ij}\alpha_{ij}!$ and, for any real valued matrix $Q$, $Q^\alpha=\prod_{ij}Q_{ij}^{\alpha_{ij}}$. For any finite set $S$, we write $|S|$ its cardinality. 

Given a graph $G=\pa{V,E}$, with $V$ the set of nodes and $E$ the set of edges, and $V'\subset V$, we write $G[V']$ the restriction of $G$ to the nodes in $V'$, i.e $G'=\pa{V', E'}$ with $E'=V'^2\cap E$. We write $cc(G)$ the number of connected components of $G$.

\section{Proof of Theorem \ref{thm:lowdegreeclusteringsharp}}\label{prf:lowdegreeclusteringsharp}

With no loss of generality, we assume in all the proof that $\sigma^2=1$.
Let $D\in \N$.
In the remaining of the proof, we write $x=M^*_{12}=\1_{k^*_1=k^*_2}$. Then, our goal is to lower-bound 
$$MMSE_{\leq D}=\inf_{f\in \R_{D}(Y)}\E\cro{\pa{f(Y)-x}^2}\enspace,$$
or equivalently, according to (\ref{eq:MMSE-corr}), to upper-bound $corr^2_{\leq D}$ defined in (\ref{def:corr}).  
More precisely,  proving Theorem \ref{thm:lowdegreeclusteringsharp} is equivalent to proving
\begin{equation*}
    corr^2_{D}\leq \frac{1}{K^2}\cro{1+\frac{\zeta}{(1-\sqrt{\zeta})^3}},\quad \textrm{for}\quad \zeta:=\frac{\Bar{\Delta}^4}{p}\max\pa{D^{18},\frac{n}{K^2}}<1\enspace.
\end{equation*}
The clustering model is a special case of the latent model (\ref{eq:latent-model}), with
$$Z=k^*,\quad \delta_{ij}(k^*)=1,\quad \textrm{and}\quad \theta_{i,j}(k^*)=(k^*_{i},j).$$
Combining Proposition \ref{thm:schrammwein} and Theorem \ref{thm:LTC}, we need to upper bound the cumulant, for any decomposition $\alpha=\beta_1+\ldots+\beta_l$, with $|\beta_s|=2$ for $s=1,\ldots, l$,

$$C_{x,\beta_{1},\ldots,\beta_{l}} =  \cumul\pa{x,\1_{\Omega_{\beta_{1}}(k^*)},\ldots,\1_{\Omega_{\beta_{l}}(k^*)}},$$

with 
$\Omega_{\beta}(k^*):= \ac{\big|\ac{(k^*_{i},j):\enspace (i,j)\in \beta}\big|=1}$.
Building on the recursive Bound (\ref{eq:rec-C}), we derive in Section \ref{prf:boundLTCclusteringsharp} the following upper-bound.

 \begin{lem}\label{lem:boundLTCclusteringsharp}
  We recall that $\#\alpha$ stands for the cardinality of the points $i\in[1,n]$ such that $\alpha_{i:}\neq 0$, and that $|\alpha|:=\sum_{ij}\alpha_{ij}$.
We have
 \begin{equation}\label{eq:boundLTCclusteringsharp}
|C_{x,\beta_{1},\ldots,\beta_l}|\leq |\alpha|^{|\alpha|-2\#\alpha+4}\pa{\frac{1}{K}}^{\#\alpha-1}\enspace.
\end{equation}
\end{lem}
Combining this bound with (\ref{eq:LTCgeneral:thm2}) and counting the number of partitions $\pi\in \mathcal{P}_2(\alpha)$ such that $C_{x,\beta_{1}(\pi),\ldots,\beta_{l}(\pi)}\neq 0$, we prove in Section \ref{prf:controlcumulantsclusteringsharp4}  the next  upper-bound on $|\kappa_{x,\alpha}|$.
\medskip

\begin{lem}\label{lem:controlcumulantsclusteringsharp}
Let $\alpha\in\N^{n\times p}$ non-zero. We have the upper bound
$$|\kappa_{x,\alpha}|\leq  \pa{\frac{1}{K}}^{\#\alpha-1}\lambda^{|\alpha|}|\alpha|^{|\alpha|-2\#\alpha+4}|\alpha|^{|\alpha|-\#\alpha-r_\alpha+1}\enspace.$$
\end{lem}

\medskip

The last stage, is to prune the multiset $\alpha$ for which $\kappa_{x,\alpha}=0$.
Next lemma gives necessary conditions for having $\kappa_{x,\alpha}\neq 0$.  For this purpose, it is convenient to introduce a bipartite multigraph $\mathcal{G}_\alpha$ on two disjoint sets $U=\ac{u_1,\ldots, u_n}$ and $V=\ac{v_1,\ldots, v_p}$ with, for $i,j\in [n]\times [p]$, $\alpha_{ij}$ edges between $u_i$ and $v_j$. We write $\mathcal{G}_\alpha^-$ the restriction of $\mathcal{G}_\alpha$ to non-isolated points. We denote $U(\alpha)$ the elements of $U$ spanned by $\mathcal{G}_\alpha^-$ and $V(\alpha)$ the elements of $V$ spanned by $\mathcal{G}_\alpha^-$. We refer to Section
\ref{prf:nullcumulantsclusteringsharp} for a proof of this lemma.

\begin{lem}\label{lem:nullcumulantsclusteringsharp}
Let $\alpha\in\N^{n\times p}$ be non-zero. If $\kappa_{x,\alpha}\neq 0$, then
\begin{itemize}
\item $u_1, u_2\in U(\alpha)$;
\item $\mathcal{G}^-_\alpha\cup \ac{(u_{1},u_{2})}$ is connected;
\item All the elements of $U(\alpha)\setminus \ac{u_{1},u_{2}}$ and $V(\alpha)$ are of degree at least 2.
\end{itemize}
In particular, we have $\#\alpha\geq 2$, $|\alpha|\geq 2r_\alpha$ and $|\alpha|\geq 2\#\alpha-2$.
\end{lem}
\medskip

\begin{rem}
In fact, we can prove that $\mathcal{G}^-_\alpha$ is connected (see \cite{Even24}), but it is sufficient and more straightforward to prove that $\mathcal{G}^-_\alpha\cup \ac{(u_{1},u_{2})}$ is connected.
\end{rem}
\medskip

We derive from Lemma \ref{lem:nullcumulantsclusteringsharp} the next lemma, which upper-bounds the cardinality of the  $\alpha$'s providing non-zero $\kappa_{x,\alpha}$, in terms of $|\alpha|$, $\#\alpha$ and $r_{\alpha}:=|\ac{j\in[1,p]\enspace \alpha_{:j}\neq 0}|$. We refer to Section \ref{prf:combinatorics_kappa} for a proof of this lemma.

\medskip

\begin{lem}\label{lem:combinatorics_kappa}
Given $m\geq 2$, $r\geq 1$, $d\geq \max(2m-2,2r)$, there exists at most $d^{3(d-r-m+2)}n^{m-2}p^r$ matrices $\alpha\in \N^{n\times p}$ satisfying the conditions of Lemma \ref{lem:nullcumulantsclusteringsharp} with $\#\alpha=m$, $r_{\alpha}=r$ and $|\alpha|=d$. 
\end{lem}

We now have all the pieces to upper-bound 
 the degree-D correlation $corr^2_{\leq D}$. Given $d\geq 1$, we set $\mathcal{D}_{d}:=\ac{m,r\in[2,d]\times [1,d],\enspace d\geq 2m-2,\enspace d\geq 2r}$. We have
\begin{align*}
corr^2_{\leq D}\leq& \underset{|\alpha|\leq D}{\sum_{\alpha\in\N^{n\times p}}}\kappa_{x,\alpha}^2\\
\leq&\frac{1}{K^2}+\sum_{d=2}^{D}\sum_{m,r\in\mathcal{D}_{d}}p^{r}n^{m-2}d^{5(d-r-(m-2))}d^{2d-4(m-2)}\lambda^{2d}\pa{\frac{1}{K^2}}^{m-1}\\
\leq&\frac{1}{K^2}+\frac{1}{K^2}\sum_{d=1}^{D}\sum_{m,r\in\mathcal{D}_{d}}(D^{7}\lambda^2)^{d}\pa{\frac{p}{D^{5}}}^r\pa{\frac{n}{K^2  D^{9}}}^{m-2}\enspace.
\end{align*}
Given $d\in[1,D]$ and $m,r\in\mathcal{D}_{d}$, let us upper-bound $(D^{7}\lambda^2)^{d}\pa{\frac{p}{D^{5}}}^r\pa{\frac{n}{K^2  D^{9}}}^{m-2}$. First, let us assume that $r\geq m-1$. By definition of $\mathcal{D}_d$, we can assume that $d\geq 2r$. Recall the definition of $\zeta$ in the statement of the theorem. We get
\begin{align*}
(D^{7}\lambda^2)^{d}\pa{\frac{p}{D^{5}}}^r\pa{\frac{n}{K^2  D^{9}}}^{m-2}=&\pa{D^{7}\frac{1}{p}\bar{\Delta}^2}^{d-2r}\pa{\frac{D^{9}\bar{\Delta}^4}{p}}^{r-(m-2)}\pa{\frac{\bar{\Delta}^4 n}{pK^2}}^{m-2}\\
\leq&\zeta^{\frac{d-2r}{2}}\zeta^{r-(m-2)}\zeta^{m-2}
\leq\zeta^{\frac{d}{2}}\enspace.
\end{align*}
Now, let us suppose that $r\leq m-2$, and let us consider the case $d\geq 2m-2$. 
\begin{align*}
(D^{7}\lambda^2)^{d}\pa{\frac{p}{D^{5}}}^r\pa{\frac{n}{K^2  D^{9}}}^{m-2}=&\pa{D^{7}\frac{1}{p}\bar{\Delta}^2}^{d-2(m-2)}\pa{\frac{nD^5}{K^2p^2}\bar{\Delta}^4}^{m-2-r}\\
&\times\pa{\bar{\Delta}^4\frac{n}{K^2 p}}^r\enspace.
\end{align*}
It is clear using directly the definition $\zeta=\frac{\Bar{\Delta}^4}{p}\max\pa{ D^{18},\frac{n}{K^2}}$ that $D^{7}\frac{1}{p}\bar{\Delta}^2\leq \sqrt{\zeta}$ and $\bar{\Delta}^4\frac{n}{K^2 p}\leq \zeta$. For the remaining factor $\frac{nD^5}{K^2p^2}\bar{\Delta}^4$, we use the hypothesis $p\geq D^5$. Thus, $\frac{nD^5}{K^2p^2}\bar{\Delta}^4\leq \frac{n}{K^2p}\bar{\Delta}^4\leq \zeta$. Hence, 
\begin{align*}
(D^{9}\lambda^2)^{d}\pa{\frac{p}{D^{7}}}^r\pa{\frac{n}{K^2  D^{11}}}^{m-2}\leq&\zeta^{\frac{d-2(m-2)}{2}}\zeta^{m-2-r}\zeta^{r}
\leq\zeta^{\frac{d}{2}}\enspace.
\end{align*}
Hence, for all $d\in[1,D]$ and $m,r\in\mathcal{D}_d$, we have $(D^{9}\lambda^2)^{d}\pa{\frac{p}{D^{7}}}^r\pa{\frac{n}{K^2 16 D^{11}}}^{m-2}\leq\zeta^{\frac{d}{2}}$. Combining this with $|\mathcal{D}_d|\leq \frac{d(d-1)}{2}$ leads to
\begin{align*}
corr^2_{\leq D}\leq&\frac{1}{K^2}+\frac{1}{K^2}\sum_{d=2}^{D}\frac{d(d-1)}{2}\zeta^{d/2}\\
\leq&\frac{1}{K^2}\cro{1+\frac{\zeta}{(1-\sqrt{\zeta})^3}}\enspace.
\end{align*} 
This concludes the proof of the theorem.

\subsection{Proof of Lemma \ref{lem:boundLTCclusteringsharp}}\label{prf:boundLTCclusteringsharp}

Let $\beta_1,\ldots,\beta_l$ such that $|\beta_s|=2$ for $s\in [l]$ and such that $\beta_1+\ldots+\beta_l=\alpha$. We seek to upper-bound

$$C_{x,\beta_{1},\ldots,\beta_{l}} =  \cumul\pa{x,\1_{\Omega_{\beta_{1}}(k^*)},\ldots,\1_{\Omega_{\beta_{l}}(k^*)}},$$

with $\Omega_{\beta}(k^*):= \ac{\big|\ac{(k^*_{i},j):\enspace (i,j)\in \beta}\big|=1}$. For $C_{x,\beta_{1},\ldots,\beta_{l}}$ to be non-zero, it is necessary that, for $s\in [l]$, $\Omega_{\beta_s}(k^*)$ is an event of positive probability. This condition implies that $\beta_s$ must be contained in a  single column of $\alpha$, which we denote $j_s$. We write $\beta_s=\ac{(i_s,j_s);(i'_s,j_s)}$, for $s\in [l]$. We also take the convention $i_0=1$, $i'_0=2$, and $j_0=0$. We then have 
$$C_{x,\beta_{1},\ldots,\beta_{l}} =  \cumul\pa{\1\ac{k^*_{i_s}=k^*_{i'_s}}_{s\in [0,l]}}\enspace.$$
For $S\subseteq [l]$, we denote $\beta[S]=\ac{\beta_s,s\in S}$ and $\alpha_S=\sum_{s\in S}\beta_s$. Applying the recursion formula \ref{eq:rec-C}, we deduce that, for all $S\subseteq [l]$,
\begin{align}\nonumber
|C_{x,\beta[S]} |\leq&   \left|\E\cro{\prod_{s\in \ac{0}\cup S}\1\ac{k^*_{i_s}=k^*_{i'_s}}}\right|+\sum_{S'\subsetneq S}\left|C_{x,\beta[S']}\right|\left|\E\cro{\prod_{s\in S\setminus S'}\1\ac{k^*_{i_s}=k^*_{i'_s}}}\right|\\
\leq&\P\cro{\forall s\in \ac{0}\cup S,\enspace k^*_{i_s}=k^*_{i'_s} }+\sum_{S'\subsetneq S}\left|C_{x,\beta[S']}\right|\P\cro{\forall s\in S\setminus S',\enspace k^*_{i_s}=k^*_{i'_s} }\enspace.\label{eq:rec-C-clustering}
\end{align}

Let us compute, for any subset $R\subseteq[0,l]$, the quantity $\P\cro{\forall s\in R,\enspace k^*_{i_s}=k^*_{i'_s} }$. To do so, let us define $\mathcal{V}$ the graph on $[0,l]$ defined by; for $s,s'\geq 0$, there is an edge between $s$ and $s'$ if and only if $\ac{i_s, i'_s}\cap \ac{i_{s'}, i'_{s'}}\neq \emptyset$. Let $\mathcal{V}[R]$ denote the restriction of $\mathcal{V}$ to $R$ and $cc(\mathcal{V}[R])$ the number of connected components of this graph. Let $cc_1,\ldots, cc_{q^*}$ be the connected components of $\mathcal{V}[R]$, with $q^*=cc(\mathcal{V}[R])$. For $q\in [q^*]$, let $i_{q}$ be any element of $\cup_{s\in cc_{q}}\ac{i_s, i'_s}$. Having for all $s\in R$, $k^*_{i_s}=k^*_{i'_s}$ is equivalent to having, for all $q\in [q^*]$ and for all $i\in \cup_{s\in cc_{q'}}\ac{i_s, i'_s}$, $k^*_i=k^*_{i_{q}}$. Such an event occurs with probability

\begin{align*}
\P\cro{\forall q\in [q^*],\enspace \forall i\in \cup_{s\in cc_{q}}\ac{i_s, i'_s}, \enspace k^*_i=k^*_{i_{q}}}=&\prod_{q\in [q^*]}\P\cro{\forall i\in \cup_{s\in cc_{q}}\ac{i_s, i'_s}, \enspace k^*_i=k^*_{i_{q}}}\\
=&\prod_{q\in [q^*]}\pa{\frac{1}{K}}^{|\cup_{s\in cc_{q}}\ac{i_s,i'_s}|-1}\enspace.
\end{align*}

By definition of the graph $\mathcal{V}$, we have $\sum_{q\leq q^*}|\cup_{s\in cc_{q}}\ac{i_s,i'_s}|=|\cup_{s\in R}\ac{i_s, i'_s}|$. If $R$ does not contain $0$, then $|\cup_{s\in R}\ac{i_s, i'_s}|=\#\alpha_R$, with $\alpha_R=\sum_{s\in R}\beta_s$. If $R$ contains $0$, $|\cup_{s\in R}\ac{i_s, i'_s}|=|supp(\alpha_{R\setminus \ac{0}})\cup \ac{1,2}|$. Plugging this in \eqref{eq:rec-C-clustering} leads to

\begin{equation}\label{eq:recursionLTCclusteringsharp}
|C_{x,\beta[S]} |\leq \pa{\frac{1}{K}}^{|supp(\alpha_S)\cup\ac{1,2}|-cc(\mathcal{V}[S])}+\sum_{S'\subsetneq S}|C_{x,\beta[S']}|\pa{\frac{1}{K}}^{\#\alpha_{S\setminus S'}-cc\pa{\mathcal{V}[S\setminus S']}}\enspace.
\end{equation}

The next lemma prunes the subsets $S\subseteq [l]$ such that $C_{x,\beta[S]}\neq 0$. In the following, we denote $\mathcal{S}([l])$ the set of all subsets $S\subseteq [l]$ satisfying;\begin{enumerate}
\item $\mathcal{V}[\ac{0}\cup S]$ is connected;
\item If $S\neq \emptyset$, then, for all $i\in \cup_{\ac{0}\cup S}\ac{i_s,i'_s}$, there exists $s\neq s'\in \ac{0}\cup S$ such that $i\in \ac{i_s,i'_s}$ and $i\in \ac{i_{s'},i'_{s'}}$. In particular, $1,2\in supp(\alpha_S)$.
\end{enumerate}
\medskip

\begin{lem}\label{lem:nullityLTCclusteringsharp}
Let $S\subseteq [l]$ such that $C_{x,\beta[S]}\neq 0$. Then $S\in \mathcal{S}([l])$
\end{lem}

\begin{proof}[Proof of Lemma \ref{lem:nullityLTCclusteringsharp}]
Let us first suppose that $\mathcal{V}[\ac{0}\cup S]$ is not connected and let us prove that $C_{x,\beta[S]}=0$. Let $S_1$ and $S_2$ be a partition of $\ac{0}\cup S$ with no edges connecting them. Then, $\pa{k^*_i}_{i\in \cup_{s\in S_1}\ac{i_s, i'_s}}$ is independent from $\pa{k^*_i}_{i\in \cup_{s\in S_2}\ac{i_s, i'_s}}$ and so $\pa{\1\ac{k^*_{i_s}=k^*_{i'_s}}}_{s\in S_1}$ is independent from $\pa{\1\ac{k^*_{i_s}=k^*_{i'_s}}}_{s\in S_2}$. Lemma \ref{lem:independentcumulant2} then implies that $C_{x,\beta[S]}=0$.

Now, suppose $S\neq \emptyset$ and suppose that there exists  $\underline{i}\in \cup_{s\in \ac{0}\cup S}\ac{i_s, i'_s}$ such that there exists only one $s_0\in \ac{0}\cup S$ such that $\underline{i}\in \ac{i_{s_0}, i'_{s_0}}$. We suppose by symmetry that $\underline{i}=i_{s_0}$. Conditionally on $\pa{\1\ac{k^*_{i_s}=k^*_{i'_s}}}_{s\in \pa{\ac{0}\cup S}\setminus \ac{s_0}}\bigcup \ac{k^*_{i'_{s_0}}}$, the variable $\1\ac{k^*_{i_{s_0}}=k^*_{i'_{s_0}}}$ is either a Bernoulli of parameter $\frac{1}{K}$ if $i_{s_0}\neq i'_{s_0}$, either deterministic and equal to $1$ if $i_{s_0}= i'_{s_0}$.  A fortiori, $\1\ac{k^*_{i_{s_0}}=k^*_{i'_{s_0}}}$ is independent from $\pa{\1\ac{k^*_{i_s}=k^*_{i'_s}}}_{s\in \pa{\ac{0}\cup S}\setminus \ac{ s_0}}$. Lemma \ref{lem:independentcumulant2} again implies $C_{x,\beta[S]}=0$.

\end{proof}

\medskip

Pruning the other terms in \eqref{eq:recursionLTCclusteringsharp} leads us to, for all $S\in \mathcal{S}([l])$, 
\begin{equation}\label{eq:recursionLTCclusteringsharp3}
|C_{x,\beta[S]}|\leq \pa{\frac{1}{K}}^{\#\alpha_S-1}+\underset{S'\in \mathcal{S}([l])}{\sum_{S'\subsetneq S}}|C_{x,\beta[S']}|\pa{\frac{1}{K}}^{\#\alpha_{S\setminus S'}-cc\pa{\mathcal{V}[S\setminus S']}}\enspace.
\end{equation}

In the following, let us define recursively a function $f$ on $\mathcal{S}([l])$ satisfying, for all $S\in \mathcal{S}([l])$,
\begin{equation}\label{eq:deffunctionf}
f(S)=1+\underset{S'\in \mathcal{S}([l])}{\sum_{S'\subsetneq S}}f(S')\enspace,
\end{equation}
with $f(\emptyset)=1$. The next lemma upper-bounds, for $S\in \mathcal{S}([l])$, $|C_{x,\beta[S]}|$ with respect to $f(S)$. We refer to Section \ref{prf:upperbounfkappaS} for a proof of this lemma.

\begin{lem}\label{lem:upperbounfkappaS}
For all $S\in \mathcal{S}([l])$, we have $|C_{x,\beta[S]}|\leq \pa{\frac{1}{K}}^{\#\alpha_{S}-1}f(S)$.
\end{lem}

It remains to upper-bound $f(S)$ for all $S\in \mathcal{S}([l])$. We postpone to Section \ref{prf:upperboundf} the computation leading to next lemma.

\medskip
\begin{lem}\label{lem:upperboundf}
For all $S\in \mathcal{S}([l])$ with $S\neq \emptyset$, we have $f(S)\leq |\alpha|^{|\alpha_S|-2\#\alpha_S+4}$.
\end{lem}

Plugging Lemma \ref{lem:upperboundf} in Lemma \ref{lem:upperbounfkappaS} leads us to 

\begin{equation*}|C_{x,\beta_1,\ldots,\beta_l}|\leq |\alpha|^{|\alpha|-2\#\alpha+4}\pa{\frac{1}{K}}^{\#\alpha-1}\enspace,
\end{equation*}

which concludes the proof of the lemma.

\subsection{Proof of Lemma \ref{lem:upperbounfkappaS}}\label{prf:upperbounfkappaS}

Let us prove by induction that, for all $S\in \mathcal{S}([l])$, $|C_{x,\beta[S]}|\leq \pa{\frac{1}{K}}^{\#\alpha_{S}-1}f(S)$. The initialization is straightforward since $C_{x}=\kappa\pa{\1\ac{k^*_1=k^*_2}}=\frac{1}{K}$ and $\alpha_{\emptyset}=0$.

For the induction, let $S\in \mathcal{S}([l])$ and let us suppose that the result holds for all $S'\in \mathcal{S}([l])$ with $S'\subsetneq S$. Applying Inequality \eqref{eq:recursionLTCclusteringsharp3} to $S$ together with the induction hypothesis leads to 

\begin{align*}
|C_{x,\beta[S]}|\leq&\pa{\frac{1}{K}}^{\#\alpha_{S}-1}+\underset{S'\in \mathcal{S}([l])}{\sum_{S'\subsetneq S}} |C_{x,\beta[S']}|\pa{\frac{1}{K}}^{\#\alpha_{S\setminus S'}-cc\pa{\mathcal{V}[S\setminus S']}}\\
\leq& \pa{\frac{1}{K}}^{\#\alpha_{S}-1}+\frac{1}{K}\pa{\frac{1}{K}}^{\#\alpha_{S}-cc(\mathcal{V}[S])}+\\
&+\underset{\emptyset\neq S'\in \mathcal{S}([l])}{\sum_{S'\subsetneq S}}f(S')\pa{\frac{1}{K}}^{\#\alpha_{S'}-1}\pa{\frac{1}{K}}^{\#\alpha_{S\setminus S'}-cc\pa{\mathcal{V}[S\setminus S']}}\\
\leq &\pa{\frac{1}{K}}^{\#\alpha_{S}-1}+\underset{S'\in \mathcal{S}([l])}{\sum_{S'\subsetneq S}}f(S')\pa{\frac{1}{K}}^{\#\alpha_{S'}-1+\#\alpha_{S\setminus S'}-cc\pa{\mathcal{V}[S\setminus S']}}\enspace.
\end{align*}
In the last line, for the term corresponding to $S'=\emptyset$, we used the fact that $cc(\mathcal{V}[S])\leq cc(\mathcal{V}[S\cup \ac{0}])+1\leq 2$, since $\mathcal{V}[S\cup \ac{0}]$ is assumed to be connected. We deduced from that $\frac{1}{K}\pa{\frac{1}{K}}^{\#\alpha_{S}-cc(\mathcal{V}[S])}\leq \pa{\frac{1}{K}}^{\#\alpha_{S}-1}$.

The next lemma uses the connectivity of the graph $\mathcal{V}[\ac{0}\cup S]$ in order to lower-bound the other exponents in the above inequality.
\medskip

\begin{lem}\label{lem:connectivityLTCclusteringsharp}
For all $\emptyset \neq S'\subsetneq S$ with $S,S'\in \mathcal{S}([l])$, $$\#\alpha_{S'}-1+\#\alpha_{S\setminus S'}-cc\pa{\mathcal{V}[S\setminus S']}\geq \#\alpha_{S}-1 \enspace.$$
\end{lem}

\medskip

Applying Lemma \ref{lem:connectivityLTCclusteringsharp} leads to 

\begin{equation*}
|C_{x, \beta[S]}|\leq \pa{\frac{1}{K}}^{\#\alpha_{S}-1}\pa{1+\underset{S'\in \mathcal{S}([l])}{\sum_{S'\subsetneq S}}f(S')}=\pa{\frac{1}{K}}^{\#\alpha_{S}-1}f(S)\enspace,
\end{equation*}

which concludes the induction.

\subsection{Proof of Lemma \ref{lem:connectivityLTCclusteringsharp}}\label{prf:connectivityLTCclusteringsharp2}

For any subset $R\subseteq [l]$, we write $\alpha_R=\sum_{s\in R}\beta_s$.

Let $\emptyset \neq S'\subsetneq S$ with $S,S'\in \mathcal{S}([l])$. Let $q^*=cc\pa{\mathcal{V}[S\setminus S']}$ and let us write $cc_1, \ldots, cc_{q^*}$ those connected components. Since $\mathcal{V}[S\cup \ac{0}]$ is connected, we deduce that, for all $q\leq q^*$, $cc_{q}$ is connected to $S'\cup \ac{0}$ in $\mathcal{V}[S\cup \ac{0}]$. This implies that, for all $q\in [q^*]$, $\#\alpha_{cc_{q}}\geq 1+|supp(\alpha_{cc_{q}})\setminus supp(\alpha_{S'})|$. Hence, we have 
\begin{align*}
\#\alpha_{S'}-1+\#\alpha_{S\setminus S'}-cc\pa{\mathcal{V}[S\setminus S']}=&\#\alpha_{S'}-1+\sum_{q\in [q^*]}(\#\alpha_{cc_{q}}-1)\\
\geq & \#\alpha_{S'}-1+\sum_{q\in [q^*]}|supp(\alpha_{cc_{q}})\setminus supp(\alpha_{S'})|\\
\geq &\#\alpha_{S}-1\enspace,
\end{align*}
which concludes the proof of the lemma.

\subsection{Proof of Lemma \ref{lem:upperboundf}}\label{prf:upperboundf}

We shall prove, by induction, that for all $S\in \mathcal{S}([l])$,
\begin{equation}
\label{eq:upper_F_S_a}
f(S)\leq |\alpha|^{|\alpha_S|-2|supp(\alpha_S)\cup \ac{1,2}|+4}\enspace.
\end{equation}
In fact, the bound~\eqref{eq:upper_F_S_a} implies the desired result. Indeed, for $S\neq \emptyset$, we deduce from the definition of $\mathcal{S}([l])$ that  $1,2\in supp(\alpha_S)$. This implies that $|supp(\alpha_{S}\cup\ac{1,2})|=\#\alpha_S$. The case $S=\emptyset$ is straightforward as $f(\emptyset)=1$.

Hence, we only need to prove~\eqref{eq:upper_F_S_a}.
The initialization is trivial since $f(\emptyset)=1=0^0$ and $\alpha_{\emptyset}=0$. Let us take $S\in  \mathcal{S}([l])$ and let us suppose that the result holds for all $S'\in \mathcal{S}([l])$ with $S'\subsetneq S$. For all $s\in S$, let $S^*(s)$ be the maximal (with respect to the inclusion) element  of $\mathcal{S}([l])$ which is included in $S\setminus \ac{s}$. The existence of such an element in ensured by the fact that the set of elements $S'\in \mathcal{S}([l])$ with $S'\subseteq S\setminus \ac{s}$ is not empty (it contains $\emptyset$) and is stable by union. We have

\begin{align*}
f(S)=&1+\underset{S'\in \mathcal{S}([l])}{\sum_{S'\subsetneq S}}f(S')\\
\leq & 1+\sum_{s\in S}\underset{S'\in \mathcal{S}([l])}{\sum_{S'\subseteq S^*(s)}}f(S')\\
\leq & 1+\sum_{s\in S} \cro{2f(S^*(s))-1}\\
\leq & 2\sum_{s\in S}f(S^*(s))\ , 
\end{align*}
where we used the recursive definition of $f$ in the third line. Applying the induction hypothesis leads us to 
\begin{equation}\label{eq:upper_F_S}
f(S)\leq 2\sum_{s\in S}|\alpha|^{|\alpha_{S^*(s)}|-2|supp(\alpha_{S^*(s)}\cup\ac{1,2})|+4}\enspace.
\end{equation}

Let $s\in S$ and let $i\in \pa{\cup_{s'\in S}\ac{i_{s'}, i'_{s'}}}\setminus \pa{\cup_{s'\in S^*(s)\cup \ac{0}}\ac{i_{s'}, i'_{s'}}}$. Since $S$ belong to $\mathcal{S}([l])$ and by definition of that collection, we know that there must exist at least two different $s_1,s_2\in S\setminus \pa{S^*(s)\cup \ac{0}}$ such that $i\in supp(\beta_{s_1})$ and $i\in supp(\beta_{s_2})$. And, from the connectivity of the graph $\mathcal{V}$, we know that there exists $i\in \pa{supp(\alpha_{S\setminus S^*(s)})\cup \ac{1,2}}\cap \pa{supp(\alpha_{S^*(s)})\cup \ac{1,2}}$. We deduce that 

$$|\alpha_S|\geq |\alpha_{S^*(s)}|+2\left|\pa{\cup_{s'\in S}\ac{i_{s'}, i'_{s'}}}\setminus \pa{\cup_{s'\in S^*(s)\cup \ac{0}}\ac{i_{s'}, i'_{s'}}}\right|+1\enspace.$$
We deduce that 
$$|\alpha_{S^*(s)}|-2|supp(\alpha_{S^*(s)})\cup\ac{1,2}|+4\leq  |\alpha_{S}|-2|supp(\alpha_{S}\cup\ac{1,2})|-1+4\enspace.$$
Coming back to~\eqref{eq:upper_F_S}, we get 
$$f(S)\leq \frac{2|S|}{|\alpha|}|\alpha|^{|\alpha_{S}|-2|supp(\alpha_{S}\cup\ac{1,2})|+4}\leq |\alpha|^{|\alpha_{S}|-2|supp(\alpha_{S}\cup\ac{1,2})|+4}\enspace\ ,$$
since, for $S\in \mathcal{S}([l])$, $|\alpha|\leq 2|S|$. We have proved~\eqref{eq:upper_F_S_a}.

\subsection{Proof of Lemma \ref{lem:controlcumulantsclusteringsharp}}\label{prf:controlcumulantsclusteringsharp4}

 To prove Lemma  \ref{lem:controlcumulantsclusteringsharp} we merely need to upper-bound the number of partitions $\pi=\pi_1,\ldots,\pi_l\in \mathcal{P}_2(\alpha)$, satisfying $C_{x,\beta_1(\pi),\ldots, \beta_l(\pi)}\neq 0$, where $\beta_s(\pi)$ counts the number of copies of $(i,j)$ in $\pi_s$. For such a partition, we have $l=\frac{|\alpha|}{2}$, and we need that all groups $\pi_s$ must be contained in a single column of $\alpha$. Hence, we only need to upper-bound the number of partition of each multiset $\alpha_{:j}$, for $j\in col(\alpha)$, into groups of size $2$.

For $j\in col(\alpha)$, the number of partitions of the multiset $\alpha_{:j}$ into pairings is at most $|\alpha_{:j}|^{|\alpha_{:j}|/2-1}$. We deduce that the number of satisfying partitions of $\alpha$ in pairings is at most $|\alpha|^{|\alpha|/2-r_\alpha}$.

Lemma \ref{lem:nullityLTCclusteringsharp} also implies that for all $i\in supp(\alpha)\setminus \ac{1,2}$, $|\alpha_{i:}|\geq 2$. We deduce that $|\alpha|\geq 2\#\alpha-2$ and so $|\alpha|^{|\alpha|/2-r_\alpha}\leq |\alpha|^{|\alpha|-\#\alpha-r_\alpha+1}$. Combining Lemma \ref{lem:boundLTCclusteringsharp} and \eqref{eq:LTCgeneral:thm2}
leads us to 
$$|\kappa_{x,\alpha}|\leq  \lambda^{|\alpha|}\pa{\frac{1}{K}}^{\#\alpha-1}|\alpha|^{|\alpha|-2\#\alpha+4}|\alpha|^{|\alpha|-\#\alpha-r_\alpha+1}\enspace,$$
which concludes the proof of Lemma \ref{lem:boundLTCclusteringsharp}.

\subsection{Proof of Lemma \ref{lem:nullcumulantsclusteringsharp}}\label{prf:nullcumulantsclusteringsharp}

Let $\alpha\neq 0$. Let us prove that if $\mathcal{G}_\alpha^{-}$ does not satisfy any of the three conditions of Lemma \ref{lem:nullcumulantsclusteringsharp}, then $\kappa_{x,\alpha}=0$.

\medskip
Let us first suppose that either $u_1\notin U(\alpha)$ or $u_2\notin U(\alpha)$. Then, conditionally on $\pa{X_{ij}}_{ij\in \alpha}$, $x$ is a Bernoulli of parameter $\frac{1}{K}$. Thus, $x$ is independent from $\pa{X_{ij}}_{ij\in \alpha}$. We conclude with Lemma \ref{lem:independentcumulant2} that $\kappa_{x,\alpha}=0$

\medskip
Let us suppose that $\mathcal{G}_\alpha^-\cup \ac{\pa{u_1,u_2}}$ is not connected. Recall  Theorem~\ref{thm:LTC}. For proving that $\kappa_{x,\alpha}=0$,  it is sufficient to prove that, for all decomposition $\beta_1+\ldots+\beta_l=\alpha$, with $\beta_s=\ac{(i_s,j_s);(i'_s,j_s)}$, we have $C_{x,\beta_1,\ldots,\beta_l}=0$. For such a decomposition to be non-zero, we need that the graph $\mathcal{V}[\ac{0}\cup [l]]$ defined in Section \ref{prf:boundLTCclusteringsharp} is connected --see Lemma \ref{lem:nullityLTCclusteringsharp}. This directly implies that $\mathcal{G}_\alpha^-\cup \ac{\pa{u_1,u_2}}$ is connected. So, if $\mathcal{G}_\alpha^-\cup \ac{\pa{u_1,u_2}}$ is not connected, all those decompositions satisfy $C_{x,\beta_1,\ldots,\beta_l}=0$ and we deduce $\kappa_{x,\alpha}=0$.

\medskip

It remains to prove that if a node in $U(\alpha)\setminus \ac{u_1,u_2}$ or $V(\alpha)$ is of degree $1$, then $\kappa_{x,\alpha}=0$. Let us first suppose that there exists $v_{j_0}\in V(\alpha)$ of degree $1$. $(x,\pa{X_{ij}}_{ij\in \alpha})$ has the same law as $\pa{x,\pa{X_{ij}(-1)^{j=j_0}}}$ and so $\kappa_{x,\alpha}=\kappa\pa{x, \pa{X_{ij}(-1)^{j=j_0}}}=-\kappa\pa{x, \pa{X_{ij}}}=-\kappa_{x,\alpha}$. Hence, $\kappa_{x,\alpha}=0$.

Let us now suppose that there exists a vertex $u_{i_0}\in U(\alpha)\setminus \ac{1,2}$ which is of degree $1$. As previously,  we know from Theorem~\ref{thm:LTC} that for proving that $\kappa_{x,\alpha}=0$, it is sufficient to prove that, for all decomposition $\alpha=\beta_1+\ldots+\beta_l$ with $\beta_s=\ac{(i_s,j_s);(i'_s,j_s)}$, we have $C_{x,\beta_1,\ldots,\beta_l}=0$. Let $\beta_1,\ldots, \beta_l$ be such a decomposition and let $s_0\in [l]$ such that $i_0\in \ac{i_{s_0},i'_{s_0}}$. Since $u_{i_0}$ is of degree at most $1$, it is clear that $i_0\notin \cup_{s\neq s_0}\ac{i_s,i'_s}$. Thus, $\1\{k^*_{i_{s_0}}=k^*_{i'_{s_0}}\}$ is independent from $\ac{x}\bigcup \pa{\1\ac{k^*_{i_s}=k^*_{i'_s}}}_{s\neq s_0}$. We deduce from lemma \ref{lem:independentcumulant2} that $C_{x,\beta_1,\ldots,\beta_l}=0$. This being true for all decomposition, we conclude $\kappa_{x,\alpha}=0$.

\subsection{Proof of Lemma \ref{lem:combinatorics_kappa}}\label{prf:combinatorics_kappa}

We adapt the proof of Lemma 5.5 of \cite{SohnWein25} to the case of bipartite multigraphs. We remark first that counting the matrices $\alpha$ is equivalent to counting the bipartite multigraphs $\mathcal{G}_\alpha$. Let us first construct the graphs $\mathcal{G}_\alpha$ with some basic operations. We will then upper-bound the number of graphs by counting the operations that are needed to construct all the graphs. In the following, for $G$ a bipartite multigraph of $U\times V$, we denote $G^-$ the graph obtained after removing all the isolated nodes.

In order to construct a bipartite multigraph $G$ of $U\times V$ such that $G^-$ satisfies the conditions of Lemma \ref{lem:nullcumulantsclusteringsharp},  we start with two isolated vertices $u_1$ and $u_2$. Then, recursively, we are allowed to add either a "path" or a "lollipop". Let us precise what these two operations on graphs correspond to:

\begin{itemize}
\item For adding a "path", we choose $l>0$ and two existing nodes $w^{(0)},w^{(l+1)}$ (not necessarily distinct) of $G^-$. Then, we choose distinct $w^{(1)},\ldots, w^{(t)}\in U\cup V$ (with the constraint that, for $l'\in [0,l]$, if $w_{(l')}\in U$ then $w_{(l'+1)}\in V$ and conversely) which are not nodes of $G^-$ and we add the edges $\pa{(w^{(l')}, w^{(l'+1)})}_{l'\in [0,l]}$ to $G$. We remark that all the new nodes added to $G^-$ are of degree $2$,
\item For constructing a "lollipop", we choose  $w^{(0)}$ an existing node of $G^-$. We choose distinct $w^{(1)},\ldots, w^{(l)}\in U\cup V$ (still with the constraint that, for $l'\in [0,l-1]$, if $w_{(l')}\in U$ then $w_{(l'+1)}\in V$ and conversely) which are not nodes of $G^-$. We then add to $G$ the edges  $((w^{(l')}, w^{(l'+1)}))_{l'\in [0,l-1]}$. Then, we choose $l'\in  [l-1]$ and we add the edge $(w^{(l)}, w^{(l')})$, with the constraint that $(w^{(l)}, w^{(l')})\in \pa{U\times V}\cup \pa{V\times U}$. We have added nodes of degree $2$ except one node of degree $3$ (which is $w^{(l')}$).
\end{itemize}

We postpone to Section \ref{prf:constructiongraphs}  the proof of the next lemma, which states that this construction of graphs is surjective.

\begin{lem}\label{lem:constructiongraphs}
All graph $G$, with $G^-$ satisfying the conditions of Lemma \ref{lem:nullcumulantsclusteringsharp}, can be obtained by a finite number of operations "path" or "lollipop".
\end{lem}

Now, suppose that a graph $\mathcal{G}_{\alpha}^-$ has been produced by $T$ operations "path" or "lollipop". Let $l_t$ denote  the number of vertices added at step $t$. The total number $|\alpha|$ of multi-edges  satisfies  $|\alpha|=\sum_t (l_t+1)$ and the total number of vertices $r_\alpha+\#\alpha$ satisfies $r_\alpha+\#\alpha-2=\sum_t l_t$. This implies that $T=|\alpha|-\#\alpha-r_\alpha+2=d-m-r+2$. Then, since for all $t\in [T]$, we have necessarily $l_t\in [1,d]$, for obtaining all the possible graphs, the number of possibilities for choosing the $l_t$'s is at most $d^{d-m-r+2}$. At each step, there are at most $d^2$ possibilities for choosing the existing vertices (counting the future existing vertex when we add a "lollipop").  Finally, the new vertices must be chosen either in $U\setminus \ac{1,2}$ or $V$ depending on where the existing edges are. It is clear that, in total, the number of nodes that we need to choose in $U$ is $m-2$ and the number of nodes that we need to choose in $V$ is $r$. When we choose a new node, the fact that it belongs to $U$ or $V$ is entirely determined by the previous choices of nodes. The final count is at most

$$d^{3(d-r-m+2)}n^{m-2}p^r\enspace,$$

which concludes the proof of the lemma.

\subsection{Proof of Lemma \ref{lem:constructiongraphs}}\label{prf:constructiongraphs}

Let $G$ a bipartite multigraph of $U\times V$ such that $G^-$ satisfies the conditions of Lemma \ref{lem:nullcumulantsclusteringsharp}, which are:
\begin{enumerate}
\item $u_1,u_2$ are nodes of  $G^-$,
\item $G^-\cup \ac{(u_1, u_2)}$ is connected;
\item All the nodes of $G^-$, except $u_1$ and $u_2$, are of degree at least $2$. Together with the first point, this implies that all the nodes of $G^-\cup \ac{(u_1, u_2)}$ are of degree at least $2$.
\end{enumerate}

Let us prove that $G^-$ can be obtained with a finite number of operations "path" or "lollipop". To do so, we deconstruct the graph $G$ by removing paths and lollipops. We write $G_0=G$ and we construct recursively a sequence of subgraphs of $G$ as follows.

Suppose that we are given a graph $G_t$, for $t\geq 0$. If it exists, choose an edge $e$ of $G_t^-$ whose absence does not disconnect the graph $G_t^-\cup \ac{(u_1,u_2)}$. Then, let $P$ be the maximal path included in $G_t^-$ containing $e$ and such that all the nodes inside this path are of degree $2$ and are not $u_1$ or $u_2$. We distinguish two cases:
\begin{itemize}
\item If the extremities of this path are distinct, or if they are not distinct but are a node of degree at least $4$ (or equal either to $u_1$ or $u_2$), we remove all the edges of the path to obtain the graph $G^-_{t+1}$. This corresponds to removing a "path".
\item If both extremities of this path are a same node of degree $3$ which is not in $\ac{u_1,u_2}$, then remove also all the edges of this path. There remains an edge $e'$ connecting the extremity to the rest of the graph. We consider the maximum path containing $e'$ such that all the nodes are of degree $2$ and are not $u_1$ or $u_2$. We also remove all the edges of this path. This corresponds to removing a "lollipop". 
\end{itemize}

We stop at the step $T$ which corresponds to the moment where there is no edge deconnecting $G_T^-\cup\ac{(u_1,u_2)}$. To conclude the proof of the lemma, it remains to prove that $G_T$ has no edge -which is the starting point of the construction of the graphs above-. First, we define the set  $\mathcal{H}$ of bipartite multigraphs satisfying, for $G\in \mathcal{H}$;

\begin{enumerate}
\item $G^-\cup \ac{(u_1,u_2)}$ is connected;
\item All the nodes of $G^-$, except $u_1$ and $u_2$, are of degree at least $2$.
\end{enumerate}

It is clear that $G_0\in \mathcal{H}$, since $G_0$ satisfies the conditions of Lemma \ref{lem:nullcumulantsclusteringsharp}. Lemma~\ref{lem:surjectivitygraphs} below implies that, if $G_t\in \mathcal{H}$, then $G_{t-1}$ belongs to $\mathcal{H}$. 

\begin{lem}\label{lem:surjectivitygraphs}
Let $G\in \mathcal{H}$ be a non empty graph. Let us remove either a "lollipop" or a "path" from $G$ with the scheme described above. Then, the obtained graph $G'$ also belongs to $\mathcal{H}$.
\end{lem}

In particular, $G_T\in \mathcal{H}$. Let us suppose that $G_T$ has an edge and let us find a contradiction. If $G_T$ only has edges between $u_1$ and $u_2$, then removing one of these edges does not disconnect $G_T^-\cup \ac{(u_1,u_2)}$ and this leads to a contradiction.

Otherwise, one can extract from $G_T^-\cup \ac{(u_1,u_2)}$ a spanning tree with $u_1$ as the root. We can suppose that there exists at least one leaf of this tree which is not $u_2$. This leaf is of degree at least $2$ for $G_T^-\cup \ac{(u_1,u_2)}$ but of degree $1$ for the spanning tree. Removing an edge which is not in the tree does not disconnect $G_T^-\cup \ac{(u_1,u_2)}$. This contradicts the fact that $T$ is the final step. All in all, we have shown that $G_{T}$ does not have any edge, which concludes the proof of the lemma.

\subsection{Proof of Lemma \ref{lem:surjectivitygraphs}}\label{prf:surjectivitygraphs}

Let $G\in \mathcal{H}$. Let $e$ an edge that does not disconnect $G^-\cup \ac{(u_1,u_2)}$. Let $P$ be the maximal path included in $G^-$ containing $e$ and such that all the nodes inside $P$ are of degree $2$ and are not $u_1$ or $u_2$.

\medskip

We first suppose that the extremities of this path are distinct, or that they are not distinct but are a node of degree at least $4$ or that they are equal to $u_1$ or $u_2$. Then, the obtained graph $G'$ is $G\setminus P$. Let us prove that $G'\in \mathcal{H}$. There are two things to check:
\begin{enumerate}
\item Let us prove that $G'^-\cup \ac{(u_1,u_2)}$ is connected. Since the absence of $e$ does not disconnect $G^-\cup \ac{(u_1,u_2)}$, so does the absence of $P$ and we get that  $G'^-\cup \ac{(u_1,u_2)}$ is connected,
\item Let us prove that all the nodes of $G'^-\cup \ac{(u_1,u_2)}$ except $u_1$ and $u_2$ are of degree at least $2$. The nodes along the path $P$ are not nodes of $G'^-\cup \ac{(u_1,u_2)}$. The nodes that are not extremities of $P$ have the same degree for $G'^-\cup \ac{(u_1,u_2)}$ than for $G^-\cup \ac{(u_1,u_2)}$. It remains to check that the extremities of $P$ which are not $u_1$ or $u_2$ are of degree at least $2$ for $G'^-\cup \ac{(u_1,u_2)}$. Let $w,w'$ be those extremities, and let us suppose that $w\notin \ac{u_1,u_2}$. If $w\neq w'$, then the degree of $w$ decreases by $1$ and since, by maximality of $P$, it was different from $2$ for $G^-\cup \ac{(u_1,u_2)}$, it is at least $2$ for $G'^-\cup \ac{(u_1,u_2)}$. If $w= w'$ then the degree of $w$ decreases by $2$ and since it was larger than $4$ for $G^-\cup \ac{(u_1,u_2)}$, it is at least $2$ for $G'^-\cup \ac{(u_1,u_2)}$.
\end{enumerate}

Thus, we have proved that $G'\in \mathcal{H}$.

\medskip

We now suppose that the two extremities of the path are the same node $w$, different from $u_1$ or $u_2$, and is of degree $3$. Let $e'$ the unique edge of $G\setminus P$ such $w$ as an extremity of $e'$ and let $P'$ denote the maximal path containing $e'$ and that only travels through nodes of degree $2$ which are not in $\ac{u_1,u_2}$. Then, using the exact same arguments as above, it is clear that $G\setminus (P\cup P')\in \mathcal{H}$. This concludes the proof of the lemma.

\section{Proof of Theorem \ref{thm:lowdegreesparsesharp}}\label{prf:lowdegreesparsesharp}

Without loss of generality, we assume through the proof that $\sigma^2=1$. Let $D\in \N$. We recall the assumption
\begin{equation}\label{eq:definition_zeta:sparse}
\zeta:=\frac{\bar{\Delta}^4}{\rho^2 p^2}\max\pa{D^{14}, D^7n, D^7\rho^2 p, \rho^2p\frac{n}{K^2}}<1\enspace.
\end{equation}

Again, as in the proof of Theorem \ref{thm:lowdegreeclusteringsharp} in Section \ref{prf:lowdegreeclusteringsharp}, the expression of the $MMSE_{\leq D}$ can be reduced to 
$$MMSE_{\leq D}=\inf_{f\in R_D[Y]}\E\cro{\pa{f(Y)-x}^2}=\frac{1}{K}-corr_{\leq D}^2\enspace,$$
with $x=\1_{k_1^*=k_2^*}$ and $corr_{\leq D}^2$ being defined by Equation \eqref{def:corr}. We shall again upper-bound $corr_{\leq D}^2$ using Proposition \ref{thm:schrammwein}, which states that
$$corr_{\leq D}^{2}\leq \underset{|\alpha|\leq D}{\sum_{\alpha\in \N^{n\times p}}}\frac{\kappa_{x,\alpha}^2}{\alpha!}\enspace,$$
with $\kappa_{x,\alpha}=\cumul\pa{x, \pa{X_{ij}}_{ij\in \alpha}}$. Here, $\alpha$ is seen as a multiset of $[n]\times [p]$, i.e $\pa{X_{ij}}_{ij\in \alpha}$ contains $\alpha_{ij}$ copies of $X_{ij}$.
The sparse clustering model is a special case of the latent model (\ref{eq:latent-model}), with
$$Z=(k^*,z,\eps),\quad \delta_{ij}(Z)=z_j\eps_i,\quad \textrm{and}\quad \theta_{i,j}(Z)=(k^*_{i},j)\enspace .$$
Combining Proposition \ref{thm:schrammwein} and Theorem \ref{thm:LTC}, we need to upper bound, for any decomposition $\beta_1+\ldots+\beta_l=\alpha$, with $|\beta_s|=2$, the cumulant

$$C_{x,\beta_{1},\ldots,\beta_{l}} =  \cumul\pa{x,\prod_{(ij)\in \beta_1}\eps_i\1_{\Omega_{\beta_{1}}(k^*)},\ldots,\prod_{(ij)\in \beta_1}\eps_i\1_{\Omega_{\beta_{l}}(k^*)}},$$

with 
$\Omega_{\beta}(k^*):= \ac{\big|\ac{(k^*_{i},j):\enspace (i,j)\in \beta}\big|=1}\cap \ac{\forall j\in col(\beta),\enspace z_j=1}$.

Building on the recursive Bound (\ref{eq:rec-C}), we derive in Section \ref{prf:boundLTCclusteringsharp} the following upper-bound.

 \begin{lem}\label{lem:boundLTCsparsesharp}
 We recall that $\#\alpha$ stands for the number of indices  $i\in[1,n]$ such that $\alpha_{i:}\neq 0$, that $r_\alpha$ stands for the number of indices $j\in[1,p]$ such that $\alpha_{:j}\neq 0$ and that $|\alpha|:=\sum_{ij}\alpha_{ij}$.
We have
 \begin{equation*}
C_{x,\beta_1,\ldots,\beta_l}\leq\rho^{r_{\alpha}}|\alpha|^{|\alpha|-r_\alpha-\#\alpha+2}\min\pa{\pa{\frac{1}{K}}^{\#\alpha+r_{\alpha}-\frac{|\alpha|}{2}-1}, \frac{1}{K}}\enspace.
\end{equation*}
\end{lem}

Combining this bound with (\ref{eq:LTCgeneral:thm2}) and counting the number of partitions $\pi\in \mathcal{P}_2(\alpha)$ for which $C_{x,\beta_{1}(\pi),\ldots,\beta_{l}(\pi)}\neq 0$, we prove in Section \ref{prf:controlcumulantssparsesharp}  the following  upper-bound on $|\kappa_{x,\alpha}|$.

\begin{lem}\label{lem:controlcumulantssparsesharp}
Let $\alpha\in\N^{n\times p}$ non-zero. We have
$$|\kappa_{x,\alpha}|\leq \lambda^{|\alpha|}\rho^{r_\alpha}|\alpha|^{2\pa{|\alpha|-r_\alpha-\#\alpha+2}}\min\pa{\pa{\frac{1}{K}}^{\#\alpha+r_{\alpha}-\frac{|\alpha|}{2}-1}, \frac{1}{K}}\enspace. $$
\end{lem}

The last stage is to prune the multisets $\alpha$ for which $\kappa_{x,\alpha}=0$.
The next lemma gives necessary conditions for having $\kappa_{x,\alpha}\neq 0$.  For this purpose, it is convenient, as in the proof of Theorem \ref{thm:lowdegreeclusteringsharp}, to introduce a bipartite multigraph $\mathcal{G}_\alpha$ on two disjoint sets $U=\ac{u_1,\ldots, u_n}$ and $V=\ac{v_1,\ldots, v_p}$ with $\alpha_{ij}$ edges between $u_i$ and $v_j$, for any $i,j\in [n]\times [p]$. We write $\mathcal{G}_\alpha^-$ the restriction of $\mathcal{G}_\alpha$ to non-isolated nodes. We denote $U(\alpha)$ the elements of $U$ which are nodes of $\mathcal{G}_\alpha^-$ and $V(\alpha)$ the elements of $V$ which are nodes of $\mathcal{G}_\alpha^-$. We refer to Section \ref{prf:nullcumulantssparsesharp} for a proof of this lemma.

\begin{lem}\label{lem:nullcumulantssparsesharp}
Let $\alpha\in\N^{n\times p}$ be non-zero. If $\kappa_{x,\alpha}\neq 0$, then
\begin{itemize}
\item $u_1, u_2\in U(\alpha)$;
\item $\mathcal{G}^-_\alpha\cup \ac{(u_{1},u_{2})}$ is connected;
\item All the elements of $U(\alpha)$ and $V(\alpha)$ are of degree at least 2.
\end{itemize}
In particular, we have $\#\alpha\geq 2$, $|\alpha|\geq 2r_\alpha$ and $|\alpha|\geq 2\#\alpha$.
\end{lem}
\medskip

\begin{rem}
In fact, we can prove that $\mathcal{G}^-_\alpha$ is connected (see \cite{Even24}), but it is sufficient and more straightforward using Theorem \ref{thm:LTC} to prove that $\mathcal{G}^-_\alpha\cup \ac{(u_{1},u_{2})}$ is connected.
\end{rem}
\medskip

Let $d\in [2,D]$, $m,r$ such that $d\geq 2\max\pa{r,m}$. Since the conditions of Lemma \ref{lem:nullcumulantssparsesharp} are more restrictive than the ones of Lemma \ref{lem:nullcumulantsclusteringsharp}, we can apply Lemma \ref{lem:combinatorics_kappa}. Thus, there exists at most $n^{m-2}p^r d^{3(d-r-m+2)}$ matrices $\alpha$ satisfying the conditions of Lemma \ref{lem:nullcumulantssparsesharp} with $|\alpha|=d$, $\#\alpha=m$ and $r_\alpha=r$.

Then, using Proposition \ref{thm:schrammwein} together with Lemma \ref{lem:controlcumulantssparsesharp}, and pruning the terms that do not satisfy the conditions of Lemma \ref{lem:nullcumulantssparsesharp}, we get
\begin{align*}
corr_{\leq D}^2-\frac{1}{K^2}\leq& \sum_{\alpha\neq 0, \enspace |\alpha|\leq D}\kappa_\alpha^2\\
\leq & \frac{1}{K^2}\sum_{d\in [D]}\underset{2\leq m\leq d/2}{\sum_{r\leq d/2}}\pa{d^7 \lambda^2}^{d}\pa{\frac{n}{d^7}}^{m-2}\pa{\frac{\rho^2 p}{d^7}}^r \min\pa{1,\pa{\frac{1}{K^2}}^{m+r-\frac{d}{2}-2}}\enspace.
\end{align*}
Let us fix $r\geq 1$, $m\geq 2$ and $d\geq \max\pa{2r, 2m}$ and let us upper-bound the quantity $$A_{d,r,m}:=\pa{d^7 \lambda^2}^{d}\pa{\frac{n}{d^7}}^{m-2}\pa{\frac{\rho^2 p}{d^7}}^r \min\pa{1,\pa{\frac{1}{K^2}}^{m+r-\frac{d}{2}-2}}\enspace.$$ First, let us suppose that $d<2(m-2)+2r$.
Decomposing into sums of positive terms $m-2=m-2-(d/2-r)+(d/2-r)$ and $d=2r+2(d/2-r)$, we get
\begin{align*}
A_{d,r,m}=
&\pa{d^7 \lambda^2}^{d}\pa{\frac{n}{d^7}}^{m-2}\pa{\frac{\rho^2 p}{d^7}}^r \pa{\frac{1}{K^2}}^{m+r-\frac{d}{2}-2}\\
\leq&D^{7(d-m-r+2)}  \lambda^{2d}n^{m-2}\pa{\rho^2 p}^r \pa{\frac{1}{K^2}}^{m+r-\frac{d}{2}-2}\\
\leq&D^{7(d-m-r+2)}\pa{\lambda^{4}\rho^{2}p}^r\pa{\lambda^{4}n}^{d/2-r}\pa{\frac{n}{K^2}}^{m-2-(d/2-r)}\\
\leq&D^{7(d-m-r+2)} \pa{\lambda^{4}\rho^{2}p}^{r-(m-2-(d/2-r))} \pa{\frac{\lambda^{4}\rho^{2}pn}{K^2}}^{m-2-(d/2-r)} \pa{\lambda^{4}n}^{d/2-r}\\
\leq & \pa{\frac{D^7\bar{\Delta}^{4}}{p}}^{r-(m-2-(d/2-r))} \pa{\frac{\bar{\Delta}^{4}n}{pK^2}}^{m-2-(d/2-r)} \pa{\frac{D^7\bar{\Delta}^4n}{ p^2\rho^2}}^{d/2-r}\\
\leq&\zeta^{r-(m-2-(d/2-r))}\zeta^{m-2-(d/2-r)}\zeta^{d/2-r}=\zeta^{\frac{d}{2}}\enspace ,
\end{align*}
where we used in the fifth line that $\lambda^2= \bar{\Delta}^2/(p\rho)$ and the definition~\eqref{eq:definition_zeta:sparse} of $\zeta$.

On the other hand, if $d\geq 2r+2(m-2)$, we decompose $d=d-\pa{2(m-2)+2r}+2(m-2)+2r$ to get 
\begin{align*}
A_{d,r,m}=&\pa{d^7 \lambda^2}^{d}\pa{\frac{n}{d^7}}^{m-2}\pa{\frac{\rho^2 p}{d^7}}^r \min\pa{1,\pa{\frac{1}{K^2}}^{m+r-\frac{d}{2}-2}}\\
\leq &\pa{D^7 \lambda^2}^{d}\pa{\frac{n}{D^7}}^{m-2}\pa{\frac{\rho^2 p}{D^7}}^r \\
\leq &\pa{D^7\lambda^2}^{d-2r-2(m-2)}\pa{D^7\lambda^4 n}^{m-2}\pa{D^7\lambda^4 \rho^2 p}^{r}\\
\leq &\pa{D^7\frac{\bar{\Delta}^2}{p\rho}}^{d-2r-2(m-2)}\pa{D^7\frac{\bar{\Delta}^4}{p^2\rho^2}n}^{m-2}\pa{D^7\frac{\bar{\Delta}^4}{p}}^{r}\\
\leq& \sqrt{\zeta}^{d-2r-2(m-2)}\zeta^{m-2}\zeta^{r}=\zeta^{\frac{d}{2}}\enspace. 
\end{align*}

We can conclude the proof of the theorem with 
\begin{align*}
corr_{\leq D}^2\leq&\frac{1}{K^2}\pa{1+\sum_{d\in [2,D]}\underset{2\leq m\leq d/2}{\sum_{r\leq d/2}}\zeta^{\frac{d}{2}}}
\leq\frac{1}{K^2}\pa{1+\sum_{d\in [2,D]}\frac{d(d-1)}{2}\zeta^{\frac{d}{2}}}    
\leq\frac{1}{K^2}\pa{1+\frac{\zeta}{\pa{1-\sqrt{\zeta}}^3}}\enspace.
\end{align*}

\subsection{Proof of Lemma \ref{lem:boundLTCsparsesharp}.}\label{prf:boundLTCsparsesharp}

Let us fix a decomposition $\beta_1+\ldots+\beta_l=\alpha$, with $|\beta_s|=2$ and let us upper-bound $|C_{x,\beta_{1},\ldots,\beta_l}|$. For having $C_{x,\beta_{1},\ldots,\beta_l}$, it is necessary that each $\beta_s$ is supported on only one column. Thus, we can write $\beta_s=\ac{(i_s,j_s);(i'_s,j_s)}$. Henceforth, we use the convention $i_0=1$, $i'_0=2$, and $j_0=0$. For $S\subseteq [l]$, we write $\beta[S]=\ac{\beta_s, s\in S}$. Building on \eqref{eq:rec-C}, we get, for all $S\subseteq [l]$, the recursion formula 
\begin{align*}
|C_{x, \beta[S]}|\leq& \P\cro{\forall s\in \ac{0}\cup S,\enspace k^*_{i_s}=k^*_{i'_s}}\P\cro{\forall s\in S,\enspace z_{j_s}=1}+\\
&+\sum_{S'\subseteq S}|C_{x, \beta[S']}|\P\cro{\forall s\in \ac{0}\cup S',\enspace k^*_{i_s}=k^*_{i'_s}}\P\cro{\forall s\in S,\enspace z_{j_s}=1}\\
\leq& \rho^{r_{\alpha_S}}\P\cro{\forall s\in \ac{0}\cup S,\enspace k^*_{i_s}=k^*_{i'_s}}+\sum_{S'\subseteq S}|C_{x, \beta[S']}|\rho^{r_{\alpha_{S\setminus S'}}}\P\cro{\forall s\in S\setminus S',\enspace k^*_{i_s}=k^*_{i'_s}}\enspace,
\end{align*}

where, for $R\subseteq [l]$, $\alpha_R=\sum_{s\in R}\beta_R$.

Let us compute, for any subset $R\subseteq [0,l]$, the quantity $\P\cro{\forall s\in R,\enspace k^*_{i_s}=k^*_{i_s}}$. To do so, let us define, as in Section \ref{prf:boundLTCclusteringsharp}, $\mathcal{V}$ the graph on $[0,l]$ defined by: for $s,s'\geq 0$, there is an edge between $s$ and $s'$ if and only if $\ac{i_s, i'_s}\cap \ac{i_{s'}, i'_{s'}}\neq \emptyset$. For $R\subseteq [0,l]$, we write $\mathcal{V}[R]$ the restriction of $\mathcal{V}$ to $R$ and $cc(\mathcal{V}[R])$ the number of connected components of this graph. As in Section \ref{prf:boundLTCclusteringsharp}, we obtain, when $0\in R$, $\P\cro{\forall s\in R,\enspace k^*_{i_s}=k^*_{i'_s}}=\pa{\frac{1}{K}}^{|supp(\alpha_{R\setminus \ac{0}})\cup\ac{1,2}|-cc\pa{\mathcal{V}[R]}}$, and, when  $0\notin R$, $\P\cro{\forall s\in R,\enspace k^*_{i_s}=k^*_{i'_s}}=\pa{\frac{1}{K}}^{\#\alpha_{R}-cc\pa{\mathcal{V}[R]}}$. In turn, for all $S\subseteq [l]$, we have

\begin{equation}\label{eq:recursionLTCsparsesharp}
|C_{x,\beta[S]}|\leq \rho^{r_{\alpha_S}}\pa{\frac{1}{K}}^{|supp(\alpha_S)\cup\ac{1,2}|-cc(\mathcal{V}[S\cup\{0\}])}+\sum_{S'\subsetneq S}|C_{x,\beta[S']}|\rho^{r_{\alpha_{S\setminus S'}}}\pa{\frac{1}{K}}^{\#\alpha_{S \setminus S'}-cc(\mathcal{V}[S\setminus S'])}\enspace.
\end{equation}

The next lemma whose proof is postponed to Section \ref{prf:nullityLTCsparsesharp}, prunes subsets $S\subseteq [l]$ such that $C_{x,\beta[S]}= 0$. To do so, we introduce the graph $W$ on $[0,l]$ with an edge between $s,s'\in [0,l]$ if and only if $j_s=j_{s'}$ or $\ac{i_s, i'_s}\cap \ac{i_{s'}, i'_{s'}}\neq \emptyset$ (we recall that for $s=0$, we write $j_0=0$). For $S\subseteq[l]$ and $j\in \cup_{s\in S\cup\ac{0}}\ac{j_s}=col(\alpha_s)\cup \ac{0}$, we write $S_j=\ac{s\in S, j_s=j}$ (in particular $S_0=\ac{0}$).  
In the following, we denote $\mathcal{S}([l])$ the collection of subset $S$ of $[l]$ satisfying;
\begin{enumerate}
\item $\mathcal{W}[S\cup \ac{0}]$ is connected;
\item If $S\neq \emptyset$, then for all $j\in col(\alpha_S)\cup \ac{0}$, there exist $\underline{i}\neq \underline{i}'\in \cup_{s\in S_j}\ac{i_s,i'_s}$ such that both $\underline{i}$ and $\underline{i'}$ are in $\cup_{s\in S\setminus S_j}\ac{i_s,i'_s}$;
\item For all $i\in supp(\alpha_S)\setminus \ac{1,2}$, $|\pa{\alpha_S}_{i:}|\geq 2$.
\end{enumerate}

In particular, the second property implies that, as long as $S\neq \emptyset$, we have 
$\{1,2\}\subset supp(\alpha_S)$. 
\medskip

\begin{lem}\label{lem:nullityLTCsparsesharp}
For $S\subseteq [l]$, if $C_{x,\beta[S]}\neq 0$, then $S\in \mathcal{S}([l])$.
\end{lem}

Pruning the other terms in \eqref{eq:recursionLTCsparsesharp} leads to, for all $S\in \mathcal{S}([l])$, 

\begin{equation}\label{eq:recursionLTCsparsesharp3}
|C_{x,\beta[S]}|\leq \rho^{r_{\alpha_S}}\pa{\frac{1}{K}}^{|supp(\alpha_S)\cup\ac{1,2}|-cc(\mathcal{V}[S\cup\{0\}])}+\underset{S'\in \mathcal{S}([l])}{\sum_{S'\subsetneq S}}|C_{x,\beta[S']}|\rho^{r_{\alpha_{S\setminus S'}}}\pa{\frac{1}{K}}^{\#\alpha_{S\setminus S'}-cc\pa{\mathcal{V}[S\setminus S']}}\enspace.
\end{equation}

In the following, let us define recursively a function $f$ on $\mathcal{S}([l])$ satisfying, for all $S\in \mathcal{S}([l])$,

\begin{equation}\label{eq:deffunctionfsparse}
f(S)=1+\underset{S'\in \mathcal{S}([l])}{\sum_{S'\subsetneq S}}f(S')\enspace.
\end{equation}

In particular, $f(\emptyset)=1$. The next lemma, proved in  Section \ref{prf:upperboundkappaSsparse},  relies on the connectivity of $\mathcal{W}[S\cup \ac{0}]$ for  $S\in \mathcal{S}([l])$, to bound $|C_{x,\beta[S]}|$ with respect to this function $f$.
\begin{lem}\label{lem:upperboundkappaSsparse}
For all $S\in \mathcal{S}([l])$, we have $|C_{x,\beta[S]}|\leq \rho^{r_{\alpha_S}}\min\pa{\pa{\frac{1}{K}}^{\#\alpha_S+r_{\alpha_S}-\frac{|\alpha_S|}{2}-1}, \frac{1}{K}}f(S)$.
\end{lem}

It remains to upper-bound $f(S)$ for all $S\in \mathcal{S}([l])$. 

\begin{lem}\label{lem:upperboundfsparse}
For all non empty $S\in \mathcal{S}([l])$, we have $f(S)\leq |\alpha|^{|\alpha_S|-\#\alpha_S -r_{\alpha_S}+2}$.
\end{lem}

\medskip

Applying Lemma \ref{lem:upperboundfsparse} and Lemma \ref{lem:upperboundkappaSsparse} to $S=[l]$ leads to

\begin{equation*}
C_{x,\beta_1,\ldots,\beta_l}\leq \rho^{r_{\alpha}}|\alpha|^{|\alpha|-r_\alpha-\#\alpha+2}\min\pa{\pa{\frac{1}{K}}^{\#\alpha+r_{\alpha}-\frac{|\alpha|}{2}-1}, \frac{1}{K}}\enspace,
\end{equation*}

which concludes the proof of the lemma.

\subsection{Proof of Lemma \ref{lem:nullityLTCsparsesharp}}\label{prf:nullityLTCsparsesharp}

Let $S\subseteq [l]$. Let us suppose that $S\notin \mathcal{S}([l])$ and let us prove that $C_{x,\beta[S]}=0$. The set $\mathcal{S}([l])$ is an intersection of three constraints. We shall suppose that one of these constraints is not satisfied;

\begin{enumerate}

\item Let us first suppose that $\mathcal{W}[S\cup \ac{0}]$ is not connected. Let $C_1$ and $C_2$ a partition of $S\cup \ac{0}$ with no edges of $\mathcal{W}$ connecting them. We suppose by symmetry that $0\in C_1$. Then, the family of random variables $((\eps_i,k^*_i)_{i\in \cup_{s\in C_1}\ac{i_s, i'_s}}, \pa{z_j}_{j\in \cup_{s\in C_1\setminus \ac{0}}\{j_s\}})$ is independent of the family $((\eps_i,k^*_i)_{i\in \cup_{s\in C_2}\ac{i_s, i'_s}}, \pa{z_j}_{j\in \cup_{s\in C_2}\{j_s\}})$. Then, Lemma \ref{lem:independentcumulant2} implies that $C_{x,\beta[S]}=0$.

\item Let us now suppose that there exists  $j_0\in col(\alpha_S)\cup \ac{0}$ with at most one element in  $\cup_{s\in S_{j_0}}\ac{i_s,i'_s}$ which is also in $\cup_{s\in S\setminus S_{j_0}}\ac{i_s,i'_s}$. Let us denote $\underline{i}$ this element.  
Then, $(\epsilon_{\underline{i}})$ is independent of 
$(\epsilon_{i_s}\epsilon_{i'_s})_{s\in S\setminus S_{j_0}}$. Indeed, since the $\epsilon_i$'s are distributed as independent rademacher, the conditional distribution of $(\epsilon_{i_s}\epsilon_{i'_s})_{s\in S\setminus S_{j_0}}$ does not depend on the value $(\epsilon_{\underline{i}})$. Since, apart from $\epsilon_{\underline{i}}$, all the other $\epsilon_i$ with $i\in \cup_{s\in S_{j_0}}\ac{i_s,i'_s}$ do not occur in $(\epsilon_{i_s}\epsilon_{i'_s})_{s\in S\setminus S_{j_0}}$, we deduce that $(\epsilon_i)_{i\in \cup_{s\in S_{j_0}}\ac{i_s,i'_s}}$ is independent of $(\epsilon_{i_s}\epsilon_{i'_s})_{s\in S\setminus S_{j_0}}$. We have proved that 
$(\eps_{i_s}\eps_{i'_s})_{s\in S_{j_0}}$ is independent of $(\eps_{i_s}\eps_{i'_s})_{s\in S\setminus S_{j_0}}$. Arguing similarly, we get that $(z_{j_0}, (\eps_{i_s}\eps_{i'_s}\1\{k^*_{i_s}=k^*_{i'_s}\})_{s\in S_{j_0}})$ is independent of $((z_{j})_{j\neq j_0}, (\eps_{i_s}\eps_{i'_s}\1\{k^*_{i_s} = k^*_{i'_s}\})_{s\in S\setminus S_{j_0}})$. From Lemma \ref{lem:independentcumulant2}, we conclude that $C_{x,\beta[S]}=0$. 

\medskip

\item Let us finally suppose that there exists $\underline{i}\in supp(\alpha_S)\setminus \ac{1,2}$ with $|\pa{\alpha_S}_{\underline{i}:}|=1$. Let $s_0$ the unique element of $S$ such that $\underline{i}\in supp(\beta_{s_0})$; we suppose $\underline{i}=i_{s_0}$ for exemple. The random variable $\eps_{\underline{i}}$ is symmetric and independent from all the other random variables. Hence, changing $\eps_{\underline{i}}$ to $-\eps_{\underline{i}}$ does not change the joint law of all the random variables and thus, by multilinearity of the cumulant, we have 
\begin{align*}
C_{x,\beta[S]}=&\cumul\pa{x, \pa{\eps_{i_s}\eps_{i'_s}\1_{z_{j_s}\neq 0}\1_{k^*_{i_s}=k^*_{i'_s}}}_{s\in S}}\\
=& \cumul\pa{x, \pa{\eps_{i_s}\eps_{i'_s}\1_{z_{j_s}\neq 0}\1_{k^*_{i_s}=k^*_{i'_s}}}_{s\in S\setminus \{s_0\}}, -\eps_{i_{s_0}}\eps_{i'_{s_0}}\1_{z_{j_{s_0}}\neq 0}\1_{k^*_{i_{s_0}}=k^*_{i'_{s_0}}}}\\
=&-\cumul\pa{x, \pa{\eps_{i_s}\eps_{i'_s}\1_{z_{j_s}\neq 0}\1_{k^*_{i_s}=k^*_{i'_s}}}_{s\in S}}\\
=&-C_{x,\beta[S]}\enspace,
\end{align*}
which, in turn, implies that $C_{x,\beta[S]}=0$.

\end{enumerate}

\subsection{Proof of Lemma \ref{lem:upperboundkappaSsparse}}\label{prf:upperboundkappaSsparse}

Let us prove by induction that, for all $S\in \mathcal{S}([l])$, 
\[
|C_{x,\beta[S]}|\leq \rho^{r_{\alpha_S}}\min (\frac{1}{K},\pa{\frac{1}{K}}^{\#\alpha_{S}+r_{\alpha_S}-\frac{|\alpha_S|}{2}-1})f(S)\ . 
\]
The initialization is trivial since $C_{x,\beta[\emptyset]}= \cumul(x)=\frac{1}{K}$ and $|\alpha_{\emptyset}|=0$.

\medskip
\paragraph*{Induction} Let $S\neq \emptyset \in \mathcal{S}([l])$ and let us suppose that the result holds for all $S'\subsetneq S$ with $S'\in \mathcal{S}([l])$. Since $S\neq \emptyset$, 
we know from the remark below the definition of $\mathcal{S}([l])$ that $\{1,2\}\subset supp(\alpha_S)$. 

Applying Inequality \eqref{eq:recursionLTCsparsesharp3} to $S$ together with the induction hypothesis leads to 
\begin{align*}
|C_{x,\beta[S]}|\leq&\rho^{r_{\alpha_S}}\pa{\frac{1}{K}}^{\#\alpha_S-cc(\mathcal{V}[S\cup\{0\}])}+\underset{S'\in \mathcal{S}([l])}{\sum_{S'\subsetneq S}} |C_{x,\beta[S']}|\rho^{r_{\alpha_{S\setminus S'}}}\pa{\frac{1}{K}}^{\#\alpha_{S\setminus S'}-cc\pa{\mathcal{V}[S\setminus S']}}\\
\leq& \rho^{r_{\alpha_S}}\pa{\frac{1}{K}}^{\#\alpha_S-cc(\mathcal{V}[S\cup\{0\}])}\\
&+\underset{S'\in \mathcal{S}([l])}{\sum_{S'\subsetneq S}}f(S')\min\pa{\frac{1}{K}, \pa{\frac{1}{K}}^{\#\alpha_{S'}+r_{\alpha_{S'}}-\frac{|\alpha_{S'}|}{2}-1}}\pa{\frac{1}{K}}^{\#\alpha_{S\setminus S'}-cc\pa{\mathcal{V}[S\setminus S']}}\rho^{r_{\alpha_{S'}}+r_{\alpha_{S\setminus S'}}}\enspace .
\end{align*}
Let us remark that $r_{\alpha_{S'}}+r_{\alpha_{S\setminus S'}}\geq r_{\alpha_S}$. 
Since $\#\alpha_S-cc(\mathcal{V}[S\cup\{0\}])\geq 1$ and  since $\#\alpha_{S\setminus S'}-cc\pa{\mathcal{V}[S\setminus S']}\geq 0$, we directly deduce that $|C_{x,\beta[S]}|\leq \frac{1}{K}f(S)\rho^{r_{\alpha_S}}$. 

It remains to prove that $|C_{x,\beta[S]}|\leq \rho^{r_{\alpha_S}}f(S)\pa{\frac{1}{K}}^{\#\alpha_S+r_{\alpha_{S}}-\frac{|\alpha_S|}{2}-1}$. Let us isolate  the term $S'=\emptyset$  in the sum.  
\begin{align*}
\rho^{-r_{\alpha_S}}|C_{x,\beta[S]}|\leq& \pa{\frac{1}{K}}^{\#\alpha_S-cc(\mathcal{V}[S\cup\{0\}])}+\underset{S'\in \mathcal{S}([l])}{\sum_{\emptyset\neq S'\subsetneq S}}f(S')\pa{\frac{1}{K}}^{\#\alpha_{S'}+r_{\alpha_{S'}}-\frac{|\alpha_{S'}|}{2}-1+\#\alpha_{S\setminus S'}-cc\pa{\mathcal{V}[S\setminus S']}}\\
&+\frac{1}{K}\pa{\frac{1}{K}}^{\#\alpha_{S }-cc\pa{\mathcal{V}[S]}}\enspace.
\end{align*}

The next lemma uses the connectivity of the graph $\mathcal{W}[S\cup \ac{0}]$ in order to lower bound the exposants of the above inequality. We refer to Section \ref{prf:connectivityLTCsparsesharp} for its proof.
\begin{lem}\label{lem:connectivityLTCsparsesharp}
For any subset $R\subseteq [0,l]$, $cc\pa{\mathcal{W}[R]}-cc\pa{\mathcal{V}[R]}\geq r_{\alpha_{R}}-\frac{|\alpha_{R}|}{2}$.
\end{lem}

Applying Lemma \ref{lem:connectivityLTCsparsesharp} together with the fact that $\mathcal{W}[S\cup \ac{0}]$ is connected leads us to
\begin{align*}
\frac{|C_{x,\beta[S]}|}{\rho^{\alpha_R}}\leq &\pa{\frac{1}{K}}^{\#\alpha_S+r_{\alpha_S}-\frac{|\alpha_S|}{2}-1}\\ &+\underset{\emptyset\neq S'\in \mathcal{S}([l])}{\sum_{S'\subsetneq S}}f(S')\pa{\frac{1}{K}}^{\#\alpha_{S'}+r_{\alpha_{S'}}-\frac{|\alpha_{S'}|}{2}-1+\#\alpha_{S\setminus S'}+r_{\alpha_{S\setminus S'}}-\frac{|\alpha_{S\setminus S'}|}{2}-cc\pa{\mathcal{W}[S\setminus S']}}\\
&+\pa{\frac{1}{K}}^{1+\#\alpha_{S}+r_{\alpha_{S}}-\frac{|\alpha_{S}|}{2}-cc\pa{\mathcal{W}[S]}}\enspace.
\end{align*}
The next lemma uses again the connectivity of $\mathcal{W}[S\cup \ac{0}]$ in order to lower-bound the exposants in the sum above. We refer to Section \ref{prf:connectivityLTCsparsesharp2} for its proof.

\medskip
\begin{lem}\label{lem:connectivityLTCsparsesharp2}
For any subset $S'\subseteq S$ that both belonf to $\mathcal{S}([l])$ and such that $S'\neq \emptyset$, we have  
$$\#\alpha_{S'}+r_{\alpha_{S'}}+\#\alpha_{S\setminus S'}+r_{\alpha_{S\setminus S'}}-cc\pa{\mathcal{W}[S\setminus S']}\geq \#\alpha_S+r_{\alpha_S}\enspace.$$
\end{lem}
\medskip

For the term $S'=\emptyset$, we use the fact that $cc\pa{\mathcal{W}[S]}\leq 2$ to get the desired  inequality

$$\pa{\frac{1}{K}}^{1+\#\alpha_{S}+r_{\alpha_{S}}-\frac{|\alpha_{S}|}{2}-cc\pa{\mathcal{W}[S]}}\leq  \pa{\frac{1}{K}}^{\#\alpha_S+r_{\alpha_S}-\frac{|\alpha_S|}{2}-1}\enspace.$$

We deduce from this and Lemma \ref{lem:connectivityLTCsparsesharp2} that

\begin{align*}
\frac{|C_{x,\beta[S]}|}{\rho^{\alpha_R}}\leq & \pa{\frac{1}{K}}^{\#\alpha_S+r_{\alpha_S}-\frac{|\alpha|}{2}-1}\pa{1+\underset{S'\in \mathcal{S}([l])}{\sum_{S'\subsetneq S}}f(S')}\\
\leq & f(S)\pa{\frac{1}{K}}^{\#\alpha_S+r_{\alpha_S}-\frac{|\alpha|}{2}-1}\enspace,
\end{align*}

which concludes the induction and the proof of the lemma.

\subsection{Proof of Lemma \ref{lem:connectivityLTCsparsesharp}}\label{prf:connectivityLTCsparsesharp}

Let us fix $R\subseteq [0,l]$. We need to prove that 
$$|\alpha_{R}|\geq 2\pa{r_{\alpha_R}-cc\pa{\mathcal{W}[R]}+cc\pa{\mathcal{V}[R]}}\enspace.$$
By definition, $\mathcal{V}(R)$ is a subgraph of $\mathcal{W}_r$. Hence, 
to show this inequality, it is sufficient to prove it for all connected components of $\mathcal{W}[R]$. We can therefore suppose without loss of generality that $\mathcal{W}[R]$ is connected.

Denote $q=cc(\mathcal{V}[R])$ and let us write $cc_1,\ldots, cc_q$ the collection of the connected components of $\mathcal{V}[R]$. Since the graph $\mathcal{W}[R]$ is connected, we can, up to a reordering of the $cc_l$'s,  suppose that, for all $q'\in [2,q]$, $\mathcal{W}[R]$ has an edge connecting $cc_{q'}$ to $\pa{\cup_{q''<q'}cc_{q''}}$. In the following, for all $q'\in [q]$, we write $\alpha_{cc_{q'}}=\sum_{s\in cc_{q'}}\beta_s$ and $\alpha^{(q')}=\sum_{q''\leq q'}\alpha_{cc_{q''}}$. Let us prove by induction over $q'\in [q]$ that 
$$|\alpha^{(q')}|\geq 2r_{\alpha^{(q')}}+2(q'-1)\enspace.$$

\paragraph*{Initialization} For all $j\in col(\alpha^{(1)})$, there exists $s\in cc_1\setminus \ac{0}$ such that $j_s=j$ (we recall that $\beta_s=\ac{(i_s,j_s);(i'_s,j_s)}$). Thus, $|\pa{\alpha^{(1)}}_{:j}|\geq |\beta_s|=2$. We deduce  that  $|\alpha^{(1)}|\geq 2r_{\alpha^{(1)}}$.

\paragraph*{Induction} Let us suppose that the result holds for some $q'\in [q-1]$ and let us prove that it still holds for $q'+1$. As for the initalisation, we have $|\alpha_{cc_{q'+1}}|\geq 2r_{\alpha_{cc_{q'+1}}}$. Since $\mathcal{W}$
has an edge connecting $cc_{q'+1}$ to $\pa{\cup_{q''\leq q'}cc_{q''}}$ whereas $\mathcal{V}$ does not have any, we know that $col(\alpha_{cc_{q'+1}})$ intersects $col(\alpha^{(q')})$. Together with the induction hypothesis, this implies that
$$|\alpha^{(q'+1)}|\geq 2r_{\alpha^{(q')}}+2(q'-1)+2r_{\alpha_{cc_{q'+1}}}\geq 2r_{\alpha^{(q'+1)}}+2q'\enspace,$$
and concludes the induction.

\subsection{Proof of Lemma \ref{lem:connectivityLTCsparsesharp2}}\label{prf:connectivityLTCsparsesharp2}

Let $  S'\subseteq S$ be a non-empty set. Since $S'\in \mathcal{S}([l])$ is non-empty, we know that  $\ac{1,2}\subset supp(\alpha_{S'})$. Let us prove the inequality
$$\#\alpha_{S'}+r_{\alpha_{S'}}+\#\alpha_{S\setminus S'}+r_{\alpha_{S\setminus S'}}\geq \#\alpha_S+r_{\alpha_S}+cc\pa{\mathcal{W}[S\setminus S']}\enspace.$$

Let $q=cc\pa{\mathcal{W}[S\setminus S']}$ and let us write $cc_1,\ldots, cc_q$ the collection of the connected components of $\mathcal{W}[S\setminus S']$. Since the graph $\mathcal{W}[S\setminus S']$ does not have any edge between the $cc_{q'}$'s, we have $\#\alpha_{S\setminus S'}=\sum_{q'\in [q]}\#\alpha_{cc_{q'}}$ and 
$r_{\alpha_{S\setminus S'}}=\sum_{q'\in [q]}r_{\alpha_{cc_{q'}}}$.

Since the graph $\mathcal{W}[S\cup \{0\}]$ is connected, for all $q'\in [q]$, $\mathcal{W}[S\cup \{0\}]$ has an edge connecting $cc_{q'}$ to $S'\cup\{0\}$. Thus, for all $q'\in [q]$, either $supp(\alpha_{cc_{q'}})$ intersects $supp(\alpha_{S'})\cup \ac{1,2}=supp(\alpha_{S'})$, or that  $col(\alpha_{cc_{q'}})$ intersects $col(\alpha_{S'})$. We deduce that 
$$\#\alpha_{cc_{q'}}+r_{\alpha_{cc_{q'}}}\geq 1+|supp(\alpha_{cc_{q'}})\setminus supp(\alpha_{S'})|+|col(\alpha_{cc_{q'}})\setminus col(\alpha_{S'})|\enspace.$$
Gathering everything, we get 
\begin{align*}
\lefteqn{\#\alpha_{S'}+r_{\alpha_{S'}}+\#\alpha_{S\setminus S'}+r_{\alpha_{S\setminus S'}}}&\\ =&\#\alpha_{S'}+r_{\alpha_{S'}}+\sum_{q'\in [q]}\#\alpha_{cc_{q'}}+\sum_{q'\in [q]}r_{\alpha_{cc_{q'}}}\\
\geq & \#\alpha_{S'}+r_{\alpha_{S'}}+\sum_{q'\in [q]}1+|supp(\alpha_{cc_{q'}})\setminus supp(\alpha_{S'})|+|col(\alpha_{cc_{q'}})\setminus col(\alpha_{S'})|\\
\geq &\#\alpha_S+r_{\alpha_S}+cc\pa{\mathcal{W}[S\setminus S']}\enspace,
\end{align*}

which concludes the proof of the lemma.

\subsection{Proof of Lemma \ref{lem:upperboundfsparse}}\label{prf:upperboundfsparse}

We proceed by induction on $S\in \mathcal{S}([l])$ to prove that  
\begin{equation}\label{eq:upper_f_S_a}
f(S)\leq |\alpha|^{|\alpha_S|-|supp(\alpha_S)\cup \ac{1,2}|-r_{\alpha_S}+2}\enspace.
\end{equation}
Since when $S\neq \emptyset$, we have $|supp(\alpha_S)\cup \ac{1,2}|= |supp(\alpha_S)|=\#\alpha_s$ --see the remark below the definition of $\mathcal{S}([l])$, this is sufficient for our purpose.

The initialization is trivial since $f(\emptyset)=1$ and $\alpha_{\emptyset}=0$. Let us take $S\in  \mathcal{S}([l])$ and let us suppose that the result holds for all $S'\subsetneq S$. For all $s\in S$, denote $S^*(s)$  the maximal element of $\mathcal{S}([l])$ which is included in $S\setminus \ac{s}$. The existence of such an element in justified by the fact that the set of elements $S'\in \mathcal{S}([l])$ with $S'\subseteq S\setminus \ac{s}$ is not empty (it contains $\emptyset$) and is stable by union. Then, we have 
\begin{align*}
f(S)=&1+\underset{S'\in \mathcal{S}([l])}{\sum_{S'\subsetneq S}}f(S')
\leq  1+\sum_{s\in S}\underset{S'\in \mathcal{S}([l])}{\sum_{S'\subseteq S^*(s)}}f(S')\\
\leq & 1+\sum_{s\in S} \cro{2f(S^*(s))-1}\\
\leq & 2\sum_{s\in S}f(S^*(s))\ , 
\end{align*}
where we used the recursive definition of $f$ in the second row. Applying the induction hypothesis leads us to 
\begin{equation}\label{eq:upper_f_S}
f(S)\leq 2\sum_{s\in S}|\alpha|^{|\alpha_{S^*(s)}|-|supp(\alpha_{S^*(s)})\cup \ac{1,2}|-r_{\alpha_{S^*(s)}}+2}\enspace.
\end{equation}

Let us fix $s_0\in S\setminus S^*(s_0)$. We denote  in the following $S^-=S\setminus S^*(s_0)$. Besides, let $S^-_0$ denote the set of all $s\in S\setminus S^*(s_0)$ such that $j_s\in col(\alpha_{S^*(s_0)})$. Then, we write $j^1,\ldots, j^r\in [p]$ the other columns of $S\setminus S^*(s_0)$ that do not appear in $col(\alpha_{S^*(s_0)})$. Besides, for $r'=1,\ldots, r$, we write $S^-_{r'}$ the set of $s\in S^-$ such that $j_s=j^{r'}$. Since the graph $\mathcal{W}[S\cup \ac{0}]$ is connected, we can suppose, up to some reordering, that for all $r'\leq r$, there exists $i\in supp(\alpha_{S^-_{r'}})\cap \pa{supp\pa{\alpha_{S^*(s_0)}+\sum_{0\leq r''\leq r'-1}\alpha_{S^-_{r''}}}\cup \ac{1,2}}$. Moreover, at the final step $r'=r$, from Lemma \ref{lem:nullityLTCsparsesharp}, we know that there exist two distinct $i,i'\in  supp(\alpha_{S^-_{r}})\cap \pa{supp\pa{\alpha_{S^*(s_0)}+\sum_{0\leq r''\leq r-1}\alpha_{S^-_{r''}}}\cup \ac{1,2}}$. If $r\neq 0$, those two claims imply that $|\alpha_{S^-}|\geq r+1+|supp(\alpha_{S^-})\setminus\pa{ supp(\alpha_{S^*(s_0)})\cup \ac{1,2}}|$. Since $\ac{1,2}\subseteq supp(\alpha_S)$, we derive from the latter that
$$|\alpha_S|-|supp(\alpha_S)\cup \ac{1,2}|-r_{\alpha_S}-1\geq |\alpha_{S^*(s_0)}|-|supp(\alpha_{S^*(s_0)})\cup \ac{1,2}|-r_{\alpha_{S^*(s_0)}}\enspace.$$

If $r=0$, we use the fact that all $i\in supp(\alpha_S)$ must satisfy $|\pa{\alpha_S}_{i:}|\geq 2$ --see Lemma \ref{lem:nullityLTCsparsesharp}. This implies that 
\begin{align*}
|\alpha_{S^-}|& \geq \min\pa{2, 2\left|supp(\alpha_{S^-})\setminus \pa{supp(\alpha_{S^*(s_0)})\cup \ac{1,2}}\right|}\\ &\geq \left|supp(\alpha_{S^-})\setminus \pa{supp(\alpha_{S^*(s_0)})\cup \ac{1,2}}\right|+1 \enspace . 
\end{align*}
Hence, as in the case $r\neq 0$, we also get in the case $r=0$ that 
$$|\alpha_S|-|supp(\alpha_S)\cup \ac{1,2}|-r_{\alpha_S}-1\geq |\alpha_{S^*(s_0)}|-|supp(\alpha_{S^*(s_0)})\cup \ac{1,2}|-r_{\alpha_{S^*(s_0)}}\enspace.$$

Hence, we deduce from~\eqref{eq:upper_f_S} that
$$f(S)\leq \frac{2|S|}{|\alpha|}{|\alpha|}^{|\alpha_S|-|supp(\alpha_S)\cup \ac{1,2}|-r_{\alpha_S}+2}\leq {|\alpha|}^{|\alpha_S|-|supp(\alpha_S)\cup \ac{1,2}|-r_{\alpha_S}+2}\enspace\ . $$
We have shown~\eqref{eq:upper_f_S_a}, which concludes the proof.

\subsection{Proof of Lemma \ref{lem:controlcumulantssparsesharp}}\label{prf:controlcumulantssparsesharp}

In order to get a suitable upper bound of $|\kappa_{x,\alpha}|$ from Theorem \ref{thm:LTC} and Lemma \ref{lem:boundLTCsparsesharp}, it is sufficient to upper-bound the number of partitions $\pi=\pi_1,\ldots,\pi_l$ of the multisets $\alpha$ into groups of size $2$ such that $C_{x,\beta_1(\pi),\ldots,\beta_l(\pi)}\neq 0$, where $\pa{\beta_s(\pi)}_{ij}$ counts the number of copies of $(i,j)$ in $\pi_s$. Remember that for such a partition, it is necessary that, for $s\in[l]$, $\pi_s$ is contained in a single column of $\alpha$. The number of partitions into pairings of each multiset $\alpha_{:j}$ is at most $|\alpha_{:j}|^{\frac{|\alpha_{:j}|}{2}-1}$. We deduce that the total number of satisfying partitions is at most $|\alpha|^{\frac{|\alpha|}{2}-r_\alpha}$. Since $|\alpha|\geq 2\#\alpha$, we upper-bound this quantity by $ |\alpha|^{|\alpha|-\#\alpha-r_\alpha}$.

Plugging this with Lemma \ref{lem:boundLTCsparsesharp} in Theorem \ref{thm:LTC} leads us to 
$$|\kappa_{x,\alpha}|\leq \lambda^{|\alpha|}\rho^{r_\alpha}|\alpha|^{2\pa{|\alpha|-r_\alpha-\#\alpha+2}}\min\pa{\pa{\frac{1}{K}}^{\#\alpha+r_{\alpha}-\frac{|\alpha|}{2}-1}, \frac{1}{K}}\enspace, $$

which concludes the proof of the lemma.

\subsection{Proof of Lemma \ref{lem:nullcumulantssparsesharp}}\label{prf:nullcumulantssparsesharp}

Let $\alpha\in \N^{n\times p}\neq 0$. We shall suppose successively that one of the conditions of Lemma \ref{lem:nullcumulantssparsesharp} is not satisfied and and prove that $\kappa_{x,\alpha}=0$.

\begin{enumerate}

\item We suppose that either $\alpha_{1:}=0$ or $\alpha_{2:}=0$. By symmetry, we suppose that $\alpha_{2:}=0$. Then, the label $k_2^*$ is independent of the random variables $\pa{X_{ij}}_{ij\in \alpha}\cup \ac{k^*_1}$. And since $k^*_2$ follows a uniform law on $[K]$, we directly deduce that $x=\1_{k_1^*=k_2^*}$ is also independent from $\pa{X_{ij}}_{ij\in \alpha}\cup \ac{k^*_1}$ (and a fortiori from $\pa{X_{ij}}_{ij\in \alpha}$). By Lemma \ref{lem:independentcumulant2}, we have $\kappa_{x,\alpha}=0$.

\item We suppose that there exists $i_{0}\in supp(\alpha)$ such that $\sum_{j\in [p]}\alpha_{i_0j}= 1$ and we shall prove that $\kappa_{x,\alpha}=0$. $\eps_{i_0}$ is symmetric and independent from the other random variables. In particular, $x,\pa{X_{ij}}_{ij\in \alpha}$ has the same distribution as $x,\pa{\pa{-1}^{\1_{i=i_0}}X_{ij}}_{ij\in \alpha}$ and so $\kappa_{x,\alpha}=-\kappa_{x,\alpha}$. We deduce $\kappa_{x,\alpha}=0$.

\item We suppose that there exists $j_{0}\in [p]$ such that $\sum_{i\in [n]}\alpha_{ij_0}= 1$. It is clear in that case that there does not exist any decomposition $\alpha=\beta_1+\ldots+\beta_l$ with $\beta_s=\ac{(i_s,j_s);(i'_s,j_s)}$. Hence, Theorem \ref{thm:LTC} ensures that $\kappa_{x,\alpha}=0$.

\item Let us suppose that the graph $\mathcal{G}_\alpha^-\cup \ac{(u_1,u_2)}$ is not connected. Let $\beta_1+\ldots+\beta_l=\alpha$ with $\beta_s=\ac{(i_s,j_s);(i'_s,j_s)}$. Let us prove that $C_{x,\beta_1,\ldots, \beta_l}$ is null. The fact that $\mathcal{G}_\alpha^-\cup \ac{(u_1,u_2)}$ is disconnected implies that the graph $\mathcal{W}$ of $[0,l]$ defined is Section \ref{prf:boundLTCsparsesharp} is also disconnected. We deduce from Lemma \ref{lem:nullityLTCsparsesharp} that $C_{x,\beta_1,\ldots,\beta_l}$ is null. This being true for all decompositions, we conclude that $\kappa_{x,\alpha}=0$.

\end{enumerate}

\section{Proof of Theorem \ref{thm:lowdegreebi}}\label{prf:lowdegreebi}

Theorem~\ref{thm:lowdegreebi} states two LD lower bounds~\eqref{eq:LB_MMSE_D_SC1} and~\eqref{eq:LB_MMSE_D_SC2} in two different regimes. We prove them separately in Subsections~\ref{prf:lowerboundbi2} and {prf:lowerboundbi1}.

\subsection{Reduction to a $L$-dimensional problem: Proof of~\eqref{eq:LB_MMSE_D_SC2}}\label{prf:lowerboundbi2}

Without loss of generality, we suppose that $\sigma^2=1$. Let us fix $D\in \N$ and let us suppose that $$\zeta':=\lambda^4D^{10}\frac{5p^2}{L}\max\pa{1,\frac{n}{K^2}}<1\enspace.$$
Working conditionally on $l^*, \eps^r,\eps^c$, we get
\begin{align*}
MMSE_{\leq D}=&\inf_{f\in \R_{D}[Y]}\E\cro{\pa{f(Y)-x}^2}\\
=&\inf_{f\in \R_{D}[Y]}\E_{l^*}\cro{\E\cro{\pa{f(Y)-x}^2|l^*, \eps^r, \eps^c}}\\
\geq& \mathbb{E}_{l^*,\eps^r,\eps^c}\cro{\inf_{f\in \R_{D}[Y]}\E\cro{\pa{f(Y)-x}^2|l^*, \eps^r, \eps^c}}\enspace.
\end{align*}

In the following, we fix $l^*, \eps^r, \eps^c$ and we consider 
\[MMSE_{\leq D}(l^*, \eps^r, \eps^c)=\inf_{f\in \R_{D}[Y]}\E\cro{\pa{f(Y)-x}^2|l^*, \eps^r, \eps^c}\enspace .
\]
 Let us suppose that, for all $l\in [L]$, we have, $\left|\ac{j\in [p], l_j^*=l}\right|\leq 5\frac{p}{L}$. We write $$MMSE_{\leq D}(l^*, \eps^r, \eps^c)=\frac{1}{K}-corr_{\leq D}^2(l^*, \eps^r, \eps^c)\enspace.$$

We shall upper-bound $corr_{\leq D}^2(l^*, \eps^r, \eps^c)$ using Proposition \ref{thm:schrammwein} which states that $$corr_{\leq D}^2(l^*, \eps^r, \eps^c)\leq \sum_{\alpha\in \N^{n\times p}}\frac{\kappa_{x,\alpha}(l^*, \eps^r, \eps^c)^2}{\alpha!}\enspace,$$

where $\kappa_{x,\alpha}(l^*, \eps^r, \eps^c)=\cumul\pa{x, \pa{X_{ij}}_{ij\in \alpha}| l^*, \eps^r, \eps^c}$, where we see $\alpha$ as a multiset of $[n]\times [p]$. This conditional biclustering model is a special case of the latent model \eqref{eq:latent-model} with  
$$Z=(k^*, l^*, \eps^r,\eps^c), \quad \delta_{ij}(k^*)=\eps^r_i\eps^c_j,\quad \textrm{and}\quad \theta_{i,j}(k^*)=(k^*_{i},l^*_j)\enspace,$$

where $l^*, \eps^r, \eps^c$ are considered as deterministic. Combining Proposition \ref{thm:schrammwein} and Theorem \ref{thm:LTC}, we need to upper-bound, for any multiset $\alpha$ and any decomposition $\alpha=\beta_1+\ldots+\beta_l$ with $|\beta_s|=2$, the cumulant 
$$C_{x,\beta_1,\ldots, \beta_l}(l^*, \eps^r, \eps^c)=\cumul\pa{x,\prod_{ij\in \beta_1}\eps^r_i\eps^c_j\1_{\Omega_{\beta_1}(k^*,l^*)},\ldots, \prod_{ij\in \beta_l}\eps^r_i\eps^c_j\1_{\Omega_{\beta_l}(k^*,l^*)} \Big|\enspace l^*, \eps^r, \eps^c}\enspace,$$

with $\Omega_{\beta}(k^*,l^*):= \ac{\big|\ac{(k^*_{i},l^*_j):\enspace (i,j)\in \beta}\big|=1}$. Building on Lemma \ref{lem:boundLTCclusteringsharp}, we derive the following upper-bound, whose proof is postponed to the end of the section.

\begin{lem}\label{lem:boundLTCbi2}
 We recall that $\#\alpha$ stands for the cardinality of the points $i\in[n]$ such that $\alpha_{i:}\neq 0$ and that $|\alpha|:=\sum_{ij}\alpha_{ij}$. We have, for all $l^*$, $\eps^r$, $\eps^c$, that 
$$|C_{x,\beta_1,\ldots,\beta_l}(l^*, \eps^r, \eps^c)|\leq |\alpha|^{|\alpha|}\pa{\frac{1}{K}}^{\#\alpha-1}\enspace.$$

 \end{lem}

The number of partition of the multiset $\alpha$ into groups of size $2$ is at most $|\alpha|^{\frac{|\alpha|}{2}-1}\leq |\alpha|^{\frac{|\alpha|}{2}}$. Combining this with Lemma \ref{lem:boundLTCbi2} and Theorem \ref{thm:LTC} leads us to 
\begin{equation}\label{eq:upper_kappa:bi_condition}
|\kappa_{x,\alpha}(l^*, \eps^r, \eps^c)|\leq |\alpha|^{\frac{|\alpha|}{2}} |\alpha|^{|\alpha|}\pa{\frac{1}{K}}^{\#\alpha-1}\enspace.
\end{equation}

Then, we prune the multisets $\alpha$ for which $\kappa_{x,\alpha}(l^*, \eps^r, \eps^c)=0$.

\begin{lem}\label{lem:nullcumulantsbisharp2}
Let $\alpha\in\N^{n\times p}$ be non-zero. If $\kappa_{x,\alpha}(l^*, \eps^r, \eps^c)\neq 0$, then: 
\begin{enumerate}
\item $1,2\in supp(\alpha)$;
\item All the elements $i\in supp(\alpha)\setminus \ac{1,2}$ are such that $|\alpha_{i:}|\geq 2$;
\item There exists a decomposition $\alpha=\beta_1+\ldots+\beta_l$, where $\beta_s=\ac{(i_s, j_s), (i'_s, j'_s)}$, and such that, for all $s\in [l]$, $l^*_{j_s}=l^*_{j'_s}$.
\end{enumerate}
In particular, we have $\#\alpha\geq 2$ and $|\alpha|\geq 2\#\alpha-2$.
\end{lem}

The last step of the proof amounts to counting the number of $\alpha$'s satisfying the conditions of Lemma \ref{lem:nullcumulantsbisharp2}.
\begin{lem}\label{lem:combinatorixbi2}
Suppose that, for all $l\in [L]$, we have $\left|\ac{j\in [p], l_j^*=l}\right|\leq 5\frac{p}{L}$. Let $d\in [D]$ and $m\in [2, \frac{d+2}{2}]$. Then, there are at most $d^{2d}n^{m-2}\pa{5\frac{p^2}{L}}^{\frac{d}{2}}$ matrices $\alpha$ satisfying the conditions of Lemma \ref{lem:nullcumulantsbisharp2} with $|\alpha|=d$ and $\#\alpha=m$.
\end{lem}

Combining Lemma \ref{lem:combinatorixbi2} and~\eqref{eq:upper_kappa:bi_condition}, and supposing that, for all $l\in [L]$, we have $\left|\ac{j\in [p], l_j^*=l}\right|\leq 5\frac{p}{L}$, we end up with

\begin{align*}
corr^2_{\leq D}-\frac{1}{K^2}\leq& \frac{1}{K^2}\sum_{d=1}^{D}\sum_{m\in [2, \frac{d+2}{2}]}d^{5d}\lambda^{2d}\pa{\frac{n}{K^2}}^{m-2}\pa{5\frac{p^2}{L}}^{\frac{d}{2}}\\
\leq& \frac{1}{K^2}\sum_{d=1}^{D}\sum_{m\in [2, \frac{d+2}{2}]}\pa{\sqrt{\frac{5p^2}{L}}\lambda^2 d^5}^d \pa{\frac{n}{K^2}}^{m-2}\\
\leq& \frac{1}{K^2}\sum_{d=1}^{D}\sum_{m\in [2, \frac{d+2}{2}]}\pa{\sqrt{\frac{5p^2}{L}}\lambda^2 D^5}^{d-2(m-2)}\pa{\lambda^4 D^{10}\frac{5p^2 n}{LK^2}}^{m-2}\\
\leq& \frac{1}{K^2}\sum_{d=1}^{D}\frac{d}{2}\zeta'^{d/2}\\
\leq& \frac{1}{K^2}\frac{\sqrt{\zeta'}}{(1-\sqrt{\zeta'})^2}\enspace.
\end{align*}

Hence, provided that ,for all $l\in [L]$, $\left|\ac{j\in [p], l_j^*=l}\right|\leq 5\frac{p}{L}$, we have $$MMSE_{\leq D}(l^*,\eps^r,\eps_c)\geq \frac{1}{K}-\frac{1}{K^2}\frac{\sqrt{\zeta'}}{(1-\sqrt{\zeta'})^2}\enspace.$$

Moreover, using a large deviation Inequality for Binomial random variables --see e.g Exercise 12.9.7 of \cite{HDS2}--, we have that, with probability at least $1-L\exp\pa{-\frac{5p}{2L}\log(5)}$, for all $l\in [L]$, $\left|\ac{j\in [p], l_j^*=l}\right|\leq 5\frac{p}{L}$. We deduce from this that 

$$MMSE_{\leq D}\geq \pa{1-L\exp\pa{-\frac{5p}{2L}\log(5)}}\pa{\frac{1}{K}-\frac{1}{K^2}\frac{\sqrt{\zeta'}}{(1-\sqrt{\zeta'})^2}}\enspace.$$

\begin{proof}[Proof of Lemma \ref{lem:boundLTCbi2}]

Let $\beta_1,\ldots,\beta_l$ such that $|\beta_s|=2$ for $s\in [l]$ and such that $\beta_1+\ldots+\beta_l=\alpha$ and let $l^*\in [L]^p$. We seek to upper-bound 

$$C_{x,\beta_1,\ldots, \beta_l}(l^*, \eps^r, \eps^c)=\cumul\pa{x,\prod_{ij\in \beta_1}\eps^r_i\eps^c_j\1_{\Omega_{\beta_1}(k^*,l^*)},\ldots, \prod_{ij\in \beta_l}\eps^r_i\eps^c_j\1_{\Omega_{\beta_l}(k^*,l^*)} \Big|\enspace l^*}\enspace,$$

with $\Omega_{\beta}(k^*,l^*):= \ac{\big|\ac{(k^*_{i},l^*_j):\enspace (i,j)\in \beta}\big|=1}$. For $C_{x,\beta_{1},\ldots,\beta_{l}}(l^*)$ to be non-zero, it is necessary that, for all $s\in [l]$, $\Omega_{\beta_s}(k^*,l^*)$ is an event of positive probability conditionally on $l^*$. This condition enforces that $|\ac{l^*_j}_{j\in col(\beta_s)}|=1$ for all $s$. We can assume that the latter property is true in the followsing.  
We write $\beta_s=\ac{(i_s,j_s);(i'_s,j'_s)}$, for $s\in [l]$, with $l^*_{j_s}=l^*_{j'_s}$. We also take the convention $i_0=1$, $i'_0=2$, and $j_0=0$. We then have

$$C_{x,\beta_{1},\ldots,\beta_{l}}(l^*, \eps^r, \eps^c) =  \pa{\prod_{s\in [l]}\prod_{ij\in \beta_s}\eps^r_i\eps^c_j}\cumul\pa{\1\ac{k^*_{i_s}=k^*_{i'_s}}_{s\in [0,l]}}\enspace,$$

which, in turn, implies 

$$\left|C_{x,\beta_{1},\ldots,\beta_{l}}(l^*, \eps^r, \eps^c) \right|= \left|\cumul\pa{\1\ac{k^*_{i_s}=k^*_{i'_s}}_{s\in [0,l]}}\right|\enspace.$$

From Lemma \ref{lem:boundLTCclusteringsharp}, we deduce

\begin{equation*}|C_{x,\beta_1,\ldots,\beta_l}(l^*, \eps^r, \eps^c)|\leq |\alpha|^{|\alpha|}\pa{\frac{1}{K}}^{\#\alpha-1}\enspace,
\end{equation*}

which concludes the proof of the lemma.

\end{proof}

\begin{proof}[Proof of Lemma \ref{lem:nullcumulantsbisharp2}]
\begin{enumerate}
\item By symmetry, let us suppose that $1\notin supp(\alpha)$. Then, conditionally on $l^*,\eps^r, \eps^c$, $x$ is independent from $\pa{X_{ij}}_{ij\in \alpha}$. We deduce from Lemma \ref{lem:independentcumulant2} that $\kappa_{x,\alpha}(l^*, \eps^r, \eps^c)= 0$.
\item Let us suppose that there exists $\underline{i}\in supp(\alpha)\setminus \ac{1,2}$ with $|\alpha_{\underline{i}:}|=1$. Consiser any decomposition $\alpha=\beta_1+\ldots+\beta_l$ with $\beta_s=\ac{(i_s,j_s);(i'_s, j'_s)}$. Let $s_0$ be the only element such that $\underline{i}\in supp(\beta_{s_0})$. It follows  that $\1\{k^*_{i_{s_0}}=k^*_{i'_{s_0}}\}$ is independent of  $(x, (\1\{k^*_{i_{s}}=k^*_{i'_{s}}\})_{s\in [l]\setminus \ac{s_0}})$ and thus, from Lemma \ref{lem:independentcumulant2}, we deduce $C_{x,\beta_1,\ldots, \beta_l}(l^*, \eps^r, \eps^c)=0$. This being true for all decompositions of $\alpha$, we have $\kappa_{x,\alpha}(l^*, \eps^r, \eps^c)= 0$.
\item The last point of the lemma is a direct consequence of Theorem \ref{thm:LTC}. 
\end{enumerate}
\end{proof}

\begin{proof}[Proof of Lemma \ref{lem:combinatorixbi2}]
Since we necessarily have $1,2\in supp(\alpha)$, there are at most $n^{m-2}$ possibilities for choosing $supp(\alpha)$. Using the third point of Lemma \ref{lem:nullcumulantsbisharp2} together with the hypothesis $\left|\ac{j\in [p], l_j^*=l}\right|\leq 5\frac{p}{L}$, for all $l\in [L]$, we deduce that the number of possibilities for choosing $col(\alpha)$ is at most $\pa{\frac{5p^2}{L}}^{d/2}$. Finally, there are at most $m^d |col(\alpha)|^d\leq d^{2d}$ possibilities for choosing $\alpha$ once $supp(\alpha)$ and $col(\alpha)$ is determined. This concludes the proof of the lemma.
\end{proof}

\subsection{Proof of the first lower bound \eqref{eq:LB_MMSE_D_SC1} of $MMSE_{\leq D}$}\label{prf:lowerboundbi1}

 Without loss of generality, we suppose through the proof that $\sigma^2=1$. Let us fix $D\in \N$ and let us suppose $$\zeta:=\lambda^4D^8\max\pa{n,p,\frac{np}{K^2},\frac{np}{L^2}}<1\enspace.$$

As in the proof of Theorems \ref{thm:lowdegreeclusteringsharp} and \ref{thm:lowdegreesparsesharp}, the expression of the $MMSE_{\leq D}$ can be reduced to $$MMSE_{\leq D}=\inf_{f\in \R_{D}[Y]}\E\cro{\pa{f(Y)-x}^2}=\frac{1}{K}-corr_{\leq D}^2\enspace,$$

with $x=\1_{k^*_1=k^*_2}$ and $corr_{\leq D}^2$ being defined in Equation \eqref{def:corr}. We shall, as in the proofs of Theorem \ref{thm:lowdegreeclusteringsharp} and \ref{thm:lowdegreesparsesharp}, upper-bound $corr^2_{\leq D}$ using Proposition \ref{thm:schrammwein}, which states that 

$$corr_{\leq D}^2\leq \underset{|\alpha|\leq D}{\sum_{\alpha\in \N^{n\times p}}}\frac{\kappa_{x,\alpha}^2}{\alpha!}\enspace,$$

with $\kappa_{x,\alpha}=\cumul\pa{x, \pa{X_{ij}}_{ij\in \alpha}}$, where we see $\alpha$ as a multiset of $[n]\times [p]$. The biclustering model is a special case of the latent model (\ref{eq:latent-model}), with

$$Z=k^*, l^*, \eps^r,\eps^c, \quad \delta_{ij}(k^*)=\eps^r_i\eps^c_j,\quad \textrm{and}\quad \theta_{i,j}(k^*)=(k^*_{i},l^*_j)\enspace.$$

Combining Proposition \ref{thm:schrammwein} and Theorem \ref{thm:LTC}, we need to upper-bound, for any multiset $\alpha$ and any decomposition $\beta_1+\ldots+\beta_l=\alpha$, with $|\beta_s|=2$, the cumulant 

$$C_{x,\beta_1,\ldots, \beta_l}=\cumul\pa{x,\prod_{ij\in \beta_1}\eps^r_i\eps^c_j\1_{\Omega_{\beta_1}(k^*,l^*)},\ldots, \prod_{ij\in \beta_l}\eps^r_i\eps^c_j\1_{\Omega_{\beta_l}(k^*,l^*)}}\enspace,$$

with $\Omega_{\beta}(k^*,l^*):= \ac{\big|\ac{(k^*_{i},l^*_j):\enspace (i,j)\in \beta}\big|=1}$. Building on the recursive Bound (\ref{eq:rec-C}), we derive in Section \ref{prf:boundLTCbi} the following upper-bound.

\begin{lem}\label{lem:boundLTCbi}
  We recall that $\#\alpha$ stands for the cardinality of the points $i\in[1,n]$ such that $\alpha_{i:}\neq 0$ and that $|\alpha|:=\sum_{ij}\alpha_{ij}$. We have

$$|C_{x,\beta_1,\ldots,\beta_l}|\leq |\alpha|^{\frac{|\alpha|}{2}}\frac{1}{K}\min\pa{1,\pa{\frac{1}{K\wedge L}}^{\#\alpha+r_\alpha-\frac{|\alpha|}{2}-2}}\enspace.$$
 \end{lem}

Combining this bound with (\ref{eq:LTCgeneral:thm2}) and counting the number of partitions $\pi\in \mathcal{P}_2(\alpha)$ for which $C_{x,\beta_{1}(\pi),\ldots,\beta_{l}(\pi)}\neq 0$ , we prove in Section \ref{prf:controlcumulantsbisharp}  the next  upper-bound on $|\kappa_{x,\alpha}|$.
\medskip

\begin{lem}\label{lem:controlcumulantsbisharp}
Let $\alpha\in\N^{n\times p}$ non-zero. We have the upper bound $$|\kappa_{x,\alpha}|\leq \lambda^{|\alpha|}|\alpha|^{|\alpha|}\frac{1}{K}\min\pa{1,\pa{\frac{1}{K\wedge L}}^{\#\alpha+r_\alpha-\frac{|\alpha|}{2}-2}}\enspace.$$
\end{lem}

The last stage, is to prune the multisets $\alpha$ for which $\kappa_{x,\alpha}=0$.
Next lemma gives necessary conditions for having $\kappa_{x,\alpha}\neq 0$.  We refer to Section \ref{prf:nullcumulantsbisharp} for a proof of this lemma.

\begin{lem}\label{lem:nullcumulantsbisharp}
Let $\alpha\in\N^{n\times p}$ be non-zero. If $\kappa_{x,\alpha}\neq 0$, then
\begin{itemize}
\item $1,2\in supp(\alpha)$;
\item For all $i\in supp(\alpha)$, $|\alpha_{i:}|\geq 2$;
\item For all $j\in col(\alpha)$, $|\alpha_{j:}|\geq 2$.
\end{itemize}
In particular, we have $\#\alpha\geq 2$, $|\alpha|\geq 2r_\alpha$ and $|\alpha|\geq 2\#\alpha$.
\end{lem}
\medskip

For any $r\geq 1$, $m\geq 2$ and $d\geq \max\pa{2r,2m}$, there are at most $p^rn^{m-2}d^{2d}$ matrices $\alpha$ satisfying the conditions of Lemma \ref{lem:nullcumulantsbisharp} with $|\alpha|=d$, $r_\alpha=r$ and $\#\alpha=m$. Using Proposition \ref{thm:schrammwein}, we have 
\begin{align*}
corr^2_{\leq D}\leq&\underset{|\alpha|\leq D}{\sum_{\alpha\in \N^{n\times p}}}\kappa_{x,\alpha}^2\\
\leq& \frac{1}{K^2}+\frac{1}{K^2}\sum_{d=2}^{D}\sum_{m=2}^{d/2}\sum_{r=1}^{d/2} p^r n^{m-2}d^{4d}\lambda^{2d}\min\pa{1,\pa{\frac{1}{K\wedge L}}^{2m+2r-d-4}}\enspace.
\end{align*} 

Let us fix $r\geq 1$, $m\geq 2$ and $d\geq \max\pa{2r, 2m}$ and let us upper-bound $p^r n^{m-2}d^{4d}\lambda^{2d}\min\pa{1,\pa{\frac{1}{K\wedge L}}^{2m+2r-d-4}}$. First, we suppose that $d\geq 2\pa{m-2+r}$ and we get 
\begin{align*}
p^r n^{m-2}d^{4d}\lambda^{2d}\min\pa{1,\pa{\frac{1}{K\wedge L}}^{2m+2r-d-4}}\leq &\pa{D^4\lambda^2}^{d}p^r n^{m-2}\\
\leq& \pa{D^4\lambda^2}^{d-2r-2(m-2)}\pa{D^8\lambda^4n}^{m-2}\pa{D^8\lambda^4p}^{r}\\
\leq& \sqrt{\zeta}^{d-2r-2(m-2)}\zeta^{m-2}\zeta^{r}= \zeta^{d/2}\enspace .
\end{align*}

Then, let us suppose that $d\leq 2\pa{m-2+r}$. By symmetry of the roles of $K$ and $L$, we suppose that $K\wedge L=K$. We get 
\begin{align*}
p^r n^{m-2}d^{4d}\lambda^{2d}\min\pa{1,\pa{\frac{1}{K\wedge L}}^{2m+2r-d-4}}= &p^r n^{m-2}\pa{2d^4}^d\lambda^{2d}\min\pa{1,\pa{\frac{1}{K}}^{2m+2r-d-4}}\\
\leq& p^r n^{m-2}\pa{D^4\lambda^2}^d\pa{\frac{1}{K}}^{2m+2r-d-4}\\
\leq& n^{\frac{d}{2}-r}\pa{\frac{n}{K^2}}^{m-2-\pa{\frac{d}{2}-r}}p^r\pa{D^4\lambda^2}^{2r+2(d/2-r)}\\
\leq &\pa{D^8\lambda^4p}^r\pa{D^8\lambda^4n}^{\frac{d}{2}-r}\pa{\frac{n}{K^2}}^{m-2-\pa{\frac{d}{2}-r}}\\
\leq& \pa{D^8\lambda^4n}^{\frac{d}{2}-r}\pa{D^8\lambda^4p}^{\frac{d}{2}-(m-2)}\pa{D^8\lambda^4p\frac{n}{K^2}}^{m-2-\pa{\frac{d}{2}-r}}\\
\leq & \sqrt{\zeta}^d\enspace.
\end{align*}

In the end, we get 
\begin{align*}
corr^2_{\leq D}-\frac{1}{K^2}\leq&\frac{1}{K^2}\sum_{d=2}^{D}\sum_{m=2}^{d/2}\sum_{r=1}^{d/2}\sqrt{\zeta}^d\\
\leq& \frac{1}{K^2}\sum_{d=2}^{D}\frac{d(d-1)}{2}\sqrt{\zeta}^d\\
\leq& \frac{1}{K^2}\frac{\zeta}{\pa{1-\sqrt{\zeta}}^3}\enspace,
\end{align*}

which concludes the proof of the theorem.

\subsection{Proof of Lemma \ref{lem:boundLTCbi}}\label{prf:boundLTCbi}

Let $\alpha=\beta_1+\ldots+\beta_l$ with, for $s\in [l]$, $\beta_s=\ac{(i_s,j_s);(i'_s,j'_s)}$. Let us upper-bound the absolute value of the cumulant $$C_{x,\beta_{1},\ldots,\beta_{l}} =  \cumul\pa{x,\pa{\eps^r_{i_s}\eps^r_{i'_s}\eps^c_{j_s}\eps^c_{j'_s}\1_{\Omega_{\beta_{s}}(k^*,l^*)}}_{s\in [l]}}.$$

For $S\subseteq [l]$, we write $\beta[S]:=\ac{\beta_s, s\in S}$ and we write $\alpha_S=\sum_{s\in S}\beta_s$. In the following, we take the convention $i_0=1$, $i'_0=2$ and $j_0=0$. Applying the recursion formula \ref{eq:rec-C}, we have, for all $S\subseteq [l]$,

\begin{align}\nonumber
|C_{x,\beta[S]} |\leq&   \P\cro{\forall s\in S\cup \ac{0}, k^*_{i_s}=k^*_{i'_s}}\P\cro{\forall s\in S, l^*_{j_s}=l^*_{j'_s}}+\\\nonumber
&+\sum_{S'\subsetneq S}\left|C_{x,\beta[S']}\right|\P\cro{\forall s\in S\setminus S', k^*_{i_s}=k^*_{i'_s}}\P\cro{\forall s\in S\setminus S', l^*_{j_s}=l^*_{j'_s}}\\\nonumber
\leq&\pa{\frac{1}{K}}^{|supp(\alpha_S)\cup \ac{1,2}|-cc(\mathcal{V}[S\cup\ac{0}])}\pa{\frac{1}{L}}^{r_{\alpha_S}-cc\pa{\mathcal{N}[S]}}+\\
&+ \sum_{S'\subsetneq S}\left|C_{x,\beta[S']}\right| \pa{\frac{1}{K}}^{\#\alpha_{S\setminus S'}-cc(\mathcal{V}[S\setminus S'])}\pa{\frac{1}{L}}^{r_{\alpha_{S\setminus S'}}-cc\pa{\mathcal{N}[S\setminus S']}}\enspace,\label{eq:rec-C-bi}
\end{align}

where $\mathcal{V}$ and $\mathcal{N}$ are two graphs on $[0,l]$ defined as follows. For $s,s'\in [0,l]$, $\mathcal{V}$ has an edge between $s$ and $s'$ if and only if $\ac{i_s,i'_s}\cap \ac{i_{s'},i'_{s'}}\neq \emptyset$. For $s,s'\in [0,l]$, $\mathcal{N}$ has an edge between $s$ and $s'$ if and only if $\ac{j_s,j'_s}\cap \ac{j_{s'},j'_{s'}}\neq \emptyset$ (we consider $j_0=j'_0=0$ which implies that $0$ is an isolated point of $\mathcal{N}$). We also define the graph $\mathcal{W}$ on $[0,l]$ with an edge between $s$ and $s'$ if and only if either $\mathcal{V}$ or $\mathcal{N}$ has an edge between $s$ and $s'$. Finally, given a subset $S\subset [0,l]$, we define $\mathcal{V}(S)$, $\mathcal{N}(S)$, and $\mathcal{W}(S)$ as the subgraphs of $\mathcal{V}$, $\mathcal{N}$, and $\mathcal{W}$ induced by $S$.

The next lemma prunes the subsets $S\subseteq [l]$ such that $C_{x,\beta[S]}\neq 0$. In the following, we denote $\mathcal{S}([l])$ the collection of all subsets $S\subseteq [l]$ such that either $S=\emptyset$ or $\mathcal{W}[\ac{0}\cup S]$ is connected and both $1,2\in supp(\alpha_S)$. We postpone to Section \ref{prf:nullityLTCbi} the proof of the next lemma.
\begin{lem}\label{lem:nullityLTCbi}
Let $S\subseteq [l]$ such that $C_{x,\beta[S]}\neq 0$. Then $S\in \mathcal{S}([l])$.
\end{lem}
In particular, we can henceforth restrict our attention to subsets $S\in \mathcal{S}([l])$, so that 
\begin{align}\nonumber
|C_{x,\beta[S]} |\leq&\pa{\frac{1}{K}}^{\#\alpha_S-cc(\mathcal{V}[S\cup\ac{0}])}\pa{\frac{1}{L}}^{r_{\alpha_S}-cc\pa{\mathcal{N}[S]}}+\\
&+ \underset{S'\in \mathcal{S}([l])}{\sum_{S'\subsetneq S}}\left|C_{x,\beta[S']}\right| \pa{\frac{1}{K}}^{\#\alpha_{S\setminus S'}-cc(\mathcal{V}[S\setminus S'])}\pa{\frac{1}{L}}^{r_{\alpha_{S\setminus S'}}-cc\pa{\mathcal{N}[S\setminus S']}}\enspace.\label{eq:rec-C-bi2}
\end{align}

In the following, let us define recursively a function $f$ on $\mathcal{S}([l])$ satisfying, for all $S\in \mathcal{S}([l])$,
\begin{equation}\label{eq:deffunctionfbi}
f(S)=1+\underset{S'\in \mathcal{S}([l])}{\sum_{S'\subsetneq S}}f(S')\enspace.
\end{equation}
In particular, $f(\emptyset)=1$. Using the connectivity of $\mathcal{W}[S\cup \ac{0}]$ whenever $S\in \mathcal{S}([l])$ is non-empty, we prove in Section \ref{prf:upperboundkappaSsparse} the following lemma.
\begin{lem}\label{lem:upperboundkappaSbi}
For all $S\in \mathcal{S}([l])$, we have $|C_{x,\beta[S]}|\leq f(S)\pa{\frac{1}{K}}^{\#\alpha_S-cc(\mathcal{V}[S\cup \ac{0}])}\pa{\frac{1}{L}}^{r_{\alpha_S}-cc(\mathcal{V}[S])}.$
\end{lem}
The next lemma, proved in Section \ref{prf:upperboundfbi}, provides an upper-bound of $f([l])$
\begin{lem}\label{lem:upperboundfbi}
For all $S\in \mathcal{S}([l])$, we have $f([l])\leq |\alpha|^{\frac{|\alpha_S|}{2}}$
\end{lem}

Combining Lemma \ref{lem:upperboundfbi} and Lemma \ref{lem:upperboundkappaSbi}  implies that 
\begin{equation}\label{eq:rec-C-bi3}
|C_{x,\beta_1,\ldots,\beta_l}|\leq \pa{|\alpha|}^{\frac{|\alpha|}{2}}\frac{1}{K}\pa{\frac{1}{K\wedge L}}^{\#\alpha+r_{\alpha}-cc(\mathcal{V}[[0,l]])-cc(\mathcal{V}[[l]])-1}\enspace.
\end{equation}

It remains to lower-bound the quantity $\#\alpha+r_{\alpha}-cc(\mathcal{V}[[0,l]])-cc(\mathcal{V}[[l]])$. To do so, we shall use the connectivity of the graph $\mathcal{W}[0,l]$. We postpone to Section \ref{prf:connectivitybi} the proof of the next lemma.

\begin{lem}\label{lem:connectivitybi}
We have 
$$\#\alpha+r_{\alpha}-cc(\mathcal{V}[[0,l]])-cc(\mathcal{V}[[l]])-1\geq \max\pa{0, \#\alpha+r_\alpha-\frac{|\alpha|}{2}-2}\enspace.$$

\end{lem}

Using Lemma \ref{lem:connectivitybi}, we are able to conclude the proof of the lemma with

$$|C_{x,\beta_1,\ldots,\beta_l}|\leq |\alpha|^{\frac{|\alpha|}{2}}\frac{1}{K}\min\pa{1,\pa{\frac{1}{K\wedge L}}^{\#\alpha+r_\alpha-\frac{|\alpha|}{2}-2}}\enspace.$$

\subsection{Proof of Lemma \ref{lem:nullityLTCbi}}\label{prf:nullityLTCbi}

Let $S\notin \mathcal{S}([l])$ and let us prove that $C_{x,\beta[S]}=0$.

Let us first suppose that $\mathcal{W}[S\cup \ac{0}]$ is not connected. Let $C_1$, $C_2$ be a partition of $S\cup \ac{0}$ with no edges of $\mathcal{W}$ connecting them. Hence, the family of random variables $((\eps^r_i,k^*_i)_{i\in \cup_{s\in C_1}\ac{i_s, i'_s}}, \pa{\eps^c_j,l^*_j}_{j\in \cup_{s\in C_1}\ac{j_s, j'_s}})$ is independent of the family~\\
$((\eps^r_i,k^*_i)_{i\in \cup_{s\in C_2}\ac{i_s, i'_s}}, \pa{\eps^c_j,l^*_j}_{j\in \cup_{s\in C_2}\ac{j_s, j'_s}})$. Then,  Lemma \ref{lem:independentcumulant2} implies that $C_{x,\beta_{[S]}}=0$.

Let us now suppose that either $1\notin supp(\alpha_S)$ either $2\notin supp(\alpha_S)$. Then, $\1_{k^*_{1}=k^*_{2}}$ is independent from $((\eps^r_i,k^*_i)_{i\in \cup_{s\in S}\ac{i_s, i'_s}}, \pa{\eps^c_j,l^*_j}_{j\in \cup_{s\in S}\ac{j_s, j'_s}})$. Thus, from Lemma \ref{lem:independentcumulant2}, we deduce that  $C_{x,\beta_{[S]}}=0$.

\subsection{Proof of Lemma \ref{lem:upperboundkappaSbi}}

Let us prove by induction on $S\in \mathcal{S}([l])$ that 
\[
|C_{x,\beta[S]}|\leq f(S)\pa{\frac{1}{K}}^{\#\alpha_S-cc(\mathcal{V}[S\cup \ac{0}])}\pa{\frac{1}{L}}^{r_{\alpha_S}-cc(\mathcal{V}[S])}\enspace . 
\]

The initialization is straightforward since $f(\emptyset)=1$ and $C_{x,\beta[\emptyset]}=\frac{1}{K}$. Consider a set $S\in \mathcal{S}([l])$ non-empty and let us suppose that the result holds for all $S'\in \mathcal{S}([l])$ with $S'\subsetneq S$. Combining \eqref{eq:rec-C-bi2} and the induction hypothesis leads to

\begin{align*}
|C_{x,\beta[S]} |\leq&\pa{\frac{1}{K}}^{\#\alpha_S-cc(\mathcal{V}[S\cup\ac{0}])}\pa{\frac{1}{L}}^{r_{\alpha_S}-cc\pa{\mathcal{N}[S]}}+\\
&+ \underset{S'\in \mathcal{S}([l])}{\sum_{S'\subsetneq S}}\left|C_{x,\beta[S']}\right| \pa{\frac{1}{K}}^{\#\alpha_{S\setminus S'}-cc(\mathcal{V}[S\setminus S'])}\pa{\frac{1}{L}}^{r_{\alpha_{S\setminus S'}}-cc\pa{\mathcal{N}[S\setminus S']}}\\
\leq& \pa{\frac{1}{K}}^{\#\alpha_S-cc(\mathcal{V}[S\cup\ac{0}])}\pa{\frac{1}{L}}^{r_{\alpha_S}-cc\pa{\mathcal{N}[S]}}+\pa{\frac{1}{K}}\pa{\frac{1}{K}}^{\#\alpha_S-cc\pa{\mathcal{V}[S]}}\pa{\frac{1}{L}}^{r_{\alpha_s}-cc\pa{\mathcal{N}[S]}}+\\
&+ \underset{S'\in \mathcal{S}([l])}{\sum_{\emptyset\neq S'\subsetneq S}}f(S')\pa{\frac{1}{K}}^{\#\alpha_{S'}+\#\alpha_{S\setminus S'}-cc(\mathcal{V}[S'\cup\ac{0}])-cc(\mathcal{V}[S\setminus S'])}\pa{\frac{1}{L}}^{r_{\alpha_{S'}}+r_{\alpha_{S\setminus S'}}-cc\pa{\mathcal{N}[S']}-cc\pa{\mathcal{N}[S\setminus S']}}\enspace.\end{align*}

Let us deal with the term corresponding to $S'=\emptyset$. It is clear that $cc\pa{\mathcal{V}[S]}\leq cc\pa{\mathcal{V}[S\cup \ac{0}]}+1$. Thus $\pa{\frac{1}{K}}\pa{\frac{1}{K}}^{\#\alpha_S-cc\pa{\mathcal{V}[S]}}\pa{\frac{1}{L}}^{r_{\alpha_s}-cc\pa{\mathcal{N}[S]}}\leq  \pa{\frac{1}{K}}^{\#\alpha_S-cc(\mathcal{V}[S\cup\ac{0}])}\pa{\frac{1}{L}}^{r_{\alpha_S}-cc\pa{\mathcal{N}[S]}}$.

By the recursive definition of $f(S)$, it is sufficient to prove that, for all non empty $S'\subsetneq S$, we both have $$\#\alpha_{S'}+\#\alpha_{S\setminus S'}-cc(\mathcal{V}[S'\cup\ac{0}])-cc(\mathcal{V}[S\setminus S'])\geq \#\alpha_{S}-cc(\mathcal{V}[S\cup\ac{0}])$$ and 
$$r_{\alpha_{S'}}+r_{\alpha_{S\setminus S'}}-cc\pa{\mathcal{N}[S']}-cc\pa{\mathcal{N}[S\setminus S']}\geq r_{\alpha_{S}}-cc\pa{\mathcal{N}[S]}\enspace.$$

Let us prove only that $\#\alpha_{S'}+\#\alpha_{S\setminus S'}-cc(\mathcal{V}[S'\cup\ac{0}])-cc(\mathcal{V}[S\setminus S'])\geq \#\alpha_{S}-cc(\mathcal{V}[S\cup\ac{0}])$, the proof being similar for the other term. Let $q=cc(\mathcal{V}[S\cup\ac{0}])$ and let $cc_1,\ldots, cc_q$ be the connected components of $\mathcal{V}[S\cup\ac{0}]$. Let $q'\in [q]$ and let $h=cc(\mathcal{V}[\pa{S'\cup\ac{0}}\cap cc_{q'}])+cc(\mathcal{V}[(S\setminus S')\cap cc_{q'}])$. Let $a_1,\ldots, a_h$ denote the collection of those connected components. Since the graph $\mathcal{V}[cc_{q'}]$ is connected, we can, up to a possible reordering of these connected components,  suppose that, for all $h'\in [2,h]$, $\cup_{s\in a_{h'}}\ac{i_s, i'_s}$ intersects $\cup_{h''<h'}\cup_{s\in a_{h''}}\ac{i_s, i'_s}$. We deduce that 
\begin{align*}
\sum_{h'\in [h]}|\cup_{s\in a_{h'}}\ac{i_s,i'_s}|= & |\cup_{s\in a_{1}}\ac{i_s,i'_s}|+\sum_{h'\geq 2}|\cup_{s\in a_{h'}}\ac{i_s,i'_s}|\\
\geq& |\cup_{s\in a_{1}}\ac{i_s,i'_s}|+\sum_{h'\geq 2}\pa{1+\left|\cup_{s\in a_{h'}}\ac{i_s,i'_s}\setminus \pa{\cup_{h''<h'}\cup_{s\in a_{h''}}\ac{i_s, i'_s}}\right|}\\
\geq& |\cup_{s\in cc_{q'}}\ac{i_s, i'_s}|+cc(\mathcal{V}[S'\cap cc_{q'}])+cc(\mathcal{V}[(S\setminus S')\cap cc_{q'}]) -1\enspace. 
\end{align*}

Together with the fact that $\sum_{h'\in [h]}|\cup_{s\in a_{h'}}\ac{i_s,i'_s}|=|\cup_{s\in \pa{S'\cup\ac{0}}\cap cc_{q'}}\ac{i_s,i'_s}|+|\cup_{s\in (S'\setminus S)\cap cc_{q'}}\ac{i_s,i'_s}|$, this leads us  to 
\begin{flalign*}
|\cup_{s\in \pa{S'\cup \ac{0}}\cap cc_{q'}}\ac{i_s,i'_s}|+|\cup_{s\in (S\setminus S')\cap cc_{q'}}\ac{i_s,i'_s}|-cc(\mathcal{V}[S'\cap cc_{q'}])-cc(\mathcal{V}[(S\setminus S')\cap cc_{q'}])
\\  \geq |\cup_{s\in cc_{q'}}\ac{i_s, i'_s}|-1\enspace.
\end{flalign*}
Summing other all $q'\in [q]$ leads us to
$$\#\alpha_{S'}+\#\alpha_{S\setminus S'}-cc(\mathcal{V}[S'\cup\ac{0}])-cc(\mathcal{V}[S\setminus S'])\geq \#\alpha_S-cc\pa{\mathcal{V}[S\cup \ac{0}]}\enspace.$$

This concludes the proof of the lemma.

\subsection{Proof of Lemma \ref{lem:upperboundfbi}}\label{prf:upperboundfbi}

We proceed by induction to prove that, for all $S\in \mathcal{S}([l])$, we have  
$f(S)\leq |\alpha|^{\frac{|\alpha_S|}{2}}$. 
The initialization is trivial since $f(\emptyset)=1$ and $\alpha_{\emptyset}=0$. Let us take $S\in  \mathcal{S}([l])$ non empty and let us suppose that the result holds for all $S'\subsetneq S$. For all $s\in S$, let $S^*(s)$ the maximal element of $\mathcal{S}([l])$ which is included in $S\setminus \ac{s}$. The existence of such an element in provided from the fact that the set of elements $S'\in \mathcal{S}([l])$ with $S'\subseteq S\setminus \ac{s}$ is not empty (it contains $\emptyset$) and is stable by union. We have 
\begin{align*}
f(S)=1+\underset{S'\in \mathcal{S}([l])}{\sum_{S'\subsetneq S}}f(S')
& \leq  1+\sum_{s\in S}\underset{S'\in \mathcal{S}([l])}{\sum_{S'\subsetneq S^*(s)}}f(S')\\
& \leq  1+\sum_{s\in S} \cro{2f(S^*(s))-1}\\
& \leq  2\sum_{s\in S}f(S^*(s)).
\end{align*}
Applying the induction hypothesis leads to 
\begin{equation*}
f(S)\leq 2\sum_{s\in S}|\alpha|^{\frac{|\alpha_{S^*(s)}|}{2}}\enspace.
\end{equation*}
Since $S^*(s)$ does not contain $s$, it follows that $|\alpha_{S^*_s}|\leq |\alpha_{S}|-2$. We deduce that 
$$f(S)\leq 2\sum_{s\in S}|\alpha|^{\frac{|\alpha_S|}{2}-1}=  \frac{2|S|}{|\alpha|} |\alpha|^{\frac{|\alpha_S|}{2}}\leq \pa{|\alpha|}^{\frac{|\alpha_S|}{2}}\enspace,$$
where the last inequality comes from the fact that $|S|\leq \frac{|\alpha|}{2}$. This concludes the induction and the proof of the lemma.

\subsection{Proof of Lemma \ref{lem:connectivitybi}}\label{prf:connectivitybi}

Recall that $\mathcal{V}=\mathcal{V}[[0,l]]$ and that $\mathcal{W}=\mathcal{W}[[0,l]]$. For short, we write $\mathcal{N}'= \mathcal{N}[[l]]$.

We know that $\ac{1,2}\subset supp(\alpha)$ and that they are in the same connected component of $\mathcal{V}$. Thus, it is clear that $\#\alpha+r_\alpha-cc(\mathcal{N}')-cc(\mathcal{V})\geq 1$. It remains to prove that $$\#\alpha+r_\alpha-cc(\mathcal{N}')-cc(\mathcal{V})\geq \#\alpha+r_\alpha-\frac{|\alpha|}{2}-1\enspace,$$
which is equivalent to
$$|\alpha|\geq 2\pa{cc(\mathcal{N}')+cc(\mathcal{V})-2}\enspace.$$
We shall use the fact that $\mathcal{W}$ is connected. We write $q=cc(\mathcal{N}')$. We write $J_1,\ldots, J_q$ the partition of $col(\alpha)$ induced by the equivalence relation; $j$ and $j'$ are equivalent if and only if there exists $s,s'$ in the same connected component of $\mathcal{N}'$ such that $j\in col(\beta_s)$ and $j'\in col(\beta_{s'})$. For $R\subseteq [l]$, we write $q(R)\subseteq [q]$ the collection of $q'$ such that $\sum_{s\in R}\sum_{j\in J_{q'}}|(\beta_s)_{:j}|\neq 0$. In other words, $q(R)$ also corresponds to the collection of connected components of $\mathcal{N}'$ that intersect with $R$. 
Let us finally write $t=cc(\mathcal{V})$ and $cc_1,\ldots, cc_t$ the connected components of $\mathcal{V}$.

The graph $\mathcal{W}$ corresponds to the superposition of $\mathcal{V}$ and of $\mathcal{N}'$. Connected components in $\mathcal{V}$ are connected by edges in $\mathcal{N}$.

Besides, recall that $\mathcal{W}$ is connected. Hence, we can assume, without loss of generality, that for all $t'\in [2,t]$, $q(cc_{t'}\setminus \ac{0})$ intersects $\cup_{t''\leq t'-1}q(cc_{t''}\setminus \ac{0})$. For all $t'\in [t]$, it is clear that $|\alpha_{cc_{t'}\setminus \ac{0}}|\geq 2|q(cc_{t'}\setminus \ac{0})|$. Hence, we conclude that 
\begin{align*}
|\alpha|\geq& \sum_{t'\leq t}|\alpha_{cc_{t'}\setminus \ac{0}}|
\geq \sum_{t'\leq t}2q(cc_{t'}\setminus \ac{0})\\
\geq& 2\sum_{t'\leq t}\pa{\1\ac{t'\neq 1}+\left|q(cc_{t'}\setminus \ac{0})\setminus \cup_{t''\leq t'-1}q(cc_{t''}\setminus \ac{0})\right|}\\
\geq& 2(t-1)+2\left|q([l])\right|= 2t+2q-1\enspace.
\end{align*}
This concludes the proof of the lemma.

\subsection{Proof of Lemma \ref{lem:controlcumulantsbisharp}}\label{prf:controlcumulantsbisharp}

In order to get a satisfying upper-bound of $|\kappa_{x,\alpha}|$ from Inequality Theorem \ref{thm:LTC} and Lemma \ref{lem:boundLTCbi}, it is sufficient to remark that $\mathcal{P}_2(\alpha)$ contains at most $|\alpha|^{\frac{|\alpha|}{2}-1}$ elements.

Plugging this with Lemma \ref{lem:boundLTCbi} in Theorem \ref{thm:LTC} leads to

$$|\kappa_{x,\alpha}|\leq |\alpha|\frac{1}{K}\min\pa{1,\pa{\frac{1}{K\wedge L}}^{\#\alpha+r_\alpha-\frac{|\alpha|}{2}-2}}\enspace.$$

which concludes the proof of the lemma.

\subsection{Proof of Lemma \ref{lem:nullcumulantsbisharp}}\label{prf:nullcumulantsbisharp}

Let $\alpha\in \N^{n\times p}$ non-zero. Let us prove that if $\alpha$ does not satisfy the three conditions of Lemma \ref{lem:nullcumulantsbisharp}, then $\kappa_{x,\alpha}=0$.

First, we suppose that either $1$ or $2$ is not in $supp(\alpha)$. This implies that $x$ is independent from $\pa{X_{ij}}_{ij\in \alpha}$ and we deduce from Lemma \ref{lem:independentcumulant2} that $\kappa_{x,\alpha}=0$. 

Let us suppose that there exists $i_0\in supp(\alpha)$ with $|\alpha_{i_0:}|=1$. Since $\eps^{r}_{i_0}$ has the same law as $-\eps^{r}_{i_0}$ and is independent from all the other variables (in particular, changing $\eps^{r}_{i_0}$ by $-\eps^{r}_{i_0}$ does not change the law of $(x,X)$). So, by multilinearity of cumulants, we have $\kappa_{x,\alpha}=\pa{-1}^{|\alpha_{i_0:}|}\kappa_{x,\alpha}=-\kappa_{x,\alpha}$. We deduce directly that $\kappa_{x,\alpha}=0$.

Similarly, if there exists $j_0$ such that $|\alpha|_{:j_0}= 1$, we use the fact that changing $\eps^c_{j_0}$ to $-\eps^c_{j_0}$ does not change the law of $(x,X)$ and we deduce $\kappa_{x,\alpha}=\pa{-1}^{|\alpha_{:j_0}|}\kappa_{x,\alpha}=-\kappa_{x,\alpha}$ so that $\kappa_{x,\alpha}=0$.

\section{Proof of Corollary \ref{cor:LTC}}\label{sec:cor:LTC} 
To get the Bound (\ref{eq:bound-C}), we start from M\"obius formula -- see   Lemma \ref{lem:mobiusformula} in Appendix~\ref{sec:cumulants} --
\begin{align*}
|C_{x,\beta_{1},\ldots,\beta_{l}}|&\leq \sum_{\pi\in \mathcal{P}\pa{[l]\cup \ac{x}}} (|\pi|-1)! \E\cro{|x|; \underset{s\in \pi_{1}\setminus \ac{x}}{\cap}\Omega _{\beta_{s}}} \prod_{k=2}^{|\pi|} \P\cro{\underset{s\in \pi_{k}}{\cap}\Omega _{\beta_{s}}}\\
&\leq   \max_{\pi\in \mathcal{P}\pa{[l]\cup \ac{x}}}\left\{\E\cro{|x|; \underset{s\in \pi_{1}\setminus\ac{x}}{\cap}\Omega _{\beta_{s}}} \prod_{k=2}^{|\pi|} \P\cro{\underset{s\in \pi_{k}}{\cap}\Omega _{\beta_{s}}}\right\} \sum_{\pi\in \mathcal{P}\pa{[l]\cup \ac{x}}} (|\pi|-1)!\enspace .
\end{align*}
Denoting by ${l+1 \brace k}$ the Stirling number of the second kind, which counts the number of partitions $\pi\in \mathcal{P}\pa{[l]\cup \ac{x}}$ with $k$ non-empty sets, we get
\begin{align*}
\sum_{\pi\in \mathcal{P}\pa{[l]\cup \ac{x}}} (|\pi|-1)! &= \sum_{k=1}^{l+1} {l+1 \brace k} (k-1)! \\
&=  l!+{l+1 \brace l} (l-1)!+ \sum_{k=1}^{l-1} {l \brace k} k!+\sum_{k=2}^{l-1} {l \brace k-1} (k-1)! \\
&= \sum_{k=1}^{l} {l \brace k} k! + \sum_{k=1}^{l-2} {l \brace k} k!+{l+1 \brace l} (l-1)! \\
&= 2 \sum_{k=1}^{l} {l \brace k} k! =2f_{l},
\end{align*}
where we used for the penultimate equality that  
$${l+1 \brace l} (l-1)!= {(l+1)\over 2} l! =  {(l-1)\over 2} l!+l!= {l \brace l-1} (l-1)!+l!\enspace. $$
The bound on $f_{l}$ readily follows from the exponential generating function evaluated at $x=1/2$
$$\sum_{l\geq 1} {f_{l}\over l!} \pa{1\over 2}^l ={1\over 2-\exp(1/2)}\leq 3.$$ 
The proof of Corollary \ref{cor:LTC} is complete.

\section{Proof of the upper bounds}\label{sec:proof:ub}

\subsection{Proof of Proposition \ref{prop:upper-boundclustering}}\label{prf:upper-boundclustering}

Let us suppose without loss of generality that $\sigma^2=1$. We shall introduce two different procedures corresponding to the different regimes. 

Both procedures proceed from the same general scheme:
\begin{enumerate}
\item We split the dataset randomly into two datasets $Y^{(1)}$ and $Y^{(2)}$;
\item We compute $\hat{v}_1,\ldots, \hat{v}_K$ the leading eigenvectors of $\pa{Y^{(1)}}^TY^{(1)}$ and we project orthogonally $Y^{(2)}$ onto $\hat{v}_1,\ldots, \hat{v}_K$;
\item We apply a low-dimensional clustering procedure on the projected dataset $\hat{p}\pa{Y^{(2)}}$;
\item We perform Linear Discriminant Analysis in order to assign each point of $Y^{(1)}$ to one of the clusters of $\hat{p}\pa{Y^{(2)}}$.
\end{enumerate}

Let $\delta_1,\ldots, \delta_n$ i.i.d uniformly taken on $\ac{1,2}$. Let $I_1=\ac{i\in [n], \delta_i=1}$ and $I_2=\ac{i\in [n], \delta_i=2}$. Let $Y^{(1)}\in \R^{|I_1|\times p}$ be the data matrix  restricted to $I_1$ and $Y^{(2)}\in \R^{|I_2|\times p}$ be the date matrix  restricted to $I_2$. Let $\hat{v}_1,\ldots, \hat{v}_K$ be the $K$ leading eigenvectors of $\pa{Y^{(1)}}^TY^{(1)}$ and let $\hat{p}$ be the orthogonal projection on $\hat{v}_1,\ldots, \hat{v}_K$. The following key lemma ensures that the projected centers are still well-separated.

 We recall that, in this section, we assume that the partition $G^*$ is balanced in the following sense.
  For some constant $\gamma>0$,we have 
 $$\frac{\max_{k}|G_k^*|}{\min_k|G_k^*|}\leq \gamma\enspace.$$

\begin{lem}\label{lem:projmixture}
We suppose $\max(K,\log(n))\leq p\leq n$. There exists a constants $c_{\gamma}$ that only depends on $\gamma$ such that the following holds provided that $\Delta^4\geq c_{\gamma}\frac{pK^2}{n}$.  With probability at least $1-\frac{4}{n^2}$, we have  $\|\hat{p}(\mu_k)-\hat{p}(\mu_l)\|^2\geq \frac{1}{4}\|\mu_k-\mu_l\|^2$ for all  $k, l\in[1,K]$.
\end{lem}

We organize the proof in the following way. In Section \ref{sec:tructenseur}, we apply \cite{LiuLi2022} to the projected dataset $Y^{(2)}$ this allows to prove the second part Proposition~\ref{prop:upper-boundclustering} --see Proposition~\ref{cor:projection+tenseur} below. In Section \ref{sec:hierarchical}, we apply some hierarchical clustering procedure which will lead to the first part of Proposition~\ref{prop:upper-boundclustering} --see Proposition~\ref{cor:hierarchical+projection}. Finally, in Section \ref{prf:projmixture}, we provide a proof of Lemma \ref{lem:projmixture}.

Throughout the proof, we shall multiple times rely on the following lemma that ensures that the restrictions of $G^*$ to $I_1$ and $I_2$ are balanced.
\begin{lem}\label{lem:balancedclusters}
Suppose that $n\geq c\gamma^2 K^2$ with $c$ a numerical constant. Then, with probability higher than $1-1/n^2$, for all $k\in [K]$, we have $|G_k^*\cap I_1|\geq |G_k^*|/4$ and  $|G_k^*\cap I_2|\geq |G_k^*|/4$.
\end{lem}

In the following, we work conditionally on $I_1$ and $I_2$ and we assume, without loss of generality, that the event of Lemma \ref{lem:balancedclusters} holds.

\begin{proof}[Proof of Lemma \ref{lem:balancedclusters}]
Let us fix $k\in [K]$ and let us consider $|G_k^*\cap I_1|$ which is a binomial of parameters $|G_k^*|\geq \frac{n}{K\gamma}$ and $\frac{1}{2}$. Using Hoeffding Inequality, we deduce that, for $t>0$, 
$$\P\cro{||G_k^*\cap I_1|-\frac{|G_k^*|}{2}|\geq t}\leq 2\exp\pa{\frac{-2t^2}{|G_k^*|}}\enspace.$$
Taking $t=\frac{|G_k^*|}{4}$, applying an union bound on all $k\in [K]$, yields that the desired result holds with probabilily higher than $1- 2K\exp[-(\min_{k=1,\ldots,K}|G_k^*|)/4]$ which is larger than $1-1/n$ as soon since $|G_k^*|\geq n/[K\gamma]\geq c^{1/2}\sqrt{n}$. 
\end{proof}

\subsubsection{Tensor method of  Li and Liu~\cite{LiuLi2022}}\label{sec:tructenseur}

In this section, we apply as a black-box the iterative tensor projection procedure of\cite{LiuLi2022} to the projected dataset $\hat{p}\pa{Y^{(2)}}$. This polynomial-time method is described in  Algorithm \ref{alg:projection+tenseur}. In this subsection, we denote $\hat{G}$ the resulting estimator of the partition. 

\begin{algorithm}
\caption{Projection and iterative tensor projection}\label{alg:projection+tenseur}
\KwData{$Y_{1},\ldots, Y_{n}$}
Draw $(\delta_i)_{i\in [1,n]}$ independently and uniformly on $\{1,2\}$;\\
Compute $\hat{v_1},\ldots,\hat{v}_K$ the leading eigenvectors of $\pa{Y^{(1)}}^T Y^{(1)}$, with $Y^{(1)}$ the restriction of $Y$ to $I_1$;\\
For $i\in I_2$, compute $\hat{p}(Y_i)$ the orthogonal projection of $Y_i$ on the space $Vect\pa{\hat{v_1},\ldots,\hat{v}_K}$ (if $p\leq \max\pa{K,\log(n)}$ keep $Y_i$);\\
Compute $\hat{G}$ the clustering of the projected dataset $\hat{p}(Y^{(2)})$ using the method from \cite{LiuLi2022};\\
\For{$k\in[1,K]$}{
Compute $\hat{\mu}_k:=\frac{1}{|\hat{G}_k|}\sum_{i\in \hat{G}_k}Y_i$}
\For{$i\in I_1$}{
Assign $i$ to the group $\hat{G}_k$ minimizing $\|Y_i-\hat{\mu}_k\|$.}
\KwResult{The partition $\hat{G}$.}
\end{algorithm}

\medskip
The following proposition states that $\hat{G}$ perfectly recover the unknown partition provided the separation $\Delta^2$ is large compared to $\log(n)+\sqrt{\frac{pK^2}{n}}$. This comes to the price of the condition that $n$ is at least polynomial in $K$.

\begin{prop}\label{cor:projection+tenseur}
For any $\eps>0$, there exist constants $c_{\gamma}$, $c'_{\gamma,\epsilon}>0$, and $c_2>0$ such that the following holds. If $\Delta^2\geq c_{\gamma}\pa{\log(n)^{1+\eps}+\sqrt{\frac{pK^2}{n}}}$ and $n\geq K^{c'_{\gamma,\epsilon}}$, the output $\hat{G}$ of Algorithm \ref{alg:projection+tenseur} satisfies $$\P\cro{\hat{G}=G^*}\geq 1-c_2/n^2\enspace.$$
\end{prop}

\medskip

\begin{proof}[Proof of Proposition \ref{cor:projection+tenseur}]\label{prf:projection+tenseur}
This proof mostly builds upon Lemma~\ref{lem:projmixture} and the work of~\cite{LiuLi2022}.
Note that if $p\leq \max\pa{\log(n), K}$, one does not need to use the projection $\hat{p}$ and we can consider $\hat{p}=I_d$. In all cases, the dimension of the projected dataset is at most $\max\pa{\log(n), K}$.

We work conditionally on $I_1$ and $I_2$. Without loss of generality, this event being of high probability when $n\geq c_{\gamma}K^2$, we suppose that $|I_1|,|I_2|\geq n/4$ and that for all $k\in [K]$, $|G_k^*\cap I_1|, |G_k^*\cap I_2|\geq \frac{|G_k^*|}{4}$ (Lemma \ref{lem:balancedclusters}). Conditionally on $I_1$ and $I_2$, the dataset $Y^{(2)}$ is independent from $Y^{(1)}$ and thus $Y^{(2)}$ is independent from $\hat{p}$. We then deduce from Lemma~\ref{lem:projmixture} that $\hat{p}(Y^{(2)})$ is a Gaussian mixture with $K$ groups of dimension at most $\max\pa{\log(n), K}$ and, with high probability, with a separatition  at least $\Delta/2$ between the groups. Then, we are in position to 
the results in Section 2.2 of \cite{LiuLi2022} that we state here as a Proposition. In fact, the original theorem of Li and Liu is stated for a Gaussian mixture model where the the group of each observation is sampled at random, whereas we are considering here a setting where the partition $G^*$ is fixed in advance. Nevertheless, by closely inspecting the proof, one readily checks that their result extends to our setting.

\begin{prop}\label{prop:LiuLi}\cite{LiuLi2022}
Let $Z_1,\ldots, Z_{n'}\in \R^{p'}$ being sampled from a mixture of Gaussian with an almost balanced partition $G'^*$ (i.e for all $k\in [K]$, $|G'^*_k|\geq \frac{n'}{\gamma 2 K}$) of $[n']$ and centers $\mu'_1,\ldots, \mu'_K\in \R^{p'}$. For all $\eps>0$, there exists positive constant $c_{\gamma}$, $c'_{\gamma,\epsilon}$, and $c''$  such that the following holds. If $\min_{k\neq l}\|\mu'_{k}-\mu'_j\|\geq c'\pa{\log(n)}^{\frac{1}{2}+\eps}$ and $n'\geq (p'K)^{c}$, there exists an algorithm $\hat{G}$ computable in polynomial time such that
$$\P\cro{\hat{G}=G}\geq 1-1/n^{c''}\enspace.$$
\end{prop}

Suppose that $c_{\gamma}$ in the condition $\Delta^{2}\geq c_{\gamma}[\log(n)^{1+\eps}+\sqrt{\frac{pK^2}{n}}]$  of Proposition~\ref{cor:projection+tenseur} is large enough and that $c'_{\gamma,\epsilon}$ such that $n\geq K^{c'_{\gamma,\epsilon}}$ is also large enough, so that Proposition \ref{prop:LiuLi} holds when applied to $\hat{p}\pa{Y^{(2)}}$. Then, we dispose of an partition $\hat{G}^{(2)}$ in $\pa{I_2}$ computable in polynomial time which is equal to the restriction of $G^*$ to $I_2$ with high probability.

For $k\in [K]$, we write $\hat{\mu}_k^{oracle}=\frac{1}{|I_2\cap G^*_k|}\sum_{i\in I_2\cap G^*_k}Y_i$ and $\hat{\mu}_k^{(2)}=\frac{1}{|I_2\cap \hat{G}_k^{(2)}|}\sum_{i\in I_2\cap \hat{G}_k^{(2)}}Y_i$. The next Lemma characterizes a regime on which linear discriminant analysis with the centers $\hat{\mu}_k^{oracle}$ does perfect classification of $I_1$. We refer for example to Section 12.7.1 (page 271) of the textbook \cite{HDS2} for a proof of this lemma. 

\begin{lem}\label{lem:LDA}\cite{HDS2}
For $i\in I_1$, let us define $\hat{k}_i=\argmin_{k\in [K]}\|Y_i-\hat{\mu}_k^{oracle}\|$, by breaking arbitrarily equality. There exist constants $c_{\gamma}$ and  $c'$ such that if $\Delta^2\geq c_{\gamma}\pa{\log(n)+\sqrt{\frac{pK\log(n)}{n}}}$, the following holds with probability at least $1-\frac{c'}{n^2}$. For all $i\in I_1$, $\hat{k}_i=k^*_i$.
\end{lem}

Then, on the high probability event on which Lemma \ref{lem:LDA} holds and on which the clustering procedure from \cite{LiuLi2022} onto the projected dataset $\hat{p}(Y^{(2)})$ recovers exactly the restriction of $G^*$ to $I_2$, Linear Discriminant Analysis with the centers $\mu_k^{(2)}$ also does perfect classification. We are then able to recover the entire partition $G^*$. This concludes the proof of the corollary.

\end{proof}

\subsubsection{Hierarchical Clustering}\label{sec:hierarchical}

In this Section, we apply a single linkage hierarchical clustering procedure on the dataset $\hat{p}\pa{Y^{(2)}}$. We consider the case where $p\geq \frac{n}{K}$.
Let $\hat{G}$ be the projected hierarchical clustering procedure obtained from Algorithm \ref{alg:hierarchical+projection}. For any two disjoint sets $A,B$ we define the single linkage function $l(A,B)=\min_{i\in A, i'\in B}\|Y_i-Y_{i'}\|$.

\begin{algorithm}
\caption{Hierarchical Clustering algorithm with single linkage after splitting and projecting the dataset}\label{alg:hierarchical+projection}
\KwData{$Y_{1},\ldots, Y_{n}$}
Draw $(\delta_i)_{i\in [1,n]}$ independently and uniformly on $\{1,2\}$;\\
Compute $\hat{v_1},\ldots,\hat{v}_K$ the leading eigenvectors of $\pa{Y^{(1)}}^T Y^{(1)}-nI_p$, with $Y^{(1)}$ the restriction of $Y$ to $I_1$;\\
For $i\in I_2$, compute $\hat{p}(Y_i)$ the orthogonal projection of $Y_i$ on the space $Vect\pa{\hat{v_1},\ldots,\hat{v}_K}$ (if $p\leq \max\pa{\log(n),K}$ keep $Y_i$);\\
$t\gets 0$;\\
$G^{(0)} \gets \ac{\{i\}_{i\in I_2}}$;\\
\While{$t<|I_2|-K$}{
    Find $\hat{a},\hat{b}$ minimizing $l\pa{G^{(t)}_{\hat{a}},G^{(t)}_{\hat{b}}}$;\\
    Build $G^{(t+1)}$ by merging the groups $G^{(t)}_{\hat{a}}$ and $G^{(t)}_{\hat{b}}$, the other groups remaining unchanged;\\$t\gets t+1$;\\
    }
$\hat{G}:=\hat{G}^{(t)}$;\\
\For{$k\in[1,K]$}{
Compute $\hat{\mu}_k:=\frac{1}{|\hat{G}_k|}\sum_{i\in \hat{G}_k}X_i$}
\For{$i\in I_1$}{
Assign $i$ to the group $\hat{G}_k$ minimizing $\|X_i-\hat{\mu}_k\|$.}
\KwResult{The partition $\hat{G}$.}
\end{algorithm}

The following proposition provides separation conditions under which Algorithm \ref{alg:hierarchical+projection}  perfectly recover the partition  $G^*$ with high probability.

\begin{prop}\label{cor:hierarchical+projection}
There numerical constants $c$ and $c'$ and a positive constant $c_{\gamma}$ that only depends on $\gamma$ such that the following holds. Suppose $n\geq p\geq \frac{n}{K}$, $n\geq c K^2$, and 
\[
\Delta^2\geq c_{\gamma}\pa{\log(n)+\sqrt{\frac{pK^2\log(n)}{n}}}\enspace .
\]  Denoting $\hat{G}$ the output of Algorithm \ref{alg:hierarchical+projection}, we have $\P[\hat{G}=G^*]\geq 1-c'/n^2$.
\end{prop}

\begin{proof}[Proof of Proposition \ref{cor:hierarchical+projection}]

If $p\leq \max\pa{\log(n), K}$, we have that $\hat{p}$ is the identity and we can therefore assume that $p\geq \max\pa{\log(n), K}$.

Conditionally on $I_1$ and $I_2$, the dataset $Y^{(2)}$ is independent from $Y^{(1)}$ and thus $Y^{(2)}$ is independent from $\hat{p}$. We deduce from this that $\hat{p}(Y^{(2)})$ is a Gaussian mixture of dimension $\max\pa{\log(n), K}$ which is well separated with high probability (Lemma \ref{lem:projmixture} provides a separation at least $\Delta^2/4$). With high probability, using Lemma \ref{lem:balancedclusters}, we also have $|I_2|\geq n/4$ and for all $k\in[K]$ $|I_2\cap G^*_k|\geq \frac{n}{4K\gamma}$. 

Hence, applying a Hierarchical procedure ensures that, if $\Delta^2\geq c'' \pa{\log(n)+\sqrt{K\log(n)}}$ --see Proposition 4 in~\cite{Even24}, then we recover exactly with high probability the restriction of the partition $G^*$ to $I_2$. We write the obtained partition $\hat{G}^{(2)}$. Note that the condition $\Delta^2\geq c''\pa{ \log(n)+\sqrt{K\log(n)}}$ is ensured if the constant $c_{\gamma}$ in the statement of the proposition such that $\Delta^2\geq c_{\gamma}[\log(n)+\sqrt{\frac{pK^2\log(n)}{n}}]$ is large enough. Hence, $I_1$ is perfectly clustered. In turn, we deduce that $I_2$ is perfectly clustered by applying by arguing as in the previous proof.

\end{proof}

The proof of Proposition \ref{cor:hierarchical+projection} also porvides a result when we do not make any assumption on the dimension. We state this result as a proposition.

\begin{prop}\label{prop:hierarchical2}
There numerical constants $c$ and $c'$ and a positive constant $c_{\gamma}$ that only depends on $\gamma$ such that the following holds. Suppose  $n\geq c K^2$, and 
\[
\Delta^2\geq c_{\gamma}\pa{\log(n)+\sqrt{K\log(n)}+\sqrt{\frac{pK^2\log(n)}{n}}}\enspace .
\]  Denoting $\hat{G}$ the output of Algorithm \ref{alg:hierarchical+projection}, we have $\P[\hat{G}=G^*]\geq 1-c'/n^2$.
\end{prop}

\subsubsection{Proof of Lemma \ref{lem:projmixture}}\label{prf:projmixture}

We work conditionally on $I_1$. To ease the notation, we write $Y=Y^{(1)}$ and $n'=|I_1|$. Without loss of generality, we restrict ourselves in the following to the event where $n'\geq n/4$ and for all $k\in [K]$, $|I_1\cap G_k^*|\geq |G_k^*|/4\geq \frac{n}{4\gamma K}$ (Lemma \ref{lem:balancedclusters}). We remark that $\hat{v}_1,\ldots, \hat{v}_K$ are also the $K$ leading eigenvectors of $Y^T Y-n'I_p$. We seek to find an event of high probability on which 
\begin{itemize}
\item the quantity $x^T\pa{Y^T Y-n'I_p}x$ is uniformly large for unit vectors such that $|\<x,\frac{\mu_k-\mu_l}{\|\mu_k-\mu_l\|}\>|$ is large enough, for some $k\neq l$,
\item the $(k+1)$-th eigenvalue $\hat{\lambda}_{k+1}$ of $Y^TY-n'I_p$ is small.
\end{itemize}

Such an event is provided by Lemma \ref{lem:uniform} and \ref{lem:orthogonalspace}, respectively proven in Section \ref{prf:uniform} and \ref{prf:orthogonalspace}.

\medskip

\begin{lem}\label{lem:uniform}
There exists a positive constant $c_{\gamma}$ that only depends on $\gamma$ such that, if $\Delta^4\geq c_{\gamma}\frac{pK^2}{n}$, the following holds with probability at least $1-\frac{2}{n^2}$. Simultaneously on all $x\in \R^p$ such that $\|x\|=1$ and such that there exists $k\neq l$ with $|\<x,\frac{\mu_k-\mu_l}{\|\mu_k-\mu_l\|}\>|\geq \frac{1}{2}$, we have $$x^T\pa{Y^TY-n'I_p}x\geq \frac{n}{256\gamma K}\Delta^2\enspace.$$
\end{lem}

\medskip

\begin{lem}\label{lem:orthogonalspace}
There exists a positive constant $c_{\gamma}$ that only depends on $\gamma$ such that, if $\Delta^4\geq c_{\gamma}\frac{pK^2}{n}$, the following holds with probability at least $1-\frac{2}{n^2}$. Simultaneously on all $x\in\R^p$ such that $\|x\|=1$ and such that $x\in\pa{\mu_1,\ldots,\mu_K}^{\perp}$, we have $$x^T\pa{Y^T Y-n'I_p}x\leq \frac{n}{512\gamma K}\Delta^2\enspace.$$ 
\end{lem}

\medskip

In the following, we suppose $\Delta^4\geq c_1 \frac{pK^2}{n}$, with $c_1$ a numerical constant large enough such that Lemma \ref{lem:uniform} and Lemma \ref{lem:orthogonalspace} both hold. We restrict ourselves to the event of probability at least $1-\frac{4}{n^2}$ defined as the union of the two events of Lemma \ref{lem:uniform} and Lemma \ref{lem:orthogonalspace}.

Lemma \ref{lem:orthogonalspace} implies that the $(k+1)$-th largest eigenvalue of $\pa{Y^T Y-n'I_p}$ satisfies $\hat{\lambda}_{k+1}\leq \frac{n}{512\gamma K}\Delta^2$. Let $k\neq l$ and $y=\frac{\mu_k-\mu_l}{\|\mu_k-\mu_l\|}$. We decompose $y=\hat{p}(y)+(y-\hat{p}(y))$. Since $\<y,y\>=1$, then either $\<y,\hat{p}(y)\>\geq \frac{1}{2}$, either $\<y,(y-\hat{p}(y))\>\geq \frac{1}{2}$. 

Let us suppose that $\<y,(y-\hat{p}(y))\>\geq \frac{1}{2}$ and let us find a contradiction. Using Lemma \ref{lem:uniform}, we deduce that $(y-\hat{p}(y))^T\pa{Y^T Y-n'I_p}(y-\hat{p}(y))\geq \frac{n}{256\gamma K}\Delta^2$. However, $y-\hat{p}(y)$ is the orthogonal projection of $y$ onto the space spread by the eigenvectors corresponding to the eigenvalues $\hat{\lambda}_{k+1},\ldots, \hat{\lambda}_{p}$. Thus, $(y-\hat{p}(y))^T\pa{Y^T Y-n'I_p}(y-\hat{p}(y))\leq \hat{\lambda}_{k+1}\leq \frac{n}{512 \gamma K}\Delta^2$, which leads to a contradiction.

Thus, $\<y,\hat{p}(y)\>\geq \frac{1}{2}$. Hence, $\|\hat{p}(\mu_k)-\hat{p}(\mu_l)\|^2=\|\mu_k-\mu_l\|^2\|\hat{p}(y)\|^2=\|\mu_k-\mu_l\|^2\<y,\hat{p}(y)\>^2\geq \frac{\|\mu_k-\mu_l\|^2}{4}$. This concludes the proof of the lemma.

\begin{proof}[Proof of Lemma \ref{lem:uniform}]\label{prf:uniform}

In the proof of this lemma, we write $A\in \ac{0,1}^{n'\times K}$ for the assignment matrix defined by $A_{ik}=\1_{i\in G^*_k}$, $\mu\in \R^{K\times p}$ for the matrix of the means whose $k$-th row is $\mu_k$, and $E\in \R^{n'\times p}$ the noise matrix $Y-\E[Y]$ which is distributed as i.i.d. standard normal distributions.

Using the decomposition $Y=A\mu+E$, we have $x^T(Y^TY-n'I_p)x=x^T(E^TE-n'I_p)x+2x^T(A\mu)^TEx+x^T(A\mu)^T(A\mu)x$. The three following lemmas, proved in Sections \ref{prf:signal}, \ref{prf:quadratic} and \ref{prf:crossproduct}, control each of these terms. Let us define the set of suibable unit vectors $$\mathcal{X}:=\ac{x\in \R^{p}: \|x\|=1 , \text{ and}\, \exists k\neq l\in [K]\enspace s.t\enspace \<x,\frac{\mu_k-\mu_l}{\|\mu_k-\mu_l\|}\>\geq \frac{1}{2}}\enspace.$$

\begin{lem}\label{lem:signal}
For any $x\in \mathcal{X}$, we have 
\[
x^T(A\mu)^T(A\mu)x\geq \frac{n}{64\gamma K}\Delta^2 \enspace .
\]
\end{lem}

\begin{lem}\label{lem:quadratic}
With probability at least $1-\frac{1}{n^2}$, simultaneously on all unit vectors $x\in \R^p$, we have $$|x^T\pa{EE^T-n'I_p}x|\leq 4\sqrt{n'(6p+4\log(n))}+48p+32\log(n)\enspace.$$
\end{lem}

\begin{lem}\label{lem:crossproduct}
With probability at least $1-\frac{1}{n^2}$, simultaneously on all unit vectors $x\in\R^p$, we have $$|x^T(A\mu)^TEx|\leq \frac{1}{4}\|A\mu x\|^2+4\pa{\sqrt{K}+7\sqrt{p+2\log(n)}}^2\enspace.$$
\end{lem}

Combining Lemmas \ref{lem:signal}, \ref{lem:quadratic}, and \ref{lem:crossproduct}, we deduce that, with probability at least $1-\frac{2}{n^2}$, simultaneously on all $x\in\mathcal{X}$, we have

\begin{align*}
x^T\pa{Y^TY-n'I_p}x\geq& \frac{1}{2}\|A\mu x\|^2-8\pa{\sqrt{K}+7\sqrt{p+2\log(n)}}^2
\\ & -4\sqrt{n'(6p+4\log(n))}-48p-32\log(n)\\
\geq& \frac{n}{128\gamma K}\Delta^2-8\pa{\sqrt{K}+7\sqrt{p+2\log(n)}}^2\\ &-4\sqrt{n'(6p+4\log(n))}-48p-32\log(n)\\
\geq& \frac{n}{128\gamma K}\Delta^2-c'\pa{K+p+\sqrt{pn}+\sqrt{n\log(n)}}\enspace,
\end{align*}
with $c'$ a numerical constant. Let us now restrict ourself to the event of probability $1-\frac{2}{n^2}$ on which, simultaneously on all $x\in \mathcal{X}$, the above inequality is true. We recall the hypothesis $n\geq p\geq \max\pa{K,\log(n)}$. Under this hypothesis, $\pa{K+p+\sqrt{pn}+\sqrt{n\log(n)}}\leq 4 \sqrt{pn}$. Thus, if the constant $c_{\gamma}$ such that $\Delta^4\geq c\frac{pK^2}{n}$ is large enough, we conclude conclude that.

$$x^T\pa{Y^TY-n'I_p}x\geq \frac{n}{256\gamma K}\Delta^2\enspace.$$
It remains to prove Lemmas  \ref{lem:signal}, \ref{lem:quadratic}, and \ref{lem:crossproduct}. 

\end{proof}

\begin{proof}[The signal term: proof of Lemma \ref{lem:signal}]\label{prf:signal}

Let us take a unit vector $x$ such that, for some $k\neq l$, we have  $|\<x,\frac{\mu_k-\mu_l}{\|\mu_k-\mu_l\|}\>|\geq \frac{1}{2}$. We write $y=\frac{\mu_k-\mu_l}{\|\mu_k-\mu_l\|}$ and we compute
\begin{align*}
x^T(A\mu)^T(A\mu)x=&\sum_{k'\in[1,K]}\sum_{a\in G_{k'}^*}\<\mu_{k'},x\>^2
\geq \frac{n}{4K\gamma}\pa{\<x,\mu_k\>^2+\<x,\mu_l\>^2}\enspace.
\end{align*}

Using the hypothesis $|\<x,y\>|\geq \frac{1}{2}$, we deduce that $|\<x, \mu_k-\mu_l\>|\geq \frac{\Delta}{2}$ and therefore  $|\<x,\mu_k\>|\geq \frac{\Delta}{4}$ or $|\<x,\mu_l\>|\geq \frac{\Delta}{4}$.  Hence, $x^T(A\mu)^T(A\mu)x\geq \frac{n}{64\gamma K}\Delta^2$. This concludes the proof of the lemma.

\end{proof}

\begin{proof}[Proof of Lemma \ref{lem:quadratic}]\label{prf:quadratic}

For any $x\in \R^p$ such that $\|x\|=1$, we have
$$|x^T(E^TE-n'I_p)x|\leq \|EE^T-n'I_p\|_{op}\enspace.$$
We use the next lemma for upper-bounding this quantity. We refer for example to the textbook~\cite{HDS2} (Lemma 12.10, page 273).
\begin{lem}\label{lem:normop}
There exists a random variable $\xi$ with exponential distribution of parameter $1$ such that $$\|E^TE-n'I_p\|_{op}\leq 4\sqrt{n'(6p+2\xi)}+48p+16\xi\enspace.$$
\end{lem}

Thus, with probability at least $1-\frac{1}{n^2}$, we have $$\|E^TE-n'I_p\|_{op}\leq 4\sqrt{n'(6p+4\log(n))}+48p+32\log(n)\enspace,$$

which concludes the proof of the lemma.

\end{proof}

\begin{proof}[Proof of Lemma \ref{lem:crossproduct}]\label{prf:crossproduct}

We denote $P$ the orthogonal projection onto the rows of $A\mu$. For any $x\in \R^p$ such that $\|x\|=1$, we have
\begin{align*}
|x^T(A\mu)^TEx|=&|\<A\mu x, Ex\>|=|\<A\mu x, PEx\>|\\
\leq&\|A\mu x\|\|PEx\|\\
\leq& \|A\mu x\| \|PE\|_{op}\\
\leq& \frac{1}{4}\|A\mu x\|^2+4\|PE\|_{op}^2\enspace.
\end{align*}

The next lemma, which is just   proved below,  provides an upper-bound of the quantity $\|PE\|_{op}$.

\begin{lem}\label{lem:operatorprojection}
With probability at least $1-\frac{1}{n^2}$, we have $$\|PE\|_{op}\leq \pa{\sqrt{K}+7\sqrt{p+2\log(n)}}\enspace.$$
\end{lem}

Lemma \ref{lem:operatorprojection} implies that, with probability at least $1-\frac{1}{n^2}$, simultaneously on all $x\in \R^p$ with $\|x\|=1$, we have $$x^T(A\mu)^TEx\leq \frac{1}{4}\|A\mu x\|^2+4\pa{\sqrt{K}+7\sqrt{K+2\log(n)}}^2\enspace.$$

\end{proof}

\begin{proof}[Proof of Lemma \ref{lem:operatorprojection}]\label{prf:operatorprojection}
This lemma is stated as an exercise in~\cite{HDS2} (exercise 12.9.6, page 288). Let $r\leq K$ denote the rank of $A\mu$. We define $u_1,\ldots,u_r$ an orthogonal basis of the space spanned by the rows of $A\mu$. Let $U=\pa{u_1,\ldots, u_r}$. Then $P=UU^T$. For any unit norm vector $x\in \R^p$, we have 
\begin{align*}
x^T(PE)^T(PE)x=&x^TE^T P^2 Ex\\
=&x^TE^TPEx\\
=&x^TE^TUU^TEx\\
=&x^T(U^TE)^T(U^TE)x\enspace.
\end{align*}
Thus, $\|PE\|_{op}=\|U^TE\|_{op}$. Moreover, the columns of $U^T E$ are independent with law $\mathcal{N}\pa{0,I_r}$. Let us now take again some $x\in \R^p$ with $\|x\|=1$. We denote $W=U^TE$ and get 
\begin{align*}
\|Wx\|^2=&x^TW^TWx= x^T(W^TW-KI_n)x+K\\
\leq& K+\|W^TW-KI_n\|_{op}\enspace.
\end{align*}
Thus, we have $\|W\|^2_{op}\leq K+\|W^TW-KI_p\|_{op}$. Applying Lemma \ref{lem:normop}, we deduce the existence of an exponential random variable $\xi'$ such that $|W|_{op}^2\leq \pa{\sqrt{K}+7\sqrt{p+\xi'}}^2$. Hence, with probability $1-\frac{1}{n^4}$, we have 
$$\|W\|_{op}\leq \sqrt{K}+7\sqrt{p+2\log(n)}\enspace,$$
which concludes the proof of the lemma.

\end{proof}

\begin{proof}[Proof of Lemma \ref{lem:orthogonalspace}]\label{prf:orthogonalspace}

Let $x\in \pa{\mu_1,\ldots, \mu_K}^{T}$ be a unit vector. We have $A\mu x=0.$ Thus, we have $x^T\pa{Y^TY-n'I_p}x=x^T\pa{E^T E-n'I_p}x$. In turn, we have $$x^T\pa{Y^TY-n'I_p}x\leq \|EE^{T}-n'I_p\|_{op}\enspace.$$

By Lemma \ref{lem:normop}, we have that, with probability at least $1-\frac{1}{n^2}$, $$\|E^TE-n'I_p\|_{op}\leq 4\sqrt{n'(6p+4\log(n))}+48p+32\log(n)\enspace.$$

Thus, with probability at least $1-\frac{1}{n^2}$, uniformly on all unitary $x$, 

$$x^T\pa{Y^TY-n'I_p}x\leq 4\sqrt{n'(6p+4\log(n))}+48p+32\log(n)\enspace,$$

Recall $n \geq p\geq \log(n)$. If the constant $c_{\gamma}$ such that $\Delta^4\geq c_{\gamma}\frac{pK^2}{n}$ is large enough, we have that, with the same high probability, uniformly on all such  unit
vectors $x$,

$$x^T\pa{Y^TY-n'I_p}x\leq \frac{n}{512\gamma K}\Delta^2\enspace,$$

which concludes the proof of the lemma.
\end{proof}

\subsection{Proof of Proposition \ref{prop:polytimesparse}}\label{prf:polytimesparse}

Without loss of generality, we suppose that $\sigma^2=1$. Let $E'\in \R^{n\times p}$ with i.i.d $\mathcal{N}(0,\frac{1}{2})$ entries. Then $Y^{(1)}=(Y+E')/\sqrt{2}$ and $Y^{(2)}=(Y-E')/\sqrt{2}$ are two independent datasets such that, when $i\in G_k^*$, we both have $Y^{(1)}_i\sim\mathcal{N}\pa{\mu_k/\sqrt{2}, I_p}$ and $Y^{(1)}_i\sim\mathcal{N}\pa{\mu_k/\sqrt{2}, I_p}$. Our strategy follows the two steps;

\begin{enumerate}
\item Use the first dataset $Y^{(1)}$ in order to estimate the set $J^*$ of active columns;
\item Use a clustering procedure to the second dataset, keeping only columns estimated in the first step. 
\end{enumerate}

For the first step, we consider $\hat{J}$ collecting the $s$ columns of $Y^{(1)}$ with the largest euclidean norm. We recall the definition
\begin{equation*}
w_{J^*}:=\min_{j\in J^*}\sum_{i\in [n]}X_{ij}^2\enspace.
\end{equation*}

Next lemma states that, if $w_{J^*}$ is large enough, then $\hat{J}$ contains $J^*$ with high probability.

\begin{lem}\label{lem:recoverycolumns}
There exists a numerical constant $c_1>0$ such that the following holds. If \begin{equation}\label{eq:condition_omega_J*}
w^2_{J^*}:=\min_{j\in J^*}\sum_{i\in [n]}X_{ij}^2\geq c_1\pa{\sqrt{n\log(pn)}+\log(p)}\enspace,
\end{equation} then, with probability higher than  $1-\frac{1}{n^2}$, $\hat{J}$ contains  $J^*$.
\end{lem}

Let us then work conditionally on $\hat{J}$ and let us suppose that $\hat{J}$ indeed contains $J^*$. Let $Y^{(2)}_{\hat{J}}$ the restriction of $Y^{(2)}$ to the columns $j\in \hat{J}$. Since $Y^{(2)}$ is independent from $\hat{J}$, we deduce that $Y^{(2)}_{\hat{J}}$ is a Gaussian Mixture with a separation $\Delta^2/2$ in dimension $s$. We can conclude the proof using Proposition \ref{prop:upper-boundclustering}. We deduce that, except when $s\in [\text{Poly-log}(n), n/K]$ and $n\in [K^2, K^c]$, with $c$ some numerical constant, if $$\Delta^2\overset{\log}{\geq }1+\min\pa{\sqrt{s}, \sqrt{\frac{sK^2}{n}}}\enspace,$$ then it is possible to recover exactly $G^*$ with high probability and with an algorithm computable in polynomial time.

\begin{proof}[Proof of Lemma \ref{lem:recoverycolumns}]

In order to prove that $\hat{J}$ contains $J^*$ with high probability, it is sufficient to prove that, with high probability, for all $j\in J^*$ and for all $j'\in [p]\setminus J^*$, we have $\|Y^{(1)}_{:j}\|^2> \|Y^{(1)}_{:j'}\|^2$.
We decompose the matrix $Y^{(1)}=X/\sqrt{2}+E$ with $E$ a gaussian matrix. For $j\in [p]$, we have $$\|Y^{(1)}_{:j}\|^2=\frac{\|X_{:j}\|^2}{2}+\|E_{:j}\|^2+\sqrt{2}\<X_{:j},E_{:j}\>\enspace.$$
Using Hanson-Wright inequality inequality for Gaussian variables (e.g. Lemma 1~\cite{Laurent00}) and the tail of a gaussian random variable, we deduce that, with  probability higher than  $1-\frac{1}{n^2}$, uniformly on all $j\in [p]$, we have 
\begin{equation}\label{eq:HW_SC}
\left|\|Y^{(1)}_{:j}\|^2-n-\frac{1}{2}\|X_{:j}\|^2\right|\leq c' \pa{\sqrt{n\pa{\log(pn)}}+\|X_{:j}\|\sqrt{\log(pn)}+\log(pn)}\enspace ,
\end{equation}
for some $c'>0$.  On this same event of high probability, we have that, for $j\in [p]$ for which $X_{:j}=0$, 
$$\|Y^{(1)}_{:j}\|^2\leq n+c'\pa{\sqrt{n\pa{\log(pn)}}+\log(pn)}\enspace.$$
In light of~\eqref{eq:HW_SC}, if we take the constant $c_1$ large enough in~\eqref{eq:condition_omega_J*}, we conclude that
$\min_{j\in J^*}\|Y_{:j}\|^2> \max_{j\notin J^*}\|Y_{:j'}\|^2$, which concludes the proof of the lemma.

\end{proof}

\subsection{Proof of Proposition \ref{prop:upperboundsparseIT}}\label{prf:upperboundsparseIT}

We suppose without loss of generality that $\sigma^2=1$. Let $E'\in \R^{n\times p}$ with i.i.d $\mathcal{N}(0,1)$ entries. Then $Y^{(1)}=(Y+E')/\sqrt{2}$ and $Y^{(2)}=(Y-E')/\sqrt{2}$ are two independent datasets such that, when $i\in G_k^*$, we both have $Y^{(1)}_i\sim\mathcal{N}\pa{\mu_k/\sqrt{2}, I_p}$ and $Y^{(1)}_i\sim\mathcal{N}\pa{\mu_k/\sqrt{2}, I_p}$.

For any partition $G$ of $[n]$ into $K$ groups, we denote $B^G$ the associated normalized partnership matrix defined by 
$$B^G_{ij}=\sum_{k\in [K]}\frac{1}{|G_k|}\1\ac{i\in G_k}\1\ac{j\in G_k}\enspace.$$
The application $G\to B^G$ is a bijection from the set of all partitions into $K$ groups to the set of matrices (Lemma 12.3 of \cite{HDS2} page 262) $$\mathcal{B}=\ac{B\in S_n(\R)^+: \enspace B_{ij}\geq 0, Tr(B)=K, B1=1, B^2=B }\enspace.$$
To alleviate the notation, we write $B^*$ for $B^{G^*}$.

For any such normalized partnership matrix  $B\in \mathcal{B}$ with associated partition $G$, we define $\hat{J}(B)$ as the subset of the $s$ indices $j\in [p]$ that maximizes the square $l_2$ norm 
 $$\sum_{k\in [K]}\pa{\sum_{a\in G_k}Y^{(1)}_{aj}}^2\enspace .$$

Then, for $B\in \mathcal{B}$ and $J\subseteq [p]$, we define the criterion on $Y^{(2)}$.
$$Crit(B,J)(Y^{(2)})=\<Y^{(2)}_J (Y^{(2)}_J)^T-|J|I_n, B\>\enspace.$$
Take any two distinct $B,B'\in \mathcal{B}$. We define the partial ordering relation $'\preceq'$ by $B \preceq B'$ if 
\begin{equation}\label{eq:definition_critere_K_means_sparse}
Crit(B,\hat{J}(B)\cup\hat{J}(B'))(Y^{(2)})\leq Crit(B',\hat{J}(B)\cup\hat{J}(B'))(Y^{(2)})\enspace .
\end{equation}
Finally, we define $\hat{B}$ and the associated partition $\hat{G}$ as any maximal $B$ with respect to this ordering.

In fact, we will show that, provided that  $\Delta^2$ and $w_{J^*}$ are large enough, we have, with high probability, $B\prec B^*$ for all $B\in \mathcal{B}$, which in turn implies that $\hat{G}=G^*$. 

Let us shortly discuss the definition of our estimator. Given $B$, $\hat{J}(B)$ selects the 
columns with empirical largest norms, ie those which are most likely to contain the informative columns. Then, $B \preceq B'$ corresponds to the fact that the Kmeans criterion restricted to the columns in $\hat{J}(B)\cup\hat{J}(B')$ is smaller for $B'$ than for $B$ --see e.g. ~\cite{PengWei07} for the connection between Kmeans criterion and $Crit$. Here, we use a simple sample splitting scheme to avoid technicalities in the simultaneous control of the $\hat{J}(B)$'s and of the Kmeans criterion.

The proofs proceeds with two main steps. First, by Lemma~\ref{prop:recoverycolumns}, $\hat{J}(B^*)$ contains $J^*$ with high probability as long as $w_{J^*}$ is large enough.

Then, we work conditionally on the event of Lemma \ref{prop:recoverycolumns}. The property $J^*\subseteq \hat{J}(B^*)$ implies that, for any $B\in \mathcal{B}$, $Y^{(2)}_{\hat{J}(B)\cup\hat{J}(B^*)}$ is a gaussian mixture of dimension at most $2s$ with a separation $\Delta^2/2$. The next lemma, which builds upon previous analyses of the exact Kmeans criterion~\cite{Even24}, states $B^*$ is a global maximum of the restricted Kmeans criterion provided the separation is large enough. 

\begin{lem}\label{lem:ITsparseKmeans}
Assume that $\Delta^2\geq c \gamma^{5/2}\left[\sqrt{\frac{sK}{n}\left[\log(n)\right]}+ \log(n)\right]$ where $c$ is a large enough numerical constant and assume that $J\subseteq \hat{J}(B^*)$. With probability at least $1-\frac{2}{n^2}$, we have $$Crit(B,\hat{J}(B)\cup\hat{J}(B^*))(Y^{(2)})< Crit(B,\hat{J}(B^*)\cup\hat{J}(B^*))(Y^{(2)})\enspace ,$$
simultaneously for all $B\in \mathcal{B}$.
\end{lem}

Combining Lemmas \ref{prop:recoverycolumns} and Lemma \ref{lem:ITsparseKmeans} leads to the desired result.

\subsubsection{Proof of Lemma \ref{prop:recoverycolumns}}

In this proof, we write $Y$ instead of $Y^{(1)}$ for simplicity of notation. 
In order to prove that $J^*\subseteq \hat{J}(B^*)$ with high probability, we prove that, with high probability, uniformly on all $j\in J^*$ and $j'\notin J^*$, $$\sum_{k\in [K]}\pa{\sum_{a\in G^*_k}Y_{aj}}^2\geq \sum_{k\in [K]}\pa{\sum_{a\in G^*_k}Y_{aj'}}^2\enspace.$$

For any $j\in [p]$, we have that 
\begin{align*}
\sum_{k\in [K]}\pa{\sum_{a\in G^*_k}Y_{aj}}^2=&\sum_{k\in [K]}\pa{|G_k^*|(\mu_k)_j+\sum_{a\in G^*_k}E_{aj}}^2\\
=&\sum_{k\in [K]}|G_k^*|^2(\mu_k)_j^2
+\sum_{k\in [K]}\pa{\sum_{a\in G^*_k}E_{aj}}^2
+2\sum_{k\in [K]}(\mu_{k})_j\pa{\sum_{a\in G^*_k}E_{aj}}\enspace.
\end{align*}

\paragraph*{The quadratic noise term}$\sum_{k\in [K]}\pa{\sum_{a\in G^*_k}E_{aj}}^2=E_{:j}^T S E_{:j}$ with $S$ the $n\times n$ matrix defined by $S_{ij}=\sum_k\1\ac{i,j\in G_k^*}$. Thus, using Hanson-Wright Lemma (see Appendix B.6 of \cite{HDS2} for example), we deduce that, with probability at least $1-\frac{1}{n^2}$, uniformly on all $j\in [p]$, we have,
\begin{align*}
\left|\sum_{k\in [K]}\pa{\sum_{a\in G^*_k}E_{aj}}^2-\sum_{k\in [K]}|G_k^*|\right|\leq& c\sqrt{\sum_{k\in [K]}|G_k^*|^2\pa{\log(n)+\log(p)}}+
c \max_{k\in [K]}|G_k^*|\pa{\log(n)+\log(p)}\\
\leq& c' \gamma\pa{\sqrt{\frac{n^2}{K}\pa{\log(n)+\log(p)}}+\frac{n}{K}\pa{\log(n)+\log(p)}}\enspace ,
\end{align*}
for some numerical constants $c$ and $c'$.

\paragraph*{The cross-product term} The random variable $2\sum_{k\in [K]}(\mu_{k})_j\pa{\sum_{a\in G^*_k}E_{aj}}$ is normally distributed with variance $4\sum_{k\in [K]}|G_k^*|(\mu_k)_j^2$. So, with probability at least $1-\frac{1}{n^2}$, for all $j\in [p]$, and for some numerical constant $c$, $$2\left|\sum_{k\in [K]}(\mu_{k})_j\pa{\sum_{a\in G^*_k}E_{aj}}\right|\leq c\sqrt{\sum_{k\in [K]}|G_k^*|(\mu_k)_j^2\pa{\log(n)+\log(p)}}\enspace ,$$
for some constant $c>0$.

Let us restrict ourselves to an event of high probability on which those two deviation hold. If $j\notin J^*$, then, for some $c>0$, we have 
\begin{align*}
\sum_{k\in [K]}\pa{\sum_{a\in G_k}Y_{aj}}^2-\sum_{k\in [K]}|G_k^*|\leq& c\gamma\pa{\sqrt{\frac{n^2}{K}\pa{\log(n)+\log(p)}}+\frac{n}{K}\pa{\log(n)+\log(p)}}\\
\leq&c\gamma\frac{n}{K}\pa{\log(np)+\sqrt{K\log(np)}}\\
\leq& c\gamma^2 \min_k |G_{k}^*|\pa{\log(np)+\sqrt{K\log(np)}}\\
\leq &\frac{1}{8}\min_k |G_k^*|w_{J^*}^2
\end{align*}

provided the constant $c_1$ such that $w_{J^*}^2\geq c_1\gamma^2(\sqrt{K\log(np)}+\log(np))$ is large enough. 

Let us turn to the case where $j\in J^*$. Provided that the numerical constant $c_1$ in the condition $w_{J^*}^2\geq c_1\gamma^2\pa{\sqrt{K\pa{\log(np)}}+\log(np)}$ is large enough, we have
\begin{align*}
\sum_{k\in [K]}\pa{\sum_{a\in G_k}Y_{aj}}^2-\sum_{k\in [K]}|G_k^*|\geq & \min |G_k^*|\sum_{k\in [K]}|G_k^*|(\mu_{k})_j^2-c\sqrt{\sum_{k\in [K]}|G_k^*|(\mu_k)_j^2\log(np)}-\min |G_k^*|\frac{1}{8}w^2_{J^*}\\
\geq& \frac{1}{2}\min|G_k^*|w_{J^*}\enspace.
\end{align*}

This concludes the proof of the lemma.

\subsubsection{Proof of Lemma \ref{lem:ITsparseKmeans}}\label{prf:ITsparseKmeans}

In this section, for the sake of simplicity, we write $Y$ instead of $Y^{(2)}$ and for any $B\in \mathcal{B}$, we write $Y_B$ the restriction of $Y$ to the columns in $\hat{J}(B^*)\cup \hat{J}(B)$. We recall that we work conditionally on $(\hat{J}(B))_{B\in \mathcal{B}}$ and that we suppose $J\subseteq \hat{J}(B^*)$. We denote $s_B=|\hat{J}(B^*)\cup \hat{J}(B)|\leq 2s$. 

For $B\in \mathcal{B}$, we decompose the observations as  $$Y_B=X_B+E_B\enspace,$$ where $(X_B)_{ij}= \mu_{kj}$ if $i\in G_k^*$ and for $j\in \hat{J}(B^*)\cup \hat{J}(B)$, and  
$E_B\in \R^{n\times s_B}$  is the restriction of $Y-\E[Y]$ to the columns $\hat{J}(B^*)\cup \hat{J}(B)$.
Let us decompose the difference of the criterions.
\begin{align*}
\lefteqn{Crit(B^*, \hat{J}(B^*)\cup \hat{J}(B))(Y)-Crit(B, \hat{J}(B^*)\cup \hat{J}(B))(Y)}&\\ & = \<Y_BY_B^T-s_BI_n, B^*-B\>\\
& =\<X_BX_B^T, B^*-B\>+
\<E_BE_B^T-s_BI_n, B^*-B\>
+2\<X_B(E_B)^T, B^*-B\>= S(B) + N(B)+ C(B)\enspace.
\end{align*}
As mentionned earlier in the proof, this corresponds to the difference of a Kmeans criterion~\cite{PengWei07}, which has been thoroughly studied in~\cite{Even24}. 

In the remainder of this proof, we write $\|A\|_1$ for its entry-wise $l_1$ norm. 
Using directly Lemma 4 from \cite{giraud2019partial}, we deduce that the signal term satisfies 
\begin{equation}\label{eq:signal:sparse:K:means}
S(B)=\<X_BX_B^T, B^*-B\>\geq \frac{1}{4}\Delta^2\delta_B\enspace,
\end{equation}
with $\delta_B=\|B^*-B^*B\|_1$. It remains to upper-bound the quadratic noise term $\<E_BE_B^T-s_BI_n, B^*-B\>$ and the crossed term $\<X_B(E_B)^T, B^*-B\>$ with respect to $\delta_B$. The next two lemmas provide an uniform control of these two terms.

\begin{lem}\label{lem:ITsparsequadratic}
There exists a numerical constant $c_1$ such that the following holds. With  probability at least $1-\frac{1}{n^2}$, we have  
$$|N(B)|\leq c_1 [\delta_B\vee 1]\left[\sqrt{\frac{\gamma sK}{n}\left[\log(n)+ \gamma^2\right]}+ \log(n)+ \gamma^2\right]
\enspace ,$$
simultaneously over all $B\in \mathcal{B}$.
\end{lem}

\begin{lem}\label{lem:ITsparsecrossed}
There exists a numerical constant $c_2$ such that the following holds. With probability at least $1-\frac{1}{n^2}$, we have  
$$|C(B)|\leq c_2 \sqrt{S(B)(\delta_{B}\vee1) \left(\log(n) + \gamma^2 \right)}\enspace  , $$
 simultaneously for all $B\in \mathcal{B}$.
\end{lem}
We consider henceforth that we are under the event of probability higher than $1-2/n^2$ where the deviation bounds in Lemma~\ref{lem:ITsparsequadratic} and~\ref{lem:ITsparsecrossed} hold. Then, for any  $B\in \mathcal{B}\setminus \{B^*\}$, we have 
\begin{eqnarray*}
S(B)- |N(B)| - |C(B)|\geq \frac{S(B)}{2} - (c_1+2c_2) [\delta_B\vee 1]\left[\sqrt{\frac{\gamma sK}{n}\left[\log(n)+ \gamma^2\right]}+ \log(n)+ \gamma\right] \\
\geq \delta_B \frac{\Delta^2}{8}- (c_1+2c_2) [\delta_B\vee 1]\left[\sqrt{\frac{\gamma sK}{n}\left[\log(n)+ \gamma^2\right]}+ \log(n)+ \gamma\right] \ . 
\end{eqnarray*}
 Besides, we know from Lemma 9 in~\cite{Even24} that $\delta_{B}\geq m/(m)\geq 1/\gamma$ if $B\neq B^*$. Hence, $[\delta_B\vee 1]\leq \gamma \delta_B$. We deduce that 
 \[
S(B)- |N(B)| - |C(B)|
\geq \delta_B \frac{\Delta^2}{8}- (c_1+2c_2) \gamma \delta_B \left[\sqrt{\frac{\gamma sK}{n}\left[\log(n)+ \gamma\right]}+ \log(n)+ \gamma\right] \ ,  
 \] 
This last quantity is positive provided that the constant $c$ in the condition 
\[
\Delta^2\geq c \gamma^{5/2}\left[\sqrt{\frac{sK}{n}\left[\log(n)\right]}+ \log(n)\right]
\]
is large enough. 

We have proved that $S(B)- |N(B)| - |C(B)|>0$ for all $B\neq B^*$ which, in light of the definition of $S(B)+N(B)+C(B)$, leads to the desired result.

\begin{proof}[Proof of Lemma \ref{lem:ITsparsequadratic}]\label{prf:ITsparsequadratic}

Le us denote $m=\min_{k\in [K]}|G_k^*|\geq \frac{n}{K\gamma}$. For a fixed $B$, $N(B)=\<E_BE_B^T-s_BI_n, B^*-B\>$ is a quadratic form of Gaussian random variables. Thus, we are in position to apply Hanson-Wright inequality for Gaussian variables--see e.g. Lemma 1 in~\cite{Laurent00}. For a fixed $B\in \mathcal{B}$, with probability higher than $1-2e^{-x}$, we have 
\begin{equation}\label{eq:quadratic_noise}
|N(B)|\leq c\pa{\sqrt{s_B x \|B^*-B\|_F^2}+x\|B^*-B\|_{op}}\enspace \ ,
\end{equation}
where $c$ is a numerical constant. The next Lemma constrol $\|B^*-B\|_F$ and $\|B^*-B\|_{op}$
\begin{lem}\label{lem:norm_Hanson_Wright}
For all $B$, we have $\|B-B^*\|_F\leq   6 \sqrt{\frac{\delta_{B}}{m}}$ and 
$\|B-B^*\|_{op}\leq   2$.
\end{lem}

We remark that the quantity $\delta_B$ is upper-bounded by $2n$. We use a peeling-type argument/For $j\in [2n]$, we denote $$\mathcal{B}_{j}:=\ac{B\in \mathcal{B},\enspace \delta_B\in (j-1,j]}\enspace .$$
We shall apply the definition inequality~\eqref{eq:quadratic_noise} together with an union bound over $\mathcal{B}_j$, this for all $j=1,\ldots, 2n$. The following lemma is a direct consequence of Lemma 17 in~\cite{Even24}. 
\begin{lem}\label{lem:cardinalite}
There exists a positive numerical constant $c$ such that, for any $j=1,\ldots, 2n$, we have
    \[
    \log\left[|\mathcal{B}_{j}|\right]\leq c j \left[\log(n)+ \frac{m^+}{m}\right]\ ,
    \]
where $m^+=\max|G^*_k|$. 
\end{lem}
By definition, we have $m^+/m\leq \gamma$.
By Lemma~\ref{lem:cardinalite}, we deduce that 
$\log(\mathcal{B}_j)\leq c j[\log(n)+ \gamma]$ for some constant $c>0$. 

Putting everything together we conclude that, with probability at leat $1- 1/n^2$, we have 
\[
|N(B)|\leq c [\delta_B\vee 1]\left[\sqrt{\frac{\gamma sK}{n}\left[\log(n)+ \gamma\right]}+ \log(n)+ \gamma\right]\ . 
\]

\end{proof}

\begin{proof}[Proof of Lemma~\ref{lem:norm_Hanson_Wright}]
The proof is based on standard linear algebra and follows from the computations in~\cite{giraud2019partial} and~\cite{Even24}.  Since $B$ and $B^*$ are projector, we have $\|B-B^*\|_{op}\leq 2$. Besides, we have
\[
B-B^*= (I-B^*)(B-B^*)(I-B^*) + B^* (B-B^*) +   (B-B^*)B^* + B^*(B-B^*)B^*
\]
Since $B^*$ is a projector, we have $\|B^*(B-B^*)B^*\|_F\leq \|(B-B^*)B^*\|_{F}$. It follows that 
\begin{equation}\label{eq:frobenius_loss_bi1}
\|B-B^*\|_F\leq \|(I-B^*)(B-B^*)(I-B^*)\|_F + 3\|(B-B^*)B^*\|_{F}\ , 
\end{equation}
In the proof of Lemma 13 in~\cite{Even24}, it is shown that 
\[
\|(I-B^*)(B-B^*)(I-B^*)\|_F\leq  \sqrt{\frac{\delta_{B}}{m}}
\]
Besides, it is shown in proof of Lemma 15 in~\cite{Even24}, that 
\[
\|(B-B^*)B^*\|_{F}\leq \sqrt{2\frac{\delta_{B}}{m}}\ . 
\]
The result follows. 
\end{proof}

\begin{proof}[Proof of Lemma \ref{lem:ITsparsecrossed}]\label{prf:ITsparsecrossed}
 Observe that the random variable $C(B)$ is distributed as a Gaussian random variables whose variance is given by 
 \[
4 \|(B^*-B)X_B\|_F^2 = tr [X_B^TX_B- X_B^T B X_B]= 4S(B)
\ ,
 \]
since $B^2=B$ and $B^* X_B = X_B$. 
Then, we apply an union bound over all $\mathcal{B}_{j}$'s and all $\mathcal{B}_{c,j'}$. Together with Lemma~\ref{lem:cardinalite}, this allows us to conclude that, with probability higher than $1- 1/n^2$, we have  
\[
C(B)\leq c \sqrt{S(B)\lceil\delta_{B}\rceil \left(\log(n) + \frac{m^+}{m}\right)}\ ,
\]
simultaneously for all $B\in \mathcal{B}$. Since $m_+/m\leq \gamma$, the result follows.

\end{proof}

\subsection{Proof of  Proposition~\ref{prp:IT_Biclusterling}}\label{prf:IT_Biclustering}

As in the proof of Proposition~\ref{prop:upperboundsparseIT}, we express the exact Kmeans criterion in terms of partnership matrices. 

Given a partition $G$, we define $B_r^G$ s the corresponding partnership matrix, that is $B_r\in\mathbb{R}^{n\times n}$ is such that $(B_{r})_{ij}=0$ if $i$ and $j$ are not in the same group, whereas $(B_r)_{i,j}$ is equal to 1 over the size of group that contains $i$ and $j$ otherwise. For short, we write $B_r^*$ for $B_r^{G^*}$. Also, $\mathcal{B}_r$ for the collection of all possible partnership matrices of size $n$ with $K$ groups.

Simarly, we define partnership matrice $B_c^H\in\mathbb{R}^{p\times p}$ associated to a partition $H$, the collection $\mathcal{B}_c$ of all such partnership matrices, whereas we denote $B^*_c$ for $B_c^{H^*}$. Finally, given the bi-Kmeans estimator $(\hat{G}, \hat{H})$ from~\eqref{ed:bi-Kmeans}, we write $\hat{B}_c$ and $\hat{B}_r$ for $\hat{B}_c^{\hat{G}}$ and $\hat{B}_r^{\hat{H}}$.

 We will often use that any $B_r\in \mathcal{B}_c$ (resp. $B_c\in \mathcal{B}_c$) is a projection matrix  and that its trace is equal to $K$ (resp. $L$).

Partnership matrices are handy representations for analyzing Kmeans criteria~\cite{PengWei07,giraud2019partial}. Indeed, equiped with this notation, the bi-Kmeans estimator~\eqref{ed:bi-Kmeans} can be reformulated as 
\begin{equation}\label{ed:bi-Kmeans:2}
\left(\hat{B}_r,\hat{B}_c\right)= \arg\max_{B_r\in \mathcal{B}_r,B_c\in \mathcal{B}_c} \tr\left[ Y^T B_r Y B_c\right]\ . 
\end{equation}

\begin{proof}[Proof of~\eqref{ed:bi-Kmeans:2}]
Developing the criterion inside~\eqref{ed:bi-Kmeans}, we have that 
\[
(\hat G,\hat H) \in \mathop{\text{argmax}}_{G,H} \underset{l\in [L]}{\sum_{k\in [K]}}\ \underset{j\in H_{l}}{\sum_{i\in G_{k}}} 2Y_{ij}\bar Y_{kl}^{G\times H} - \pa{\bar Y_{kl}^{G\times H}}^2 .
\]
Then, we observe 
\begin{align*}
\underset{j\in H_{l}}{\sum_{i\in G_{k}}} 2Y_{ij}\bar Y_{kl}^{G\times H} - \pa{\bar Y_{kl}^{G\times H}}=  \underset{l\in [L]}{\sum_{k\in [K]}} |G_k||H_l| \pa{\bar Y_{kl}^{G\times H}}^2= \|B_r Y B_c\|_F^2= \tr\left[ Y^T B_r Y B_c\right] \ , 
\end{align*}
since $B_r$ and $B_c$ are orthogonal projectors. 
\end{proof}

In the following proposition, we write $m_r= \min_{a=k,\ldots , K} |G^*_k|$ and $m^+_r= \max_{k=1,\ldots , K} |G^*_k|$, the respective size of the smallest group and of the largest group of the true partition of the rows. We similarly define $m_c$ and $m_c^+$. By assumption, we have
$\max(m_r^+/m_r, m_c^+/m_c)\leq \gamma$.

Throughout the proof of the proposition, we write $X=\E\cro{Y}$ the signal matrix and $E=Y-X$ the noise matrix, whose entries are i.i.d $\mathcal{N}(0,\sigma^2)$. We suppose without loss of generality that $\sigma^2=1$. As for Proposition~\ref{prop:upperboundsparseIT}, this proof heavily builds upon the analysis of the exact Kmeans criterion in~\cite{Even24}. By definition of our estimator, we have
\[
\tr\left[ Y^T B_r Y B_c\right]\geq \tr\left[ Y^T B^*_r Y B^*_c\right]
\]
The latter inequality implies that 
\begin{equation}\label{eq:inequality_fundamental}
S(\hat{B}_r,\hat{B}_c)\leq N(\hat{B}_r,\hat{B}_c) + C(\hat{B}_r,\hat{B}_c)
\end{equation}
where
\begin{eqnarray}\label{eq:condition1:bi_it}
    S(B_r,B_c)&:=& \tr\left[ X^T B^*_r X B^*_c\right] -  \tr\left[ X^T B_r X B_c\right]\ ; \\ \label{eq:condition2:bi_it}
    N(B_r,B_c) & :=& \tr\left[ E^T B_r E B_c\right] -  \tr\left[ E^T B^*_r E B^*_c\right]\ ;\\ \label{eq:condition3:bi_it}
    C(B_r,B_c) & := & 2\tr\left[ X^T B_r E B_c\right] -  2\tr\left[ X^T B^*_r E B^*_c\right]\ .
\end{eqnarray}
Here, $S(B_r,B_c)$ is a deterministic signal term that only depends on $X$, whereas $N(B_r,B_c)$ is a pure noise term that only depends on $E$. For $B_r\in \mathcal{B}_r$, we write $\delta_{B_r}=\|B_r^*-B_r^*B_r\|_1$. Similarly, for $B_c\in \mathcal{B}_c$, we write $\delta_{B_c}=\|B_c^*-B_c^*B_c\|_1$.

\begin{lem}\label{lem:signal:biclustering}
 For any $B_r$ and $B_c$, we have 
 \[
    S(B_r,B_c)\geq \left[ \delta_{B_c}  \frac{\Delta^2_c}{4}\right]\vee \left[ \delta_{B_r}  \frac{\Delta^2_r}{4}\right]\enspace.
 \]   
\end{lem}

\begin{lem}\label{lem:pure:noise:biclustering}
There exists numerical constants $c,c'$ such that, with probability higher than $1- c'/(n\vee p)^2$, we have
\begin{eqnarray*}
N(B_c,B_r) &\leq&  c  \frac{m_r^+}{m_r}\delta_{B_r} \left[\sqrt{\frac{p}{m_rm_c} \left[\log(n\vee p)+  \frac{m_r^+}{m_r}+ \frac{m_c^+}{m_c}\right]} \right]\\ && + c \frac{m_c^+}{m_c}\delta_{B_c} \left[\sqrt{\frac{n}{m_rm_c} \left[\log(n\vee p)+  \frac{m_r^+}{m_r}+ \frac{m_c^+}{m_c}\right]} \right]\\
& &+c \left[\frac{m_r^+}{m_r} \delta_{B_r}+ \frac{m_c^+}{m_c} \delta_{B_c} \right]\left[\log(n\vee p)+ \frac{m_r^+}{m_r}+ \frac{m_c^+}{m_c}\right]\ ,
\end{eqnarray*}
simultaneously over all $B_r$ and $B_c$. 
\end{lem}

\begin{lem}\label{lem:cross:noise:clustering}
There exists numerical constants $c,c'$ such that, with probability higher than $1- c'/(n\vee p)^2$, we have
\[
C(B_c,B_r)\leq c \sqrt{S(B_r,B_r)\left[\delta_{B_r}\frac{m_r^+}{m_r} \left(\log(n\vee p) + \frac{m_r^+}{m_r}\right)+\delta_{B_c}\frac{m_c^+}{m_c}\left(\log(n\vee p) + \frac{m_c^+}{m_c}\right) \right]}\ .
\]
simultaneously over all $B_r$ and $B_c$. 

\end{lem}

We now use that $m^+_r/m_r\leq \gamma$ and $m_c^+/m_c\leq \gamma$. By combining the two previous lemmas, we deduce that, for some numerical constants $c,c'$, with probability at least $1- c'/(n\vee p)^2$, we have

\begin{eqnarray*}
N(B_c,B_r)+C(B_c,B_r) &\leq & \frac{S(B_r,B_r)}{2}+ c \gamma^{5/2}\delta_{B_r} \left[\sqrt{\frac{KL\log(n\vee p)}{n}}+ \log(n\vee p)\right]\\
& & + c \gamma^{5/2}\delta_{B_c} \left[\sqrt{\frac{KL\log(n\vee p)}{p}}+ \log(n\vee p)\right]\ . 
\end{eqnarray*}
Let us specify this inequality to $\hat{B}_c$ and $\hat{B}_r$. Coming back to~\eqref{eq:inequality_fundamental} and using the lower bound of $S(\hat{B}_c,\hat{B}_c)$ from Lemma~\ref{lem:signal:biclustering}, we observe that, necessarily we have $\hat{B}_c=B_c^*$ and $\hat{B}_r=B_r^*$. It remains to prove the lemmas.

\begin{proof}[Proof of Lemma~\ref{lem:signal:biclustering}]
By Linearity, we have 
\[
S(B_r,B_c) = S(B^*_r,B_c) + \tr\left[ X^T (B^*_r-B_r) X B_c\right]= S(B^*_r,B_c) + \tr\left[ (X B_c)(XB_c)^T (B^*_r-B_r) \right]\ , 
\]
since $B_c$ is a projector. Observe that the rows of $XB_c$ are identical on each group of the true partition of the rows. Hence, it follows from Lemma 4 in~\cite{giraud2019partial} that $\tr[(X B_c)(XB_c)^T (B^*_r-B_r)]\geq 0$ for any $B_r\in \mathcal{B}_r$. Hence, we have $S(B_r,B_c) \geq  S(B^*_r,B_c)$. Then, we deduce again from Lemma~4 in~\cite{giraud2019partial}, that 
\[
S(B^*_r,B_c)\geq \frac{\Delta_c^2}{4}\delta_{B_c} \  .
\]
The result of the lemma then follows by reversing the role of $B_c$ and $B_r$.

\end{proof}

\begin{proof}[Proof of Lemma~\ref{lem:pure:noise:biclustering}]

We first decompose $N(B_r,B_l)$ into a sum of three terms $N(B)= N_1(B_c)+ N_2(B_r)+ N_3(B_r,B_c)$ where
\begin{eqnarray*}
N_1(B_c) = \tr\left[E^TB_r^*E(B_c-B_c^*)\right] \\
N_2(B_r) = \tr\left[E^T(B_r-B_r^*)EB_c^*\right] \\
N_3(B_r,B_c) = \tr\left[E^T(B_r-B_r^*)E(B_c-B_c^*)\right] \ .
\end{eqnarray*}
Since $B_r^*$ is a rank $K$ projector, observe that $N_1(B_c)$ corresponds to the pure noise term in the analysis of a Kmeans criterion for a Gaussian mixture model in dimension $K$ with $p$ observations and $L$ groups. Thus, we could apply Lemma 11 in~\cite{Even24} to control  it. Similarly, $N_2(B_r)$ corresponds to the pure noise term in the analysis of a Kmeans criterion for a Gaussian mixture model in dimension $L$ with $n$ observations and $K$ groups. Still, as the term $N_3(B_r,B_c)$ is slightly more involved, we provide a dedicated proof. 

For the simultaneous control of these three quantities, we will apply  Hanson-Wright inequality together with a dedicated pealing argument. For any integer $j\in [1,2n]$, (resp. $j\in [1,2p]$), we define $\mathcal{B}_{r,j}= \{B\in \mathcal{B}_r: \delta_{B_c}\in (j-1,j]\}$ (resp. $\mathcal{B}_{c,j}= \{B_c\in \mathcal{B}_c: \delta_{B_c}\in (j-1,j]\}$). Since $\delta_{B_r}$ is always smaller than $2n$, this give us a partition of $\mathcal{B}_r\setminus \{B^*_r\}$. We shall apply Hanson-Wright inequality together with an union bound to each of these sets.

For this purpose, we need to control the Frobenius and operator norm of $B_r-B_r^*$ and of $B_c-B_c^*$. Adapting Lemma~\ref{lem:norm_Hanson_Wright} to our setting lead us to
\begin{lem}\label{lem:norm_Hanson_Wright2}
For all $B_r$ and $B_c$, we have 
\begin{eqnarray}
\|B_r-B_r^*\|_F&\leq  & 6 \sqrt{\frac{\delta_{B_r}}{m_r}}\\
\|B_c-B_c^*\|_F&\leq  & 6  \sqrt{\frac{\delta_{B_c}}{m_c}}\\ 
\|B_r-B_r^*\|_{op}&\leq  & 2\text{ and }\|B_c-B_c^*\|_{op}\leq  2\enspace .
\end{eqnarray}
\end{lem}

The following lemma is a straightforward adaptation of Lemma~\ref{lem:cardinalite}. 
\begin{lem}\label{lem:cardinalite2}
There exists a positive numerical constant $c$ such that 
    \[
    \log\left[|\mathcal{B}_{r,j}|\right]\leq c j \left[\log(n)+ \frac{m_r^+}{m_r}\right]\ ,
    \]
for any $j=1,\ldots, 2n$. A similar result holds for $\mathcal{B}_{c,j}$. 
\end{lem}

Now, we are in position to apply Hanson-Wright inequality for Gaussian variables (e.g. Lemma 1~\cite{Laurent00}) to all $B_r$ belonging to $\mathcal{B}_{r,j}$, this for all $j=1,\ldots, 2n$. For a fixed $B_r$, the random variable $N_2(B_r)$ is of form $U^T H U$ where $U$ is a standard Gaussian vector of dimension $nK$, and $H$ is a symmetric matrix satisfying $tr[H]=0$, $\|H\|_F\leq 6\sqrt{K\frac{\delta_{B_r}}{m_r}}$ and $\|H\|_{op}\leq 2$. We deduce, that with probability 
higher than $1/(n\vee p)^2$, we have 
\[
N_2(B_r)\leq  c \left[\delta_{B_r}\vee 1  \right] \left[\sqrt{\frac{K}{m_r} [\log(n\vee p)+ \frac{m_r^+}{m_r}]} + \log(n\vee p)+ \frac{m_r^+}{m_r} \right].
\]
for any $B_r\neq B_r^*$. A similar bound holds for $N_1(B_c)$. For $N_3(B_c,B_r)$, we apply Hanson-Wright inequality to all $B_r$ and $B_c$ belonging to $\mathcal{B}_{c,j}$ and $\mathcal{B}_{r,j'}$. The random variable  $N_3(B_c,B_r)$ is of form $U^T H U$ where $U$ is a standard Gaussian vector of dimension $np$, and $H$ is a symmetric matrix defined by; for $(i,j)\in [n]\times [p]$ and $(i',j')\in [n]\times [p]$, we have $H_{(i,j), (i',j')}=\pa{B_r-B_r^*}_{ii'}\pa{B_c-B^*_c}_{jj'}$. This matrix satisfies $tr[H]=0$, $\|H\|_F\leq 36\sqrt{\frac{\delta_{B_r}\delta_{B_c}}{m_cm_r}}$ and $\|H\|_{op}\leq 4$.  We deduce, that with probability 
higher than $1/(n\vee p)^2$, we have 
\begin{eqnarray*}
N_3(B_c,B_r)&\leq&  c  \left[\sqrt{\frac{\delta_{B_r}\delta_{B_c}[\delta_{B_r}+ \delta_{B_c}+1]}{m_rm_c} \left[\log(n\vee p)+ \frac{m_r^+}{m_r}+ \frac{m_c^+}{m_c}\right]} \right]\\ 
&&+ c[\delta_{B_r}+ \delta_{B_c} + 1]\left[\log(n\vee p)+ \frac{m_r^+}{m_r}+ \frac{m_c^+}{m_c}\right]\ ,
\end{eqnarray*}
as long as $B_c\neq B^*_c$ and $B_r\neq B_r^*$. Recall that $\delta_{B_r}\leq 2n$ and $\delta_{B_c}\leq 2p$. Besides, we know from Lemma 9 in~\cite{Even24} that $\delta_{B_r}\geq m_r/(m_r^+)$ if $B_r\neq B_r^*$ and  $\delta_{B_c}\geq m_c/(m_c^+)$ if $B_c\neq B_c^*$.
This leads to 
\begin{eqnarray*}
N_3(B_c,B_r)&\leq&  c  \frac{m_r^+}{m_r}\delta_{B_r} \left[\sqrt{\frac{p}{m_rm_c} \left[\log(n\vee p)+  \frac{m_r^+}{m_r}+ \frac{m_c^+}{m_c}\right]} \right]\\ && + c \frac{m_c^+}{m_c}\delta_{B_c} \left[\sqrt{\frac{n}{m_rm_c} \left[\log(n\vee p)+  \frac{m_r^+}{m_r}+ \frac{m_c^+}{m_c}\right]} \right]\\
& &+c \left[\frac{m_r^+}{m_r} \delta_{B_r}+ \frac{m_c^+}{m_c} \delta_{B_c} \right]\left[\log(n\vee p)+ \frac{m_r^+}{m_r}+ \frac{m_c^+}{m_c}\right]\ .
\end{eqnarray*}

\end{proof}

\begin{proof}[Proof of Lemma~\ref{lem:cross:noise:clustering}]

For a fixed $B$, $C(B)$ is distributed a Gaussian random variable whose variance is given by
\[
4 \|B_rX B_c - B^*_rX B^*_c\|_F^2 = 4\tr\left[X^TX - X^T B_r X B_c\right]= 4S(B_r,B_c)\ ,
\]
since $X= B_r^*XB^*_c$. Then, we apply an union bound over all $\mathcal{B}_{r,j}$'s and all $\mathcal{B}_{c,j'}$. Together with Lemma~\ref{lem:cardinalite2}, this allows us to conclude that, with probability higher than $1- c'/(n\vee p)^2$, we have  
\[
C(B_c,B_r)\leq c \sqrt{S(B_r,B_r)\left[\lceil\delta_{B_r}\rceil \left(\log(n\vee p) + \frac{m_r^+}{m_r}\right)+\lceil \delta_{B_c}\rceil\left(\log(n\vee p) + \frac{m_c^+}{m_c}\right) \right]}\ .
\]

\end{proof}

\section{Proofs for technical discussions}\label{sec:proofs:technicaldiscussion}

\begin{proof}[Proof of Lemma \ref{lem:whyomega2}]

It is a consequence of the following lemma whose proof is given below. 

\begin{lem}\label{lem:whyomega}
There exists a subset $\Bar{J}\subseteq J^*$ and a subset $\mathcal{K'}\subset [K]$ with $|\mathcal{K'}|\geq \frac{9K}{10}$ satisfying;
\begin{enumerate}
\item For all $j\in \Bar{J}$, $\sum_{k\in [K]}|G_k^*|\pa{\mu_k}_j^2\geq \frac{n\Delta^2\sigma^2}{80s\gamma}$;
\item For all $k\neq l\in \mathcal{K'}$, $\|(\mu_k)_{\Bar{J}}-(\mu_l)_{\Bar{J}}\|^2\geq \frac{1}{2}\Delta^2\sigma^2$.
\end{enumerate}
\end{lem}
Assume that there exist $k\in \mathcal{K'}$ and $l\in [K]$ such that $\|(\mu_k)_{\Bar{J}}-(\mu_l)_{\Bar{J}}\|^2< \frac{1}{8}\Delta^2\sigma^2$. By Lemma \ref{lem:whyomega}, we deduce that for the all $k'\in \mathcal{K'}\setminus\{k\}$, we have 
 $\|(\mu_{k'})_{\Bar{J}}-(\mu_l)_{\Bar{J}}\|^2\geq  \frac{1}{8}\Delta^2\sigma^2$. As a consequence, there exist  at most $K/10$ elements $k\in \mathcal{K}'$ such that there exists $l\in [K]$ with $\|(\mu_k)_{\Bar{J}}-(\mu_l)_{\Bar{J}}\|^2< \frac{1}{8}\Delta^2\sigma^2$. Defining $\mathcal{K}$ by removing all these elements from $\mathcal{K}$, we arrive at the desired conclusion.
\end{proof}

\begin{proof}[Proof of Lemma \ref{lem:whyomega}]
We define $\Bar{J}$ the set of all $j$ such that $\sum_{k\in [K]}|G_k^*|\pa{\mu_k}_j^2\geq \frac{n\Delta^2\sigma^2}{80s\gamma}$. Since the signal is supported on $J^*$ which is of size at most $s$, we deduce that $\sum_{j\notin \Bar{J}}\sum_{k\in [K]}|G_k^*|\pa{\mu_k}_j^2\leq \frac{n\Delta^2\sigma^2}{80\gamma}$. Let $\mathcal{K}^-$ the set of all $k\in [K]$ with $\|\pa{\mu_k}_{J^*\setminus \Bar{J}}\|^2_2\geq \frac{\Delta^2\sigma^2}{8}$. We have
$$|\mathcal{K}^-|\frac{n}{K\gamma}\frac{\Delta^2\sigma^2}{8}\leq \frac{n\Delta^2\sigma^2}{80\gamma}\enspace,$$
which, in turn, implies that  $|\mathcal{K^-}|\leq \frac{K}{10}$. 

We set $\mathcal{K'}=[K]\setminus \mathcal{K}^-$ which is of size at least $\frac{9K}{10}$. For $k,l\in \mathcal{K'}$, we have $$\|(\mu_k)_{\Bar{J}}-(\mu_l)_{\Bar{J}}\|^2\geq \pa{\|\mu_k-\mu_l\|-2\frac{\Delta\sigma}{\sqrt{8}}}^2\geq \pa{\frac{\sqrt{2}}{2}\Delta\sigma}^2\geq \frac{1}{2}\Delta^2\sigma^2\enspace.$$
\end{proof}

\begin{proof}[Proof of Lemma \ref{lem:homogeneity}]\label{prf:homogeneity}
Suppose that $X$ satisfies Assumption \ref{ass:homogeneity} for some $\eta\geq 1$. Since $\min_{k\neq l}\frac{\|\mu_k-\mu_l\|^2}{2\sigma^2}\geq \Delta^2$, we deduce that at least all except one of the $\mu_k$'s satisfy $\|\mu_k\|^2\geq \frac{1}{2}\Delta^2\sigma^2$. We deduce that $$\sum_{j\in J^*}\|X_{:j}\|^2\geq (K-1)\min_{k\in [K]}|G_k^*|\frac{1}{2}\Delta^2\sigma^2\geq \frac{n(K-1)}{2K\gamma}\Delta^2\sigma^2\enspace. $$
On the other hand, $$\sum_{j\in J^*}\|X_{:j}\|^2\leq s\max_{j\in J^*}\|X_{:j}\|^2\leq s\eta w_{J^*}^2\sigma^2\enspace.$$
We conclude the proof of the lemma with $$w_{J^*}^2\geq \frac{n(K-1)}{2sK\gamma\eta}\Delta^2\enspace.$$
\end{proof}

\begin{proof}[Proof of Lemma~\ref{lem:MG-random}]

We have
\begin{align*}
n(n-1)\ MMSE_{poly}&= \E\cro{\|M^*\|^2_{F}} - \sup_{\hat M\ poly-time,\ \E\cro{\|\hat M\|^2_{F}}=1} \E\cro{\langle M^*,\hat M\rangle_{F}}^2\\
&=:  \E\cro{\|M^*\|^2_{F}} - corr^2 = {n^2\over K}(1+o(1)) - corr^2.
\end{align*}

In particular $corr^2=o(n^2/K)$. 
Since 
$$\sup_{\hat G\ poly-time} \E\cro{\langle M^*, M^{\hat G}\rangle_{F}} \leq \sqrt{ \E\cro{\|M^{\hat G}\|^2_{F}}}\ corr$$
and $\E\cro{\|M^{\hat G}\|^2_{F}}\geq n^2/K$,
we get
\begin{align*}
\inf_{\hat G\ poly-time} \E\cro{\| M^*-M^{\hat G}\|^2_{F}}&\geq \E\cro{\|M^*\|^2_{F}}+\E\cro{\|M^{\hat G}\|^2_{F}} - 2 \sqrt{ \E\cro{\|M^{\hat G}\|^2_{F}}}\ corr\\
&\geq   \E\cro{\|M^*\|^2_{F}}+\min_{a\geq n/\sqrt{K}}(a^2 - 2a\ corr)\\
& = \E\cro{\|M^*\|^2_{K}} + {n^2\over K}- {2n\over \sqrt{K}} corr = {2 n^2\over K} (1+o(1))\enspace ,
\end{align*}
where we used $corr^2=o(n^2/K)$ and $\E[\|M^*\|^2_{F}]=n^2K^{-1}(1+o(1))$ for the last two equalities.
\end{proof}

\begin{proof}[Proof of Proposition~\ref{prp:Impossibility:reconstruction}]
The proof is obtained by combining Lemma~\ref{lem:MG-random} with the following lemma.
\begin{lem}\label{prop:loss:clustering}
Assume that both $G^*$ and $G$ are $\gamma$-balanced as defined in~\eqref{eq:balanced}. Then, it
follows that 
\[
[1- err(G,G^*)]^2\leq  \gamma^2 - \frac{K\|M^G-M^*\|_F^2}{2n^2}\ .
\]
\end{lem}

\end{proof}

\begin{proof}[Proof of Lemma~\ref{prop:loss:clustering}]
Without loss of generality, we assume that the permutation $\pi$ in the definition of $err(G,G^*)$ is the identity. 
\[
err(G,G^*)= \frac{1}{2n}\sum_{k=1}^K |G_k \Delta G^*_k|= 1 - \frac{1}{n}\sum_{k=1}^K
|G_k\cap G^*_k|\ , 
\]
which implies that $\sum_{k=1}^K|G_k\cap G^*_k| = n[1- err(G,G^*)]$.
Let us define  $N_1(G)= |\{(i,j): i\stackrel{G}{\sim} j\}|$,
$N_1(G^*)= |\{(i,j): i\stackrel{G^*}{\sim} j\}|$,
and $N_{11}=|\{(i,j): i\stackrel{G^*}{\sim} j\text{ and } i\stackrel{G}{\sim} j \}|$.
Expanding the squares, we have, for $\gamma$-balanced partitions $G,G^*$,  that 
\begin{equation}\label{eq:M2N11}
\|M^G-M^*\|_F^2=  2[N_1(G)+N_1(G^*)]- 4N_{11} \leq 2\gamma^2\frac{n^2}{K} -4 N_{1,1} \ .
\end{equation}
Furthermore,
\begin{eqnarray}\nonumber
    N_{11}&\geq& {1\over 2}\sum_{k}|G_k \cap G^*_k|^2 \\
    &\geq & \frac{1}{2K} \left[\sum_{k=1}^K  |G_k \cap G^*_k|\right]^2 = \frac{n^2}{2K}[1- err(G,G^*)]^2 \ ,\label{eq:N11}
\end{eqnarray}
where we used Cauchy-Schwarz inequality in the second line. Plugging \eqref{eq:N11} in \eqref{eq:M2N11}
gives
$${K\over 2 n^2} \|M^G-M^*\|_F^2 \leq \gamma^2-[1- err(G,G^*)]^2\enspace .$$
The proof of Lemma~\ref{prop:loss:clustering} is complete
\end{proof}

\begin{proof}[Proof of Proposition~\ref{prop:sparse:bad:prior}]
Given any subset $I\subset \{3,\ldots, n\}$ and any $j=1,\ldots, p$, define the matrix $\alpha^{(I,j)}$ 
\[
\alpha^{(I,j)}_{i,j'}= \mathbf{1}\{j'=j\}\mathbf{1}\{i \in I\cup \{1,2\}\}
\]
By permutation invariance of the problem $\kappa_{x,\alpha^{(I,j)}}$ does not depend on $j$ and only depends on $|I|$ through its cardinality. Denote $I_0= \{3,\ldots, D\}$. 
By permutation invariance, we know that 
\begin{eqnarray}\label{eq:lower:correlation:precis}
 \left(\widetilde{corr}^{(SW)}_{\leq D}\right)^2\geq p \binom{n-2}{D-2}\kappa^2_{x,\alpha^{(I_0,1)}}\ . 
\end{eqnarray}
To alleviate the notation, we henceforth write $\alpha$ for  $\alpha^{(I_0,1)}$.

From Theorem~\ref{thm:LTC}, we deduce that 
\begin{equation}\label{eq:LTCgeneral:thm:exemple}
\kappa_{x,\alpha}=\lambda^{D}\sum_{\pi\in \mathcal{P}_{2}(\alpha)}C_{x, \beta_1(\pi),\ldots, \beta_l(\pi)}\enspace,
\end{equation}
where $l=D/2$ here. We recall that $\beta_s=\beta_s(\pi)$ satisfies $|\beta_{s}(\pi)|=2$ so that we can write 
$\beta_{s}$ as $\ac{(i_s,1), (i'_s,1)}$. Equipped with this notation, we have
\begin{align}
C_{x,\beta_{1},\ldots,\beta_{l}} &= \cumul\pa{x, z_1\1_{k^*_{i_1}=k^*_{i'_1}},\ldots,z_1 \1_{k^*_{i_l}=k^*_{i'_l}}}, \label{eq:def-C:bis}
\end{align}
In order to compute \eqref{eq:def-C:bis}, we apply the law of total cumulance (Lemma~\ref{lem:totalcumulance}) by conditionning on $z_1$. Let us define $W_0:=x$ and $W_s= z_1\1_{k^*_{i_s}=k^*_{i'_s}}$ for $s\in [l]$. 
Consider any partition $\overline{\pi}\in \mathcal{P}([0;l])$. By Lemma~\ref{lem:totalcumulance}, we have 
\[
C_{x,\beta_{1},\ldots,\beta_{l}} = \sum_{\overline{\pi}\in \mathcal{P}([0;l]) }\cumul\pa{\cumul\pa{(W_i)_{i\in R}|z_1}_{R\in \overline{\pi}}}
\]
Denote $R_0$ the group that contains $W_0=x$. If $|R_0|=1$, then $\cumul\pa{W_0|z_1}=1/K$ and is constant almost surely. As a consequence, we have  $\cumul\pa{\cumul\pa{(W_i)_{i\in R}|z_1}_{R\in \overline{\pi}}}=0$ by Lemma~\ref{lem:independentcumulant2} since a constant is independent from any other random variable. If $|R_0|=2$ and the other random variable $W_s\in R_0$ is of the form $z_1\1_{k^*_{1}=k^*_{2}}$, we have $\cumul\pa{(W_i)_{i\in R_0}|z_1}= z_1K^{-1}(1-1/K)$. For any other choice of $R_0$, we claim that $\cumul\pa{(W_i)_{i\in R}|z_1}=0$. Indeed, conditionally to $z_1$, $\1_{k^*_1=k^*_2}$ is independent from all the other random variables since each $k^*_i$, for $i\in [D]$ occurs at most once in the other random variables. Now consider a group $R\neq R_0$ of $\overline{\pi}$ that do not contain $0$. For the same independence argument, we have $\cumul\pa{(W_i)_{i\in R}|z_1}=0$ if $|R|>1$. We conclude that $C_{x,\beta_{1},\ldots,\beta_{l}}=0$ unless there exists $s\in [l]$ such that $\beta_s=(1,2)$, in which case, we have 
\[
C_{x,\beta_{1},\ldots,\beta_{l}}= \frac{1}{K^{l}}\left(1 -\frac{1}{K}\right)\cumul\pa{z_1,\ldots,z_1}
\]
Coming back to~\eqref{eq:LTCgeneral:thm:exemple} and counting , we conclude that
\[
\kappa_{x,\alpha} = \frac{(D-2)!}{2^{D/2-1}(D/2-1)!}\cdot \frac{1}{K^{D/2}}\lambda^{D}\left(1-\frac{1}{K}\right)\cumul\pa{z_1,\ldots,z_1}
\]
Let us lower bound the cumulant $\cumul\pa{z_1,\ldots,z_1}$ between Bernoulli distribution of parameter $\rho$. By M\"obius formula in Lemma~\ref{lem:mobiusformula}, we have 
\[
\cumul\pa{z_1,\ldots,z_1}\geq  \rho - \rho^2 \sum_{\pi\in \mathcal{P}(l)} (|\pi|-1)!\geq \rho - 6 \rho^2 l! 2^{l}\ , 
\]
where we used the same computation as in the proof of Corollary~\ref{cor:LTC}. Coming back to~\eqref{eq:lower:correlation:precis} and relying on our condition,  we conclude that 
\[
 \left(\widetilde{corr}^{(SW)}_{\leq D}\right)^2\geq c'e^{-cD\log(D) }p n^{D-2} \frac{1}{K^D}\lambda^{2D} \rho^2  \ , 
\]
where $c$ and $c'$ are positive numerical constants.

\end{proof}

\end{document}